%% file: SatoTate.tex
\documentclass{amsart}
\usepackage[bookmarksopen,bookmarksdepth=2]{hyperref}
\usepackage{amssymb,amsmath,amsfonts,amsthm,epsfig,amscd,stmaryrd}
\usepackage{stmaryrd}
\usepackage[all,cmtip,poly]{xy}
\usepackage{color}
\usepackage{enumitem}
\usepackage{url}

\newcommand{\llb}{\llbracket}
\newcommand{\rrb}{\rrbracket}

\def\ppr{\mathrm{par}}
\def\compositum{\cdot}
\def\Sp{\mathrm{Sp}}
\def\SO{\mathrm{SO}}
\def\sss{\mathrm{ss}}

\def\diag{\mathrm{diag}}

\def\Fps{\F_{p^2}}
 \def\GU{\mathrm{GU}} 
 \def\PSU{\mathrm{PSU}}

\renewcommand{\setminus}{\smallsetminus}

\def\Cl{\mathrm{Cl}}
\def\Br{\mathrm{Br}}
\def\CM{\mathrm{CM}}
\def\UA{\mathrm{UA}}

\def\aux{\mathrm{aux}}

\def\psibar{\overline{\psi}}

 \def\avoid{\mathrm{avoid}}
 
 \def\q{\mathfrak{q}}
\def\r{\mathfrak{r}}
 
\newcommand{\red}{\operatorname{red}}
\newcommand{\loc}{\operatorname{loc}}
\newcommand{\ad}{\operatorname{ad}}
\newcommand{\Art}{\operatorname{Art}}

\newcommand{\rank}{\operatorname{rank}}
\newcommand{\To}{\longrightarrow}
\newcommand{\isoto}{\stackrel{\sim}{\To}}

\renewcommand{\mathbb}{\mathbf}

\newcommand{\Iw}{\mathrm{Iw}}

\newcommand{\crBrMod}{\operatorname{BrMod}^{\operatorname{cr}}}
\newcommand{\crFBrMod}{\F\!\operatorname{BrMod}^{\operatorname{cr}}}
\newcommand{\crFpbarBrMod}{\Fpbar\!\operatorname{BrMod}^{\operatorname{cr}}}

\newcommand{\qFBrMod}{\F\!\operatorname{qBrMod}}

\newcommand{\resK}{\operatorname{res}_K}
\newcommand{\resKinfty}{\operatorname{res}_{K_{\infty}}}

\def\RR{\mathcal{R}}
\def\XX{\mathcal{X}}
\def\SSS{\mathcal{S}}
\def\HH{\mathbf{H}}
\def\OL{\mathcal{O}}
\def\crys{\mathrm{crys}}

\def\sbar{\overline{s}}
\def\rbar{\overline{r}}
\def\chibar{\overline{\chi}}

\def\rhobar{\overline{\rho}}
\def\varepsilonbar{\overline{\varepsilon}}

\def\barK{\overline{{K}}}

\def\cXbar{\overline{\cX}}
\def\sigmabar{\overline{\sigma}}   

\def\pp{\mathfrak{p}}
\def\PP{\mathfrak{P}}
\def\QQ{\mathfrak{Q}}
\def\fg{\mathfrak{g}}

\newcommand{\ffrm}{{\mathfrak m}}
\newcommand{\p}{\mathfrak{p}}
\newcommand{\m}{\mathfrak{m}}

\newcommand{\gM}{\mathfrak{M}}

\newcommand{\gS}{\mathfrak{S}}
\newcommand{\Acirc}{A^\circ}

\def\Ghat{\widehat{G}}

\def\wH{\widetilde{H}}
\def\PSL{\mathrm{PSL}}
\def\PGL{\mathrm{PGL}}

\newcommand{\HT}{\mathrm{HT}}
\newcommand{\ur}{\mathrm{ur}}
\newcommand{\Frob}{\mathrm{Frob}}

\newcommand{\univ}{{\operatorname{univ}}}
\newcommand{\bL}{\mathbf{L}}
\newcommand{\bD}{\mathbf{D}}
\newcommand{\A}{\mathbf{A}}
\newcommand{\C}{\CC}
\newcommand{\F}{\FF}
\newcommand{\N}{{\mathbb N}}
\newcommand{\Q}{{\mathbb Q}}
\newcommand{\R}{{\mathbb R}}
\newcommand{\Z}{{\mathbb Z}}
\newcommand{\CC}{{\mathbb C}}
\newcommand{\FF}{{\mathbb F}}
\newcommand{\GG}{{\mathbb G}}
\newcommand{\TT}{{\mathbb T}}

\newcommand{\cE}{{\mathcal E}}
\newcommand{\cL}{{\mathcal L}}
\newcommand{\cW}{{\mathcal W}}
\newcommand{\tr}{\mathrm{tr}}

\newcommand{\cD}{{\mathcal D}}
\newcommand{\cO}{{\mathcal O}}
\newcommand{\cM}{{\mathcal M}}
\newcommand{\cX}{{\mathcal X}}
\newcommand{\cU}{{\mathcal U}}
\newcommand{\cV}{{\mathcal V}}
\newcommand{\cS}{{\mathcal S}}
\newcommand{\cR}{{\mathcal R}}
\newcommand{\cT}{{\mathcal T}}

\newcommand{\Fp}{\F_p}
\newcommand{\Fpbar}{\Fbar_p}

\newcommand{\Qp}{\Q_p}
\newcommand{\Qpbar}{\Qbar_p}

\newcommand{\Zpbar}{\Zbar_p}
\newcommand{\Zp}{\Z_p}
\newcommand{\Fpbartimes}{\Fpbar^{\times}}
\newcommand{\Fbar}{\overline{\F}}
\newcommand{\Qbar}{\overline{\Q}}
\newcommand{\Zbar}{\overline{\Z}}

\DeclareMathOperator{\GL}{GL}
\DeclareMathOperator{\Gal}{Gal}
\DeclareMathOperator{\Ext}{Ext}
\DeclareMathOperator{\Aut}{Aut}

\DeclareMathOperator{\End}{End}
\DeclareMathOperator{\Ind}{Ind}
\DeclareMathOperator{\im}{im}
\DeclareMathOperator{\Hom}{Hom}
\DeclareMathOperator{\Sym}{Sym}
\DeclareMathOperator{\SU}{SU}
\DeclareMathOperator{\SL}{SL}
\DeclareMathOperator{\Spf}{Spf}
\DeclareMathOperator{\Spec}{Spec}
\DeclareMathOperator{\Supp}{Supp}

 \newcommand{\into}{\hookrightarrow}
 \newcommand{\onto}{\twoheadrightarrow}
\newcommand{\Gm}{\GG_m}
\newcommand{\Res}{\operatorname{Res}}

\newtheorem{ithm}{Theorem}

\newtheorem{theorem}[subsubsection]{Theorem}
\newtheorem{thm}[subsubsection]{Theorem}
\newtheorem{lemma}[subsubsection]{Lemma}
\newtheorem{lem}[subsubsection]{Lemma}

\newtheorem{prop}[subsubsection]{Proposition}
\newtheorem{cor}[subsubsection]{Corollary}

\theoremstyle{definition}
\newtheorem{df}[subsubsection]{Definition}
\newtheorem{defn}[subsubsection]{Definition}

\theoremstyle{remark}
\newtheorem{rem}[subsubsection]{Remark}

\def\numequation{\addtocounter{subsubsection}{1}\begin{equation}}
  \def\nummultline{\addtocounter{subsubsection}{1}\begin{multline}}
    \def\numeqnarray{\addtocounter{subsubsection}{1}\begin{eqnarray}}
\def\anumequation{\addtocounter{subsection}{1}\begin{equation}}
\def\anummultline{\addtocounter{subsection}{1}\begin{multline}}

\newif\iffinalrun
\iffinalrun
\else
 \fi

\iffinalrun
  \newcommand{\need}[1]{}
  \newcommand{\mar}[1]{}
\else
  \newcommand{\need}[1]{{\tiny *** #1}}
\newcommand{\mar}[1]{\marginpar{\raggedright\tiny FIXME: #1 }}\fi

 \ifdefined\NewCommandCopy
  \NewCommandCopy\latexbibitem\bibitem
\else
  \let\latexbibitem\bibitem
  \usepackage{xparse}
\fi
\ExplSyntaxOn
 \RenewDocumentCommand{\bibitem}{O{}m}
  {
   \nineteen_bibitem:en { \nineteen_bibitem_check:n { #1 } } { #2 }
  }

 \cs_new_protected:Nn \nineteen_bibitem:nn { \latexbibitem[#1]{#2} }
 \cs_generate_variant:Nn \nineteen_bibitem:nn { e }
 \cs_new:Nn \nineteen_bibitem_check:n
  {
   \str_case:nnF { #1 }
    {
     {Bia92}{Bia1892}
     {Hec20}{Hec1920}
    }
    { #1 }
  }
 \ExplSyntaxOff

\title[Ramanujan and Sato--Tate for Bianchi modular forms]{The Ramanujan and Sato--Tate Conjectures for Bianchi modular forms}

\author[G.~Boxer]{George Boxer}  \email{g.boxer@imperial.ac.uk} \address{Department of
  Mathematics, Imperial College London,
  London SW7 2AZ,~UK}

\author[F.~Calegari]{Frank Calegari}  \email{fcale@math.uchicago.edu} \address{The University of Chicago,
5734 S University Ave,
Chicago, IL 60637, USA}

\author[T.~Gee]{Toby Gee} \email{toby.gee@imperial.ac.uk} \address{Department of
  Mathematics, Imperial College London,
  London SW7 2AZ,~UK}

\author[J.~Newton]{James Newton} \email{newton@maths.ox.ac.uk}
\address{Mathematical Institute, University of Oxford, Woodstock Road, Oxford OX2 6GG, UK}
  
  \author[J.~A.~Thorne]{Jack A. Thorne} \email{thorne@dpmms.cam.ac.uk}
\address{Department of Pure Mathematics and Mathematical Statistics, Wilberforce Road, Cambridge CB3 0WB,~UK}

  \thanks{G.B.\ was supported by a Royal Society University Research Fellowship}
  
 \thanks{F.C. \ was supported in part by NSF Grant DMS-2001097.}

 \thanks{ T.G.\ was supported in part by an ERC Advanced grant. This
   project has received funding from the European Research Council
   (ERC) under the European Union’s Horizon 2020 research and
   innovation programme (grant agreement No. 884596)}

\thanks{J.N.\ was supported by a UKRI Future Leaders Fellowship, grant MR/V021931/1.}

\begin{document}

\begin{abstract}
  We prove the Ramanujan and Sato--Tate conjectures for Bianchi modular
  forms of weight at least~$2$. More generally, we prove these
  conjectures for all regular algebraic cuspidal
  automorphic representations of~$\GL_2(\A_F)$ of parallel weight,
  where~$F$ is any CM field. We deduce these theorems from a new 
  potential automorphy theorem for the symmetric powers of  $2$-dimensional %
  compatible systems of Galois representations of parallel weight. %
  \end{abstract}

 \maketitle

\setcounter{tocdepth}{2}
{\footnotesize
\tableofcontents
}
 
 \section{Introduction}\label{sec: intro}
 
 Let~$f = \sum_{n=1}^{\infty} a_n q^n$ be a cuspidal modular form of
 weight~$k \ge 2$ and level~$\Gamma_1(N) \subset \SL_2(\Z)$ 
which 
is  an eigenform for all the Hecke operators~$T_p$ for~$(p,N)=1$ and normalized so that~$a_1 = 1$.
The Ramanujan conjecture for~$f$ --- proved by Deligne~\cite{Deligne} as a consequence
of the Weil conjectures --- is the claim that
$$|a_p| \le 2 \cdot p^{(k-1)/2}.$$
Suppose that the coefficients of~$f$ are real.
The Sato--Tate conjecture
 (proved in a sequence of papers~\cite{cht,tay, HSBT,blght}) is the theorem that the normalized values~$a_p/2p^{(k-1)/2} \in [-1,1]$
are equidistributed with respect to the Sato--Tate measure~$2/\pi \cdot \sqrt{1-x^2} dx$
unless~$f$ is a so-called CM form,  in which case the corresponding
measure is the average of the atomic measure with support zero and the measure~$1/\pi \cdot 1/\sqrt{1-x^2} dx$ (the proof in this CM case is much easier and follows from~\cite{Hecke}).
(If the coefficients~$a_p$ are not real, some minor modifications are required to
formulate the conjecture properly.)
These conjectures were originally made for the particular  (non-CM) form~$f = \Delta = q \prod_{n=1}^{\infty} (1 - q^n)^{24} = \sum_{n=1}^{\infty} \tau(n) q^n$ of level~$\SL_2(\Z)$
and weight~$k = 12$ studied by Ramanujan;  this particular case turns out to
be no easier than the general case.

Both of these conjectures have an equivalent reformulation in the language  of automorphic representations.
Associated to a cuspidal modular eigenform~$f$ (as above) is an automorphic
representation~$\pi$ for~$\GL(2)/\Q$. The data of~$\pi$ includes  irreducible admissible infinite dimensional
complex representations~$\pi_p$ of~$\GL_2(\Q_p)$ for
all~$p$. For~$(p,N)=1$, the representations~$\pi_p$ satisfy the
additional property of being so-called \emph{spherical}, and are in particular classified
by a pair of complex numbers~$\{\alpha_p,\beta_p\}$ known as Satake parameters, which
are related to the original coefficients~$a_p$ via the equation
$$x^2 - a_p x + p^{k-1} \chi(p) = (x - \alpha_p)(x - \beta_p),$$
where~$\chi: (\Z/N \Z)^{\times} \rightarrow \C^{\times}$ is the Nebentypus character
of~$f$. 
The Ramanujan conjecture is equivalent to the
equality~$|\alpha_p|=|\beta_p| = p^{(k-1)/2}$, which can be reformulated as saying that
the representation~$\pi_p$ is tempered. The Sato--Tate conjecture is equivalent
(for non-CM forms) to the claim that the conjugacy classes of the matrices
$$\frac{1}{p^{(k-1)/2}} \cdot \left( \begin{matrix} \alpha_p & 0
\\ 0 & \beta_p \end{matrix}  \right)$$
 are equidistributed in $\SU(2)/\text{conjugacy}$ with respect
to the probability Haar measure.

One advantage of these reformulations is that they can be  generalized; the original Ramanujan conjecture becomes the statement that if~$\pi$ is a regular algebraic cuspidal automorphic representation for~$\GL(2)/\Q$, then~$\pi_p$ is tempered for all~$p$. 
The general Ramanujan conjecture is  the statement that if~$\pi$ is a cuspidal automorphic representation for~$\GL(n)/F$ for any number field~$F$, then~$\pi_v$ is tempered for all primes~$v$ of $F$. (One
can generalize further to groups beyond~$\GL(n)$ but then the formulation becomes more subtle.)
This conjecture is still open in the case of~$\GL(2)/\Q$; after one
drops the adjectives ``regular algebraic'' (or even just ``regular''), one then allows Maass forms, which seem beyond the reach of all current
techniques.
On the other hand, one can consider regular algebraic automorphic representations~$\pi$
for~$\GL(2)/F$ for number fields~$F$.
If~$F$ is a totally real field, then these correspond to Hilbert
modular forms of weight~$(k_i)_{i=1}^{d}$ (with~$d=[F:\Q]$) with all weights~$k_i$ at least~$2$, and parity independent of $i$; the
theory here is close  to the original setting of classical modular forms.
One point of similarity is that Hilbert modular forms can also be written as~$q$-series
(now in more variables). Moreover, just as for classical modular forms, there is a direct link
between Hilbert modular forms and the \'{e}tale cohomology of certain algebraic (Shimura)
varieties, which allows one to deduce the Ramanujan conjecture in these cases as a consequence
of the Weil conjectures~(\cite{Brylinski-Labesse,MR2327298}). The Sato--Tate conjecture can also be proved
in these cases by arguments generalizing those used for modular forms
\cite{blgg}.

In this paper, we consider the Ramanujan and Sato--Tate conjectures
for regular algebraic cuspidal automorphic
representations for~$\GL(2)/F$ where~$F$ is now an imaginary quadratic field (or more generally 
an imaginary CM field). In this case, the classical interpretation of these objects 
(sometimes called Bianchi modular forms when~$F$ is imaginary quadratic) 
 looks quite
different from the familiar~$q$-expansions associated to classical or Hilbert modular forms; for example,
if~$F$ is an imaginary quadratic field, they can be thought of as vector valued differential one-forms
on arithmetic hyperbolic three manifolds. The Eichler--Shimura map allows one to relate classical
modular forms of weight~$k \ge 2$ to the cohomology of local systems for congruence subgroups
of~$\SL_2(\Z)$; the analogous theorem also allows one to relate Bianchi modular forms
to the cohomology of local systems for subgroups of~$\SL_2(\OL_F)$.
However, what is missing in this setting is that there is now 
no longer any direct link  to the cohomology
of algebraic varieties. Despite this, in this paper, we
 prove the Ramanujan conjecture for regular algebraic cuspidal automorphic
representations in full for all quadratic fields and with the parallel weight condition for arbitrary imaginary CM fields. 

For a precise clarification of what parallel weight~$k$ means, see Definition~\ref{parallelweight}. The meaning of `regular algebraic' is also recalled immediately before this definition. When $F$ is imaginary quadratic, all regular algebraic cuspidal automorphic representations for $\GL(2)/F$ have parallel weight.
\begin{ithm}[Ramanujan Conjecture, Theorem~\ref{thm: Ramanujan thm main
    paper}] \label{Ramanujan}
Let~$F/\Q$ be an imaginary CM field.
Let~$\pi$ be a cuspidal algebraic automorphic representation for~$\GL(2)/F$ 
of parallel weight~$k\ge 2$.
Then~$\pi_v$ is tempered for all finite places~$v$; in particular, for places~$v$
prime to the level of~$\pi$, the Satake parameters~$\{\alpha_v,\beta_v\}$ of~$\pi_v$ satisfy~$|\alpha_v| = |\beta_v| = N(v)^{(k-1)/2}$.
\end{ithm}

\begin{ithm}[Sato--Tate Conjecture, Theorem~\ref{thm_Sato_Tate_general_case}]
\label{SatoTate}
Let~$F/\Q$ be an imaginary CM field.
Let~$\pi$ be a cuspidal algebraic automorphic representation for~$\GL(2)/F$ 
of parallel weight~$k\ge 2$. 
Assume that~$\pi$ does not have CM,
equivalently, $\pi$ is not the automorphic induction of an algebraic Hecke character
from a quadratic CM extension~$F'/F$.
 For each finite place $v$
prime to the level of~$\pi$, let~$a_v=  (\alpha_v + \beta_v)/(2 N(v)^{(k-1)/2})$ denote the normalized
parameter, and suppose that the~$a_v$ are real. Then the~$a_v$ are uniformly distributed with
respect to the Sato--Tate measure~$2/\pi \cdot \sqrt{1-x^2} dx$.
\end{ithm}

As in the case~$F=\Q$, a minor modification of the statement
 is needed when the~$a_v$ are not real;
we relegate the details of this to Section~\ref{subsec:satotate}.
We also
discuss some alternate formulations of Theorem~\ref{Ramanujan} when~$F/\Q$
is an imaginary quadratic field in Section~\ref{Bianchi}.

To prove Theorems \ref{Ramanujan} and \ref{SatoTate}, we prove 
the potential automorphy of the symmetric powers of the compatible systems of Galois representations associated to a cuspidal, regular algebraic automorphic representation $\pi$ of $\GL_2(\A_F)$. Here again is a simplified version of our main result in this direction.
(To orient the reader, the integer~$k \ge 2$ parametrizing the weight in this
discussion above is related the integer~$m \ge 1$ below via the relation~$m = k - 1$.
This mirrors the fact that the Hodge--Tate weights of $p$-adic Galois representations
associated to modular forms of weight~$k$ are equal to~$\{0,k-1\}$.)

\begin{ithm}[Potential automorphy of symmetric powers, Theorem~\ref{PO}]\label{ithm_potential_automorphy_of_symmetric_powers} Let~$F$ be a CM field, let~$M$ be a number field, and let~$m \ge 1$ be an integer.
Suppose we have a system of Galois representations
$$\rho_{\lambda}: G_F \rightarrow \GL_2(\overline{M}_{\lambda})$$
indexed by primes~$\lambda$ of~$M$
with the following compatibilities:
\begin{enumerate}
\item $\rho_{\lambda}$ is unramified outside a finite set of primes~$\{v \in S\} \cup \{v | N(\lambda)\}$
where~$S$ is independent of~$\lambda$. For any~$v$ not in this set, 
the characteristic polynomial
$P_v(X) = X^2 + a_v X + b_v$ of~$\rho_\lambda(\Frob_v)$  lies in~$M[X]$
and is independent of~$v$.
\item For all but finitely many~$\lambda$, the representations~$\rho_{\lambda} |_{G_v}$
for primes~$v|N(\lambda)$ and~$v \notin S$ are crystalline with Hodge--Tate weights~$H = \{0,m\}$
for every embedding of~$F$ into~$\Qbar_p$.
\end{enumerate}
Assume that at least one~$\rho_{\lambda}$ is irreducible. Then:
\begin{enumerate}
\item {\bf Purity:\rm} for any embedding of~$M \hookrightarrow \mathbf{C}$, the roots~$\alpha_v$ and~$\beta_v$ of~$X^2 + a_v X + b_v$
have absolute value~$q^{m/2}$ where~$q = N(v)$.
\item {\bf Potential automorphy:\rm} There is a number field~$F'/F$ 
such that the restrictions~$\rho_{\lambda} |_{G_{F'}}$ are all automorphic
and associated to a fixed cuspidal algebraic~$\pi$ for~$\GL(2)/F'$.
\item {\bf Potential automorphy of symmetric powers:\rm} Fix~$n-1 \ge 2$. Either:
\begin{enumerate} 
\item The~$\rho_{\lambda}$ are all induced from a compatible system
associated to an algebraic Hecke character~$\chi$
of some quadratic extension~$F'/F$. Then~$\Sym^{n-1} \rho_{\lambda}$ is reducible
and decomposes into representations
of dimension two and one which are all automorphic over~$F$.
\item There is a number field~$F'/F$ such that the 
representations~$\Sym^{n-1} \rho_{\lambda} |_{G_{F'}}$ %
are all irreducible and automorphic, associated to
a fixed cuspidal algebraic~$\Pi$ for~$\GL(n)/F'$.
\end{enumerate}
\end{enumerate}
\end{ithm}
The Galois representations associated to cuspidal, regular algebraic automorphic representations of $\GL_2(\A_F)$ are not yet known to satisfy the conditions of Theorem \ref{ithm_potential_automorphy_of_symmetric_powers}, but rather a weaker condition (they form a `very weakly compatible system'). We establish potential automorphy of symmetric powers also under this weaker condition. Once again, we refer to the statement of Theorem \ref{PO} in the main body of the paper for the precise statement that is used to deduce Theorems  \ref{thm: Ramanujan thm main paper} and \ref{thm_Sato_Tate_general_case} (and therefore Theorems \ref{Ramanujan} and \ref{SatoTate} above).

\subsection{The new ideas in this paper}
When $m = 1$, Theorems \ref{Ramanujan}, \ref{SatoTate} and \ref{ithm_potential_automorphy_of_symmetric_powers} were proved in \cite{10author} (see \cite[Thm~1.01, Thm~1.0.2, Thm~7.1.14]{10author}). The deduction of Theorems \ref{Ramanujan} and \ref{SatoTate}  from Theorem \ref{ithm_potential_automorphy_of_symmetric_powers} exactly parallels the arguments in \cite{10author}, so we now focus on explaining the proof of Theorem \ref{ithm_potential_automorphy_of_symmetric_powers}.

Unsurprisingly, our arguments build on those of~\cite{10author}: in particular, we prove the potential automorphy of the compatible system of symmetric powers $\Sym^{n-1}\RR = \{\Sym^{n-1} \rho_{\lambda}\}$ by checking the residual automorphy over some extension $F' / F$ and then applying an automorphy lifting theorem. 
We would like to highlight three new ingredients which appear here:
\begin{enumerate}
	\item A result on generic reducedness of special fibres of
          weight $0$ (local) crystalline deformation rings (see \S\ref{subsec:
		Ihara avoidance discussion} for a further introductory discussion). Using the local-global compatibility result of \cite{caraiani-newton}, this leads to a new automorphy lifting theorem in the setting of arbitrary ramification (Theorem \ref{thm:main_automorphy_lifting_theorem}).
	
	\item An application of a theorem of Drinfeld and Kedlaya \cite{DrinfeldKedlaya}, showing generic ordinarity of families of Dwork motives (Proposition \ref{ordinarypoints}). This makes it possible to verify the potential residual automorphy of certain residual representations by an automorphic motive which is crystalline ordinary at some set of $p$-adic places. 
	
	\item A ``$p$-$q$-$r$'' switch including a version of the ``Harris tensor product trick'' which incorporates an additional congruence between two tensor products of compatible families. One is a tensor product of $\Sym^{n-1}\RR$ with an induction of a character, as usual. The other is a tensor product of $\Sym^{n-1}\RR$ with a different auxiliary compatible family, which gives us more flexibility to realise different local properties at places related by complex conjugation. We discuss in the remainder of the introduction the need for this argument, and give a more detailed sketch in \S \ref{subsec_proof_of_pot_aut} below. 
\end{enumerate}
To explain in more detail the need for these innovations, suppose given a compatible system $\RR = \{\rho_{\lambda}\}$ as in the statement of Theorem \ref{ithm_potential_automorphy_of_symmetric_powers}, therefore of Hodge--Tate weights $\{0, m \}$ for some $m \geq 1$ (and with $m \geq 2$ if we hope to go beyond the cases treated in \cite{10author}). The general strategy
 for proving potential automorphy (the so-called ``$p$-$q$ switch'') is as follows:
 \begin{enumerate}
 \item After making some CM base extension~$H/F$ (depending on~$n$),
   find an auxiliary $n$-dimensional compatible system~$\SSS =\{ s_{\lambda}\}$
 such that:
 \begin{enumerate}
 \item For one prime~$\lambda$, the residual 
 representations~$\Sym^{n-1} \rhobar_{\lambda} |_{G_H}$ and~$\sbar_{\lambda}$ coincide, and moreover satisfy a number of standard
 ``Taylor--Wiles'' conditions.
 \item For a second prime~$\lambda'$, the residual representation~$\sbar_{\lambda'}$
 is induced from a character and is thus
  residually automorphic.
  \item The Hodge--Tate weights of the compatible system~$\SSS$ coincide with those of~$\Sym^{n-1}\RR |_{G_H}$.
  \end{enumerate}
\item Apply an automorphy lifting theorem at~$\lambda'$ to deduce that the compatible system~$\SSS$ is automorphic.
Then
deduce that the residual representation~$\sbar_{\lambda}$ is automorphic,
and  use automorphy lifting theorems again to deduce that~$\Sym^{n-1}\RR |_{G_H}$
is automorphic.
\end{enumerate}
In our setting, both of these steps  cause problems, but those affecting the second
step are more serious. 

The issue in the first step is the requirement (1)(c) on the Hodge--Tate weights. The most natural source of compatible systems~$\SSS$
 are those arising from motives, and a geometrically varying family  of motives 
 cannot have Hodge--Tate weights $0,m,\ldots,m(n-1)$ with~$m \ge 2$ by Griffiths transversality.
(This difficulty is already present if~$F=\Q$ and one wants to prove
the Sato--Tate conjecture for a classical modular form of weight
greater than~$2$, such as~$\Delta$.)
The now-standard resolution to this problem is to employ the ``Harris
tensor product trick'' \cite{HarrisDouble}, and replace~$\Sym^{n-1} \RR$
 by~$\Sym^{n-1} \RR \otimes \Ind^{G_{F}}_{G_{L}} \XX$ for some cyclic CM extension~$L/F$,
  where~$\XX$ is a compatible system
 of algebraic Hecke characters chosen sufficiently carefully so that this new compatible system has
 consecutive Hodge--Tate weights. Now in the second step, one wants to prove this new compatible
 system is potentially automorphic (using for~$\SSS$ a compatible system coming
 from the cohomology of the Dwork family). The potential automorphy of $\Sym^{n-1}\RR$ can then be deduced using cyclic base change~\cite{MR1007299}.
 
 For the second step, applying an automorphy lifting theorem typically requires that the
 compatible systems~$\SSS$ and~$\Sym^{n-1} \RR \otimes \Ind^{G_{F}}_{G_{L}} \XX$ have ``the same'' behaviour at places~$v|p$. There are two problems with this. Firstly, we will need an automorphy lifting theorem that applies to arbitrarily ramified $F$, including non-ordinary representations. Secondly, the compatible system of characters $\XX$ has a restricted form, and in particular its local behaviour can't be chosen arbitrarily at a pair of conjugate places. We explain more about these difficulties and their resolution below.
 
 For context, we first recall the situation when $F=\Q$. For example, one might try to demand that the $p$-adic representations in the compatible systems $\SSS$ and $\Sym^{n-1} \RR \otimes \Ind^{G_{F}}_{G_{L}} \XX$ are both ordinary, and indeed it is straightforward (at least after a ramified base change) to find ordinary
 representations in the Dwork family, and presumably difficult to
 understand the non-ordinary representations in any %
 generality. This means that one would like to show that many of the
 representations in the compatible system~$\RR$ are ordinary.
 
 For a weight~$2$ modular form, it is relatively easy to prove that
 there are infinitely many primes~$p$ for which the~$p$-adic Galois
 representation is ordinary at~$p$.  However, the existence of
 infinitely many ordinary primes for~$\Delta$ (or for any non-CM form
 of weight~$k \ge 4$) remains an open question, so one also has to
 consider the possibility that the residual
 representation~$\rbar_{\lambda}|_{G_{F_v}}$ is locally of the
 form~$\omega^m_2 \oplus \omega^{mp}_2$ on inertia at~$p$.  This
 problem was resolved for classical modular forms in~\cite{blght}, via
 a further study of the Dwork family; in particular, showing that
 certain residual representations of the
 shape~$\Sym^{n-1} (\omega_2 \oplus \omega^p_2)$ arise (locally on
 inertia) as residual representations in that family.

 We now consider the case of an imaginary CM field~$F$, and explain why we need an automorphy lifting theorem allowing ramification at non-ordinary places. Given the
 automorphy lifting theorems for CM fields proved in~\cite{10author},
 the most serious difficulty in adapting the strategy of~\cite{blght}
 is that there is no way to avoid the possibility that a
 representation~$\rho_{\lambda}$ can be simultaneously ordinary at one
 prime~$v|p$ and non-ordinary at the complex conjugate
 place~$v^c$. (One might hope to avoid this by considering places
 with~$v=v^c$, but then we would have to show that certain residual representations of $G_{\Q_{p^2}}$ occur in the Dwork family which seemed to us to be a difficult task.) This
 is a problem because the automorphy lifting theorems in \cite{10author} for
 non-ordinary representations require~$F$ to be unramified at our
 non-ordinary prime~$v^c$; while at the ordinary prime~$v$, we need to
 be able to make a highly ramified base change of (imaginary) CM fields~$F'/F$ to
 find an appropriate representation in the Dwork family.  It is
 however impossible to arrange that such an extension of CM fields is unramified
 at~$v^c$ and ramified at~$v$. One of the key innovations in this
 paper is to prove an automorphy lifting theorem that allows us to
 make a ramified base change at~$v^c$ (Theorem \ref{thm:main_automorphy_lifting_theorem}). This was done in the two-dimensional case in \cite{caraiani-newton}; we discuss the difficulties in
 extending this result to higher dimensions and how we overcome them in Section~\ref{subsec: Ihara avoidance discussion} below. Note that, even when making a ramified base change, it is still important for us to keep track of the inertial type of residual representations in the Dwork family in order to show that the representations of interest are connected in the deformation space.
  
 There turns out to be one final wrinkle, where the second problem mentioned above arises. The $p$-adic representations in our compatible system $\cS$ will satisfy one of two, mutually exclusive, local conditions at each $p$-adic place: they are crystalline at~$p$ and are either ordinary
 or are (on the same component of a local crystalline deformation ring as) a symmetric power of an induction of a Lubin--Tate character of $G_{\Q_{p^2}}$. It turns
 out that we can't always arrange for a tensor product $\Sym^{n-1} \RR \otimes \Ind^{G_F}_{G_L} \XX$ to have local $p$-adic representations of this shape. The problem is
 that algebraic Hecke characters have a very restricted form, and 
 the fact that~$F$ is an imaginary CM field  %
 implies that a suitable choice of~$\XX$ will exist only if~$\rho_{\lambda}$
 is either both ordinary or non-ordinary at each pair of
 places~$\{v,v^c\}$ permuted by complex conjugation
 in~$\Gal(F/F^{+})$.

 Our solution is to instead consider tensor
 products of the form~$(\Sym^{n-1} \cR)\otimes \cR_{\aux}$, where~$\cR_{\aux}$ is a compatible system coming from (part of) the cohomology of the Dwork hypersurface~\cite{QianPotential}. We will be able to choose $\cR_{\aux}$ so that one of the local conditions mentioned in the previous paragraph will be satisfied by $\cS_{\aux} = (\Sym^{n-1} \cR) \otimes  \cR_{\aux}$ at each $p$-adic place.
 It is now no longer possible to  directly deduce  the potential automorphy  of~$\Sym^{n-1} \RR$
 from the potential automorphy of this product.
This is not necessary to prove the Ramanujan conjecture  ---
 already the  automorphy of this tensor product combined with the Jacquet--Shalika bounds (and the fact
 that~$\RR_{\aux}$ is pure) is enough to deduce  purity --- but it is
 necessary to prove the Sato--Tate conjecture.
 However, once the potential automorphy of~$\cS_{\aux}$
 is established, we can (having chosen~$\cR_{\aux}$ carefully
 to begin with) find a third compatible system~$\cR_{\CM}$ such 
 that~$\cS_{\aux} = (\Sym^{n-1} \cR) \otimes \cR_{\aux}$
 and~$\cS_{\CM} = (\Sym^{n-1} \cR) \otimes \cR_{\CM}$ are residually
 the same at a \emph{third} prime~$r$, and~$\cR_{\CM}$ is induced from a character.
 Even though we do not
 have any control over the~$r$-adic representation associated to~$\Sym^{n-1} \cR$ locally at~$v|r$,
the fact that
 it occurs as the same tensor factor in the
 $r$-adic representations of both~$\cS_{\aux}$ and~$\cS_{\CM}$
means we can still put ourselves
 in a situation where both $r$-adic Galois representations lie on the same component
 of a local deformation ring at~$v|r$. From this $p$-$q$-$r$ switch, we can show
 that~$\cS_{\CM} = (\Sym^{n-1} \cR) \otimes \cR_{\CM}$ is potentially automorphic,
 from which we deduce that~$\Sym^{n-1} \cR$ is potentially automorphic. %

 One might also ask whether for general CM fields~$F$ one can drop the
 hypothesis that~$\cR$ has parallel weight. The difficulty in doing so is as follows: in order to pass from~$\Sym^{n-1} \RR$ to a compatible system with consecutive Hodge--Tate weights, one needs to tensor this compatible system
 with a second compatible system with certain prescribed local properties. If~$\cR$ does not have parallel weight, this auxiliary compatible system
 cannot have consecutive Hodge--Tate weights and for reasons explained above also cannot be induced from a compatible system of characters.
 It is very hard to construct such compatible systems because  of the constraints on families
 of geometric local systems imposed by Griffiths transversality.
 The existence of even a \emph{single} regular algebraic cuspidal automorphic representation for~$\GL_2(\A_F)$
 for some CM field~$F$
 which is neither of parallel weight~$2$, nor of CM type, nor arising from base change from the totally real subfield~$F^{+}$ was only found (by a computation) 
 in~\cite[Lemma~8.11(2)]{calmazur} (see also~\cite{MR3091734}).

\subsection{Ihara avoidance and the Emerton--Gee stack}\label{subsec:
  Ihara avoidance discussion}%
There are two main difficulties in proving automorphy lifting liftings
for~$p$-adic representations with~$p$ ramified in~$F$. One is having
local--global compatibility theorems at the places dividing~$p$; this
was resolved in the recent work of
Caraiani--Newton~\cite{caraiani-newton}. The other difficulty was
alluded to above: the usual Taylor--Wiles method for automorphy
lifting only allows us to deduce the automorphy of a $p$-adic
representation~$r$ from the automorphy of a congruent
representation~$r'$ if we  know that for all finite places~$v$,
the representations~ $r|_{G_{F_v}}$ and~$r'|_{G_{F_v}}$ are
``connected'', in the sense that they lie on the same component of
the appropriate local deformation ring.

As we have sketched above, in the particular cases that we consider in
this paper, we have arranged this property at the places~$v|p$ by
considering the ordinary and non-ordinary cases separately. (It was
this construction that required us to pass to a situation where $p$ is
highly ramified in~$F$.) We are not, however, able to arrange that our
representations are connected at all the places~$v\nmid
p$. Fortunately, Taylor~\cite{tay} found a way to prove automorphy
lifting theorems when the representations fail to be connected at some
places~$v\nmid p$, using his so-called ``Ihara avoidance''
argument. This argument makes an ingenious use of two different local
deformation problems at places $v\nmid p$, which are congruent
modulo~$p$, and relates two corresponding patched modules of
automorphic forms. The key point which makes this argument possible is
to work with local deformation rings having the following ``unique
generalization'' property:
any generic point of their special fibre has a unique generalization
to the generic fibre. More geometrically, we need to avoid having two
distinct irreducible components in characteristic zero which specialize to a
common irreducible  component in  the  special fibre. 

In order to apply this argument one also needs the unique generalization property for
the deformation rings at the places~
$v|p$. %
This  was previously only
known in the Fontaine--Laffaille and ordinary contexts, in which case
the crystalline deformation rings can be understood completely
explicitly (and in the former case, there is even a unique irreducible
component).  (This problem was sidestepped to some extent
in~\cite{blgg,BLGGT}, but the approach there combines the Ihara
avoidance argument with the Khare--Wintenberger lifting argument to
produce characteristic zero lifts of residual representations of the
prescribed weight and level. In our~$\ell_0 > 0$ situation (in the
language of~\cite{CG}) such lifts do not always exist.)

One way to establish the unique generalization property (when it
holds) would be to explicitly compute the irreducible components of
the generic fibres of the deformation rings, but this appears to be
hopeless for crystalline deformation rings in any generality. However,
as was already observed in~\cite[\S 3]{tay} in the case~$v\nmid p$,
an alternative approach is to consider an
appropriate moduli stack of Galois representations, for which the deformation rings are versal rings at closed points. One shows that its
special fibre is generically reduced (or even generically smooth), for
example by showing that the deformation rings for generic choices of
the residual Galois representation are formally smooth. It then
follows that for an arbitrary residual representation, the special
fibres of the ($\Zp$-flat quotients of) the deformation rings are
generically reduced, which implies the unique generalization
property. (While~\cite[\S 3]{tay} does not explicitly work with moduli
stacks of Galois representations, \cite[Lem.\ 3.2]{tay} is easily
reformulated in these terms; and while ~\cite[Prop.\ 3.1(3)]{tay} does
not explicitly state that the $\Zp$-flat quotient of the deformation
ring has generically reduced special fibre, this follows from the
argument, as in~\cite[Prop.\ 3.1]{MR4190048}.)

The unique generalization property has subsequently been used by
Thorne in a context with~$v\nmid p$ in the proof of~\cite[Thm.\
8.6]{jack} (in order to avoid any additional hypotheses when
introducing an auxiliary prime to make the level structure
sufficiently small), and in the case that~$v|p$ by
Caraiani--Newton~\cite{caraiani-newton}, who used the results
of~\cite{caraiani2022geometric}, which establish the generic
reducedness of the special fibres of the crystalline deformation rings
in the 2-dimensional (tamely potentially) Barsotti--Tate case, by an
analysis of the corresponding Emerton--Gee stacks.

Unfortunately an (unconditional) argument with the Breuil--M\'ezard
conjecture shows that generic reducedness is extremely rare when~$v|p$
(see Remark~\ref{rem: comparison to CEGS}). We are however able to
prove the following theorem.
\begin{ithm}[Theorem~\ref{thm: special fibre weight 0 crystalline def ring generically
    reduced}]%
  \label{intro thm: special fibre weight 0 crystalline def ring generically
    reduced}Suppose that $p>n$, that $K/\Qp$ and~$\F/\Fp$ are finite extensions,
  and that ~$\rhobar: G_K \to \GL_n(\F)$ is a continuous representation.
  Let~$R^{\crys,0}$
  be the universal lifting ring for crystalline lifts of~$\rhobar$ of
  parallel Hodge--Tate weights $0,1,\dots,n-1$. Then the special fibre of
  $\Spec R^{\crys,0}$ is generically reduced.
\end{ithm}

We refer the reader to the introduction to Section~\ref{sec:genred}  for a
detailed overview of the proof of Theorem~\ref{intro thm: special fibre weight 0 crystalline def ring generically
    reduced}, which as above relies on proving the
corresponding property of the relevant Emerton--Gee
stacks~\cite{emertongeepicture} (whose versal rings are the
crystalline lifting rings). The irreducible components of the
special fibres of these stacks were described
in~\cite{emertongeepicture}, and we prove our result by combining this
description with a computation of extensions of rank~$1$ Breuil
modules. An amusing feature of this argument is that we prove a
result about the deformation rings of arbitrary $n$-dimensional
mod~$p$ representations %
by reducing to a calculation for reducible
$2$-dimensional representations.

\subsection{Bianchi Modular Forms} \label{Bianchi}
Let us specialize to the case when~$F/\Q$ is an imaginary quadratic field.
Let~$\pi$ be a regular algebraic cuspidal automorphic representation of~$\GL_2(\A_F)$.
Let~$\chi$
be the central character of~$\pi$. By definition, the representation~$\pi$ occurs 
in~$L^2_{\mathrm{cusp}}(\GL_2(F) \backslash \GL_2(\A_F))$. Let~$\fg$ be the Lie algebra of~$\GL_2(\C)$ as a real group.
The assumption that~$\pi$ is regular algebraic is equivalent to the condition that the infinitesimal character
of~$\pi_{\infty}$ is the same as~$V^{\vee}$ for an algebraic representation~$V$ of~$\mathrm{Res}_{F/\Q} \GL_2$.
The assumption that~$\pi$ is cuspidal places a restriction on~$V$ corresponding to the fact (noted earlier) 
that such~$\pi$
has parallel weight;
the corresponding 
representations are parametrized (up to twist) by an integer~$k \ge 2$, where~$k=2$ corresponds to the case
when~$V$ is trivial. 
 This choice of~$k$
determines the action of~$Z(\fg)$ on~$\pi_{\infty}$, and by taking functions which are suitable eigenvectors under~$Z(\fg)$,
we may arrive at certain vector valued Hecke eigenfunctions~$\Phi$ on~$\GL_2(\A_F)$ with Fourier expansions (\cite[\S1.2]{Williams},
\cite[\S6]{HidaCritical})
$$f \left[ \left( \begin{matrix} t & z \\
0 & 1 \end{matrix} \right) \right] =
|t|_F  \sum_{\alpha \in F^{\times}}
c(\alpha t \delta_F,f) W(\alpha t_{\infty}) e_F(\alpha z),$$
where~$\delta = \delta_F$ is the different, $\alpha t \delta$ can be interpreted as a fractional ideal of~$\OL_F$,
$c(I,f)$ is a Fourier
coefficient which vanishes unless~$I \subset \OL_F$ and which we may assume is normalized so that~$c(\OL_F,f) = 1$,
$W$ is an explicit Whittaker function which is vector valued in some explicit representation of~$\SU(2)$ depending on~$k$,
and~$e_F$ is an explicit additive character of~$F \backslash \A_F$. This has a direct
translation into more classical language, and
 can be interpreted as a collection of~$h_F$ functions on a finite union of hyperbolic 
 spaces~$\HH^3$.
 The explicit functions~$f$ (either adelically or classically) are known as Bianchi modular forms. For a normalized Bianchi eigenform~$f$
of weight~$k$ and level prime to~$\p$,  Theorem \ref{thm: Ramanujan thm main paper}  implies
 the following bound:
 
 \begin{ithm}\label{ithm_simplified_Ramanujan}
 Let~$f$ be a cuspidal Bianchi modular eigenform of level~$\mathfrak{n}$ and weight~$k$.
Let $\mathfrak{p}$ be a prime ideal of $\mathcal{O}_F$ not dividing $\mathfrak{n}$, and let $c(\mathfrak{p},f)$ be an eigenvalue of $T_\mathfrak{p}$ on $H$. Then 
\begin{equation}
\label{ramanujanbianchi}
|c(\mathfrak{p},f)| \le 2 N(\p)^{(k-1)/2}.
\end{equation}
\end{ithm}

This connects our theorem with the more classical version of the Ramanujan conjecture for modular
forms~\cite{Deligne} as discussed earlier in the introduction.

The eigenvalues~$c(\mathfrak{p},f)$ associated to~$f$ have a second interpretation
in terms of the cohomology of arithmetic groups and arithmetic hyperbolic~$3$-manifolds.
 The algebraic representations~$V$
of~$\Res_{F/\Q}\GL_2$ are all, up to twist, given on real
 points~$\Res_{F/\Q}\GL_2(\R) = \GL_2(\C)$
by the representations~$ \Sym^{k-2} \C^2 \otimes \overline{\Sym^{l-2} \C^2}$ for a pair of
integers~$k, l \geq 2$. 
Let $\mathfrak{n} \leq \mathcal{O}_F$ be a non-zero ideal. Having fixed~$k$ and~$l$,  we can form the group cohomology
\[ H = H^1(\Gamma_1(\mathfrak{n}), \Sym^{k-2} \C^2 \otimes \overline{\Sym^{l-2} \C^2}) \]
of the standard congruence subgroup $\Gamma_1(\mathfrak{n}) \leq \GL_2(\mathcal{O}_F)$. Then $H$ is a finite-dimensional $\C$-vector space. 
Let~$H_{\ppr} \subset H$ denote the subgroup consisting of classes which
vanish under the restriction of~$H$ to~$H^1(P, \Sym^{k-2} \C^2 \otimes \overline{\Sym^{l-2} \C^2})$ for 
 any parabolic subgroup~$P \subset \Gamma_1(\mathfrak{n})$.
 More geometrically, one can interpret~$H$ as the cohomology of a local system 
on the Bianchi manifold~$Y_1(\mathfrak{n}) = \HH^3/\Gamma_1(\mathfrak{n})$.
If~$X_1(\mathfrak{n})$ is the
Borel--Serre compactification of~$Y_1(\mathfrak{n})$,  then parabolic cohomology
consists of classes which are trivial on the boundary $X_1(\mathfrak{n}) \setminus Y_1(\mathfrak{n})$; this boundary may be identified
(when~$\Gamma_1(\mathfrak{n})$ is torsion free)
with a finite disjoint union of  complex tori.
The spaces $H$  and~$H_{\ppr}$ are  equipped with a commuting family of linear operators, the unramified Hecke operators $T_{\mathfrak{p}}$, indexed by the principal ideals $\mathfrak{p} \leq \mathcal{O}_F$ not dividing $\mathfrak{n}$. 
More precisely,
 if one writes~$\mathfrak{p} = (\pi)$, then the group~$A \Gamma_1(\mathfrak{n}) A^{-1}
\cap \Gamma_1(\mathfrak{n}) = \Gamma_1(\mathfrak{n},\mathfrak{p})$ has finite index in~$\Gamma_1(\mathfrak{n})$,  where
$$A = \left(\begin{matrix} \pi & 0 \\ 0 & 1 \end{matrix} \right);$$
the map~$T_{\mathfrak{p}}$ is induced by composing (in a suitable order)
a restriction map, a conjugation by~$A$ map, and a trace map respectively.
Note that the existence of (a large family of) such operators comes from 
the fact that~$\Gamma = \GL_2(\OL_F)$
has infinite index inside its commensurator in~$\GL_2(F)$;  as shown
by Margulis~\cite[Thm.\ IX.1.13]{Margulis}, this characterizes the arithmeticity of~$\Gamma$.

In order to obtain an action of
~$T_{\mathfrak{p}}$ for more general prime ideals~$\mathfrak{p}$, one needs to 
replace~$Y_1(\mathfrak{n})$ by a disconnected union of~$h_F$
commensurable arithmetic hyperbolic manifolds~$Y_1(\mathfrak{n};\mathfrak{a}) = \HH/\Gamma_1(\mathfrak{n};\mathfrak{a})$
indexed by ideals~$\mathfrak{a}$ in the
class group~$\Cl(\OL_F)$ of~$F$ prime to~$\mathfrak{n}$, and where~$Y_1(\mathfrak{n};\OL_F) = Y_1(\mathfrak{n})$.
The group~$\Gamma(\OL_F;\mathfrak{a})$ is the automorphism group
of the~$\OL_F$-module~$\OL_F \oplus \mathfrak{a}$, which consists explicitly of matrices
of the form
$$\left( \begin{matrix} \OL_F & \mathfrak{a}^{-1} \\ \mathfrak{a} & \OL_F \end{matrix} \right)
\cap \GL_2(F)$$
with determinant in~$\OL^{\times}_F$.
 The space~$H_{\ppr}$ vanishes unless~$k = l$ (\cite[3.6.1]{Harder}),
   which we now assume.
 If~$h_F = 1$,
the Eichler--Shimura isomorphism (\cite[\S3.6]{Harder}) gives a map from 
Bianchi cuspidal modular eigenforms~$f$  of weight~$k$ as described above
and cohomology classes~$\eta_f \in H_{\ppr}$  which are simultaneous
eigenforms for all the Hecke operators. Moreover, the eigenvalues of~$T_{\mathfrak{p}}$
on~$\eta_f$ are given exactly by~$c(\mathfrak{p},f)$. 
If~$h_F > 1$, one must replace~$H$ and~$h_{\ppr}$ by the direct sum of the corresponding
cohomology groups over~$Y_1(\mathfrak{n};\mathfrak{a})$ for~$\mathfrak{a} \in \Cl(\OL_F)$.
Theorem~\ref{ithm_simplified_Ramanujan} now implies:

\begin{ithm}\label{ithm_simplified_Ramanujanagain}
Let $\mathfrak{p}$ be a principal prime ideal of $\mathcal{O}_F$ not dividing $\mathfrak{n}$, and let $a_{\mathfrak{p}}$ be an eigenvalue of $T_\mathfrak{p}$ on $H_{\ppr}$. Then $| a_{\mathfrak{p}} | \leq 2 N(\mathfrak{p})^{(k-1)/2}$. 
\end{ithm}

These explicit formulations of our theorems can be generalized in a number of ways.
Remaining in the setting of arithmetic hyperbolic~$3$-manifolds (or orbifolds), 
we can replace~$\GL_2(\OL_F)$ by a congruence subgroup~$\Gamma$ of the
norm
one units in a maximal order~$\OL$ of a division algebra~$D/F$ where~$F \hookrightarrow \C$
is a number field with one complex place and~$D$ is definite at all real
places of~$F$. When~$[F:\Q] = 2$, we %
obtain the Bianchi manifolds as above (when~$D/F$ is split)  but also certain compact hyperbolic arithmetic three manifolds; our theorem applies equally well in the latter case (note that~$H=H_{\ppr}$ in this setting). On the other hand, suppose that~$F$ has at least one real place; for example, take~$F = \mathbf{Q}[\theta]/(\theta^3 - \theta + 1)$, let $k = 2$, let~$D/F$
be ramified at the real place and the unique prime of norm~$5$. Now~$H = H_{\ppr}$ is
 the first cohomology group of a congruence cover of the Weeks manifold.
The generalized Ramanujan
conjecture  still predicts a bound of the shape~$|a_{\p}| \le 2 N(\p)^{1/2}$ for the eigenvalues
of the Hecke operators~$T_{\mathfrak{p}}$. 
However,  our
methods do not apply in this situation, and the best current bounds remain those of 
the form~$|a_{\p}| \le 2 N(\p)^{1/2 + 7/64}$ proved using analytic methods (see~\cite{Sarnak}).

We finish with an application of a different sort. Let~$\Gamma = \SL_2(\OL_F)$.
The quotients~$\Gamma \backslash \HH^3$
  were first investigated by  Bianchi~\cite{Bianchi}, and for that reason they
 are known as Bianchi orbifolds.
For a Bianchi modular form~$f$ of level one, one may~\cite[\S3]{marshall}
 associate to~$f$ a normalized measure~$\mu_f$ on~$\Gamma \backslash \HH^3$.
  One then has the following~\cite[Cor~3]{marshall}:
 
 \begin{ithm}  Assume that~$F$ has class number one\footnote{The
 paper~\cite{marshall} has the following to say about this assumption:
 ``For simplicity, we assume our fields
to have narrow class number one throughout the paper, but this is not essential.''
One might therefore expect it to be possible to prove the
 more general adelic statement for all imaginary quadratic~$F$ under the additional
hypothesis (as explained in~\cite[\S1]{everymanwilldohisduty}) that, when~$h_F$ is even, one avoids
certain dihedral forms   which   vanish identically on half of the connected
components of the adelic quotient.
Similarly, concerning the assumption on the level, the paper~\cite{marshall} says
``The proof may easily be modified to allow a nontrivial level in any case.''}.
 For any sequence of Bianchi modular eigenforms~$f$ of weight tending to~$\infty$,
the measures~$\mu_f$ converge
 weakly to the hyperbolic volume on~$Y = \SL_2(\OL_F) \backslash \HH^3$.
 \end{ithm}
 
 \begin{proof}
 As noted in~\cite{marshall-erratum}, the proof given in~\cite{marshall} assumes the Ramanujan conjecture
 for Bianchi modular forms --- this is now a consequence of Theorem~\ref{ithm_simplified_Ramanujan}.
 \end{proof}
\subsection{Recent work of Matsumoto} A few months after the first preprint version of this work was circulated, a remarkable new work by Matsumoto appeared \cite{matsumoto} which proves Theorems A and B with no parallel weight condition.
Matsumoto's approach introduces several new ideas of a global nature, whilst our approach here is based on refining our understanding of the local ingredients in the potential automorphy argument. For this reason, one might hope that the two approaches could be profitably combined in the future.

 \subsection{Acknowledgements}\label{subsec: acknowledge}Some of the
 ideas in Section~\ref{sec:genred} were found in joint discussions
 with Matthew Emerton, and we are grateful to him for allowing us to
 include them here, as well as for his assistance in proving
 Theorem~\ref{thm:Xdred is algebraic}~\eqref{item: we can get generic
   Ext classes in rhobar}. We would also like to thank Patrick Allen and
 Matthew Emerton for their comments on an earlier
 version of the paper, together with an anonymous referee whose careful reading and many comments were very helpful. Thanks to Dat Pham for pointing out that the assertion made in Remark \ref{rem:twistedphi} of the published version of the paper is incorrect. 

 \subsection{Notation}\label{subsec: notation}
 Let $K/\Qp$ be a finite
extension. %
If $\sigma:K \into \Qpbar$ is a continuous
embedding of fields then we will write $\HT_\sigma(\rho)$ for the
multiset of Hodge--Tate numbers of $\rho$ with respect to $\sigma$,
which by definition contains $i$ with multiplicity
$\dim_{\Qpbar} (W \otimes_{\sigma,K} \widehat{\barK}(i))^{G_K} $. We
write~$\varepsilon$ for the $p$-adic cyclotomic character,
which is a crystalline representation with $\HT_\sigma(\varepsilon)=\{
-1\}$ for each~$\sigma$.

We say that~$\rho$ has weight~$0$ if for each~$\sigma:K\into\Qpbar$ we
have~$\HT_\sigma(\rho)=\{0,1,\dots,d-1\}$. We often somewhat abusively
write that a representation $\rho:G_K\to\GL_d(\Zpbar)$ is crystalline of
weight~$0$ if the corresponding
representation $\rho:G_K\to\GL_d(\Qpbar)$ is crystalline of weight~$0$.

Let~$\cO$ be the ring of integers in some finite extension~$E/\Qp$,
and suppose that~$E$ is large enough that it contains the images of
all embeddings~$\sigma:K\into\Qpbar$. Write~$\varpi$ for a uniformizer
of~$\cO$, and~$\cO/\varpi=\F$ for its residue field.  We write
$\Art_K: K^\times\to W_K^{\operatorname{ab}}$ for the isomorphism of
local class field theory, normalized so that uniformizers correspond
to geometric Frobenius elements. 

Let $\rhobar:G_K\to\GL_d(\Fpbar)$ be a continuous representation. Then
after enlarging~$E$ and thus~$\F$ if necessary, we may assume that the image of~$\rhobar$ is contained in~$\GL_d(\F)$.
We write~$R^{\square,\cO}_{\rhobar}$ for the universal 
lifting~$\cO$-algebra of~$\rhobar$; by definition, this (pro-)represents the
functor $\cD^{\square,\cO}_{\rhobar}$ given by lifts of~$\rhobar$ to representations $\rho: G_K \to
\GL_d(A)$, for~$A$ an Artin local $\cO$-algebra with residue field~$\F$. 
The precise choice of~$E$ is unimportant, in the sense that
if~$\cO'$ is the ring of integers in a finite extension~$E'/E$, then
by~\cite[Lem.\ 1.2.1]{BLGGT} we have
$R^{\square,\cO'}_{\rhobar}=R^{\square,\cO}_{\rhobar}\otimes_{\cO}\cO'$.

We write
~$R^{\crys,\underline{0},\cO}_{\rhobar}$  
for the unique $\cO$-flat quotient
of~$R_{\rhobar}^{\square,\cO}$ with the property that if~$B$ is a
finite flat~$E$-algebra, then an $\cO$-algebra homomorphism
$R_{\rhobar}^{\square,\cO}\to B$ factors through ~$R^{\crys,0,\cO}_{\rhobar}$
 if and only if the corresponding representation of~$G_K$ is
 crystalline  of weight~$0$. %

We will let $\operatorname{rec}_K$ be the local Langlands correspondence of
\cite{ht}, so that if $\pi$ is an irreducible complex
admissible representation of $\GL_n(K)$, then $\operatorname{rec}_K(\pi)$ is a
Frobenius semi-simple Weil--Deligne representation of the Weil group $W_K$. 
We write $\operatorname{rec}^T_K$ for the arithmetic normalization of the local Langlands correspondence, as defined in e.g.\ ~\cite[\S 2.1]{Clo14}; it is defined on irreducible admissible representations of $\GL_n(K)$ defined over any field which is abstractly isomorphic to $\C$ (e.g.\ $\overline{\Q}_l$).

 Let $F$ be a number field. If~$v$ is a finite place of~$F$ then we
 write~$k(v)$ for the residue field of~$F_v$. We identify dominant weights $\lambda$ of $\Res_{F/\Q}\GL_n$ with sets of tuples of integers $(\lambda_{\tau,1}\ge \lambda_{\tau,2} \ge \cdots \ge \lambda_{\tau,n})_{{\tau: F\hookrightarrow \C}}$ indexed by complex embeddings of $F$ (cf.~\cite[\S 2.2.1]{10author}). If $\pi$ is an irreducible admissible representation of
$\GL_n(\A_F)$ and $\lambda$ is a dominant weight, we say that $\pi$
is regular algebraic of weight $\lambda$ if the infinitesimal character of $\pi_\infty$ is the same as that of $V_\lambda^\vee$, where~$V_\lambda$ is the algebraic
representation of $\Res_{F/\Q}\GL_n$ of highest weight~$\lambda$. We say that $\pi$ is regular
algebraic if it is regular algebraic of some weight.

\begin{df}[Parallel Weight] \label{parallelweight} Suppose $\pi$ is regular algebraic of weight $\lambda$. We say that $\pi$ is of
	parallel weight if $\lambda_{\tau, 1} - \lambda_{\tau, 2}$ is
	independent of $\tau$; equivalently, if $\pi$ admits a regular
	algebraic twist of weight $\mu = (m-1, 0)_\tau$ for some
	$m \ge 1$ in~$\Z$.
	We say that~$\pi$ has parallel weight~$k$ for some
	integer~$k \ge 2$ if~${\mu} = (k-2, 0)_\tau$.
\end{df}

Let $F$ be an imaginary CM field, and let $\pi$ be a cuspidal, regular algebraic weight $\lambda$ automorphic representation of $\GL_2(\A_F)$. 
The weight ${\lambda} = (\lambda_{\tau, 1}, \lambda_{\tau, 2})_\tau \in (\Z^2)^{\Hom(F, \C)}$ satisfies:
\begin{itemize}
	\item There is an integer $w \in \Z$ such that for all $\tau$, we have $\lambda_{\tau, 1} + \lambda_{\tau c, 2} = w$. In particular, for all $\tau$ we have $\lambda_{\tau, 1} - \lambda_{\tau, 2} = \lambda_{\tau c, 1} - \lambda_{\tau c, 2}$.
\end{itemize} 
This is a consequence of Clozel's purity lemma \cite[Lemma 4.9]{MR1044819}. In particular, if $F$ is imaginary quadratic, $\pi$ is necessarily of parallel weight.

\section{The special fibres of weight~$0$
  crystalline lifting rings are generically reduced}\label{sec:genred}%
The
goal of this section is to prove Theorem~\ref{thm: special fibre
  weight 0 crystalline def ring generically reduced}, which shows that
if~$p>n$, then for any finite extension ~$K/\Qp$ and any
$\rhobar:G_K\to\GL_n(\Fpbar)$, the special fibre of the corresponding
weight~$0$ crystalline lifting ring is generically reduced. We
deduce this from the corresponding statement for the special fibre of
the weight~$0$ crystalline Emerton--Gee stack. This stack was
introduced in~\cite{emertongeepicture}. We recall the results from~\cite{emertongeepicture} that we
need in Section~\ref{subsec: special fibre EG stacks} below, but for
this introduction the key points are as follows: the full Emerton--Gee
stack~$\cX$ is a stack of $(\varphi,\Gamma)$-modules which sees
all~$\rhobar$ at once, and whose versal ring at any~$\rhobar$ is the
corresponding unrestricted lifting ring; and the
weight~$0$ crystalline Emerton--Gee stack~$\cX^{\underline{0}}$ is a
closed substack whose versal ring at any~$\rhobar$ is the
corresponding weight~$0$ crystalline lifting ring.

Generic reducedness for the (special fibre of the) stack~$\cX^{\underline{0}}$ is equivalent
to the generic reducedness for the special fibres of the crystalline
lifting rings, as we show by a direct argument below (in the
proofs of Theorem~\ref{thm: BM cycle weight 0} and Theorem~\ref{thm:
  special fibre weight 0 crystalline def ring generically
  reduced}). Working on the stack allows us to argue more
geometrically, and in particular one of the main theorems
of~\cite{emertongeepicture} classifies the irreducible components of
the underlying reduced substack of~$\cX$, and shows that the
underlying reduced substack of the special fibre $\cXbar^{\underline{0}}$
of~$\cX^{\underline{0}}$ is a union of these irreducible components.

In order to show that the 
~$\cXbar^{\underline{0}}$ is generically reduced, it therefore suffices
to determine which irreducible components are contained in this
special fibre, and to show that $\cXbar^{\underline{0}}$ is reduced at
a generic point of each such component.

The classification in ~\cite{emertongeepicture} of  the irreducible components is via a description of the generic ~$\rhobar$ which occur on that
component. These are all of the form
\[\rhobar\cong \begin{pmatrix}
     \chibar_1 &*&\dots &*\\
      0& \chibar_2&\dots &*\\
      \vdots&& \ddots &\vdots\\
      0&\dots&0& \chibar_n\end{pmatrix}\] where
  the~$\chibar_i:G_K\to \Fpbar^{\times}$
  are characters and the extension
  classes~$*$ are in generic position (in particular nonsplit). The
  characters~$\chibar_i|_{I_K}$ are fixed on each irreducible
  component, and the components are usually determined by the data of
  the ~$\chibar_i|_{I_K}$ (see Theorem~\ref{thm:Xdred is algebraic}~(3) for a precise statement).

In order to prove our results, we show that the condition that a
generic such~$\rhobar$ has a crystalline lift of weight~$0$
seriously constrains the possible~$\chibar_i$. We use a theorem of
Tong Liu~\cite{MR2388556}, which in particular shows (under the
assumption that~$p>n$) that~$\rhobar$ is obtained from a (crystalline)
Breuil module. In the case~$n=2$, we can then argue as follows: we can
compute the possible extensions of rank 1 Breuil modules, and we find
that a sufficiently generic extension of~$\chibar_2$ by~$\chibar_1$
can only have a crystalline lift of weight~$0$ if it is ordinary, in
the sense that ~$\chibar_1$ is unramified and~$\chibar_2$ is an
unramified twist of~$\varepsilonbar^{-1}$, where $\varepsilonbar$ is
the mod~$p$ cyclotomic character. 

Furthermore, a generic such
extension arises from a unique Breuil module, which is ordinary.
(Since two-dimensional weight~$0$ crystalline representations are
given by the generic fibres of $p$-divisible groups, we can
alternatively phrase this result as showing that~$\rhobar$ comes from
a unique finite flat group scheme over~$\cO_K$, which is ordinary in
the sense that it is an extension of a multiplicative by an \'etale
group scheme.) By an argument of Kisin~\cite[Prop.\ 2.4.14]{kis04},
the deformations of this Breuil module are also ordinary, so that the
weight~$0$ crystalline lifting rings for these generic~$\rhobar$'s
are also ordinary. It is then easy to show that the crystalline ordinary
lifting ring for a generic~$\rhobar$ is formally smooth, and thus has reduced special
fibre, which completes the argument.

Perhaps surprisingly, we are able to make a similar argument for
general~$n$, without making any additional calculations. For each~$i$,
we apply our computation of extension classes of Breuil modules to the
extension of ~$\chibar_{i+1}$ by~$\chibar_i$ arising as a subquotient
of~$\rhobar$. If all of these extensions are sufficiently generic, we
show that~$\rhobar$ can only admit crystalline lifts of weight~$0$
if~$\chibar_i|_{I_K}=\varepsilonbar^{1-i}$ for all~$i$. Furthermore,
we also see that a generic such~$\rhobar$ can only arise from an ordinary
Breuil module, and again deduce that all weight~$0$ crystalline lifts
of~$\rhobar$ are ordinary. From this we deduce the formal smoothness of the
corresponding lifting rings for generic~$\rhobar$, and conclude as above.

The organization of the proof is as follows. In
Section~\ref{subsec:Breuil modules and strongly divisible modules} we
recall Liu's results~\cite{MR2388556} on strongly divisible modules and lattices in
semistable representations, and deduce the results that we need on
crystalline Breuil modules. In Section~\ref{subsec:
  extensions rank one Breuil modules} we compute extensions of rank
one Breuil modules. This is essentially elementary, using only
semilinear algebra and some combinatorics. We deduce from this in
Section~\ref{subsec: generic weight 0} that sufficiently generic
crystalline representations of weight~$0$ are ordinary. In Section~\ref{subsec:
  special fibre EG stacks} we give a brief introduction to
Emerton--Gee stacks, and prove some slight generalizations  of some
results of~\cite{emertongeepicture}, before deducing our generic
reducedness results in Section~\ref{subsec: generic reducedness}.

\subsection{Breuil modules and strongly divisible modules}\label{subsec:Breuil modules and strongly divisible modules}We begin by
recalling some standard results about Breuil modules and Breuil--Kisin
modules. The results we use are largely due to Breuil, Kisin and Liu,
but for convenience we mostly cite the papers~\cite{MR3079258,
  MR3705291} which deduce versions of these results with coefficients
and prove some exactness properties of the functors to Galois
representations which we will make use of in our main arguments. (Note
that~\cite[App.\ A]{MR3705291} makes a running assumption on the
ramification of the field $K/\Qp$, but this is only made in order to
discuss tame descent data and compare to Fontaine--Laffaille theory,
and it is easy to check that all of the results we cite from there are
valid for general $K/\Qp$ with trivial descent data, with identical
proofs (or often with simpler proofs, as there is no need to consider
the descent data).) %

Let $K/\Qp$ be a finite extension for some $p>2$, with ring of
integers~$\cO_K$ and residue field~$k$. Write~$e$ for the absolute
ramification degree of~$K$, and~$f$ for its inertial degree
$[k:\Fp]$. Fix a uniformizer $\pi\in K$ with Eisenstein polynomial~$E(u)$, which
we choose so that ~$E(0)=p$. %
Fix also a compatible choice~$(\pi^{1/p^n})_{n\ge 1}$ of
$p$-power roots of~$\pi$ in~$\Qpbar$, and set ~$K_n:=K(\pi^{1/p^n})$ and~$K_\infty:=\cup_{n\ge 1}K_n$.

Let $E/\Qp$ be a finite extension containing the normal closure
of~$K$, with ring of integers~$\cO$ and residue field~$\F$. We will
consider various semilinear algebra objects with coefficients in a
finite $\cO$-algebra $A$, and it is trivial to verify that all of our
definitions are compatible with extension of scalars of~$A$ in an
obvious way. In particular, we often take $A=\F$, but we are free to
replace~$\F$ by an arbitrary finite extension, or (after passing to a
limit in an obvious fashion) by~$\Fpbar$.

For any finite $\cO$-algebra $A$ we let
$\gS_A:=(W(k)\otimes_{\Zp}A)\llb u\rrb $, equipped with the usual
$A$-linear, $W(k)$-semilinear Frobenius endomorphism~ $\varphi$, with
$\varphi(u)=u^p$. For any integer
$h\ge 0$, a \emph{Breuil--Kisin module with $A$-coefficients} of
height at most~$h$ is a finite free $\gS_{A}$-module $\gM$ equipped
with a $\varphi$-semilinear map $\varphi:\gM\to\gM$ such that the
cokernel of the linearized Frobenius
$\varphi^*\gM\xrightarrow{1\otimes\varphi}\gM$ is killed
by~$E(u)^h$, where as usual $\varphi^*\gM$ denotes the Frobenius pullback $\gS_A\otimes_{\varphi,\gS_A}\gM$.  (Here we indulge in
a standard abuse of notation in writing $\varphi$ for both the
endomorphism of~$\gS_A$ and of~$\gM$, which should not cause any
confusion.)

Suppose that~$A=\F$, and let
$\overline{S}_{\F}:=\gS_{\F}/u^{ep}$. %
If~$h\le p-2$, then a \emph{quasi-Breuil module with $\F$-coefficients}~$\cM$ of height~$h$ is a
finite free $\overline{S}_\F$ module~$\cM$ equipped with a
$\overline{S}_\F$-submodule \[u^{eh}\cM\subseteq\cM^h\subseteq\cM\] and
a $\varphi$-semilinear map $\varphi:\cM^h\to\cM$ such
that \[\overline{S}_{\F}\cdot\varphi(\cM^h)=\cM.\] (The morphism~$\varphi$ is
usually denoted~$\varphi_h$ in the literature, but we will shortly fix
the choice $h=p-2$ for the rest of the paper, so we have omitted the
subscript for the sake of cleaner notation.)%

For each $0\le h\le p-2$, there is by~\cite[Thm.\ 4.1.1]{MR1695849} an
equivalence of categories between the category of Breuil--Kisin
modules with $\F$-coefficients of height at most~$h$ and the category
of quasi-Breuil modules with $\F$-coefficients of height at most~$h$.
Explicitly, a Breuil--Kisin module~$\gM$ of height $h\le p-2$
determines a quasi-Breuil module as follows. Write
$\gM^h:=(1\otimes\varphi)^{-1}(u^{eh}\gM)\subseteq\varphi^*\gM$. Set
$\cM:=\varphi^*\gM/u^{pe}$, and
$\mathcal{M}^h=\mathfrak{M}^h/u^{pe}\varphi^*\mathfrak{M}$. Then
$\varphi:\cM^h\to\cM$ is defined by the
composite
\[\cM^h\xrightarrow{1\otimes
    \varphi}u^{eh}\overline{S}_{\F}\otimes_{\gS_{\F}}\gM\xrightarrow{\varphi_h\otimes
    1} \overline{S}_{\F}\otimes_{\varphi,\gS_{\F}}\gM=\cM,\] where
$\varphi_h:u^{eh}\overline{S}_{\F}\to \overline{S}_{\F}$ is the $\varphi$-semilinear
morphism $\varphi_h(u^{eh}x):=\varphi(x)$. (Note that this is
well-defined because if~$x$ is divisible by $u^{e(p-h)}$, then
$\varphi(x)$ is divisible by $u^{ep(p-h)}$ and in particular
by~$u^{ep}=0$.) %
  We will often say that the Breuil--Kisin
module~$\gM$ \emph{underlies} the quasi-Breuil module~$\cM$.%

We define $c \in (\overline{S}_{\F})^\times$ to be the image of $\varphi(E(u))/p$ in $\overline{S}_{\F}$. We note that $c = \varphi(d)$ for $d \in (\overline{S}_{\F})^\times$.

\begin{rem}\label{rem:twistedphi}
	The equivalence between Breuil--Kisin modules and quasi-Breuil modules recalled above is usually defined with the map $\varphi_h(u^{eh}x) := c^h\varphi(x)$. It is easily checked that multiplying by $d^h$ gives an isomorphism between the quasi-Breuil modules with differently scaled $\varphi$. 
\end{rem}

Write $N:\overline{S}_{\F}\to \overline{S}_{\F}$ for the $(k\otimes_{\Fp}\F)$-linear
derivation $-u\frac{\partial}{\partial u}$.  A \emph{Breuil module with $\F$-coefficients} $\cM$ of height~$h$ is a
quasi-Breuil module equipped with the additional data of a map $N:\cM\to\cM$ which satisfies:
\begin{itemize}
\item $N(sx)=sN(x)+N(s)x$ for all $s\in \overline{S}_{\F}, x\in\cM$,
\item   $u^{e}N(\cM^h)\subseteq\cM^h$, 
\item and $\varphi(u^{e}N(x))=c N(\varphi(x))$ for all $x\in\cM^h$.
\end{itemize}
We say that a Breuil module~$\cM$ is \emph{crystalline} if
$N(\cM)\subseteq u\cM$.%

\begin{rem}
  \label{rem: another u divisibility for N}While we will not
  explicitly need this below, it can be checked that if~$\cM$ is
  crystalline, then $u^eN(\cM^h) \subseteq u N(\cM^h)$. To see this,
  note that since $\overline{S}_{\F}\cdot\varphi(\cM^h)=\cM$, there is
  an induced $\Fp$-linear surjection
  $\cM^h/u\cM^h\to\cM/u\cM$, which is in fact an isomorphism
  (comparing dimensions as in \cite[Lem.~2.2.1.1]{Breuil-ENS}). %
If~$\cM$ is crystalline then~$N$ acts by~$0$ on~$\cM/u\cM$,
and the commutation relation between~$N$ and~$\varphi$ then shows that
$u^eN$ acts by~$0$ on $\cM^h/u\cM^h$, as required.
\end{rem}

We now define the Galois representations associated to Breuil modules
and to Breuil--Kisin modules, beginning with the latter. An
\emph{\'etale $\varphi$-module with $\F$-coefficients} is by
definition a finite free $(k\otimes_{\Fp}\F)((u))$-module~$M$ with a
semilinear endomorphism~$\varphi:M\to M$ such that the linearized
Frobenius $\varphi^*M\xrightarrow{1\otimes\varphi}M$ is an
isomorphism. Note that by definition, if $\gM$ is a Breuil--Kisin
module with $\F$-coefficients, then $\gM[1/u]$ is an \'etale
$\varphi$-module with
$\F$-coefficients. Let~$k((u))^{\operatorname{sep}}$ denote a
separable closure of~$k((u))$. By the results of~\cite{MR1106901} (see
e.g.\ \cite[1.1.12]{kis04}), the 
functor
\[T_\infty:M\mapsto(M\otimes_{k((u))}k((u))^{\operatorname{sep}})^{\varphi=1} \]
is an equivalence of categories between the category of \'etale
$\varphi$-modules with $\F$-coefficients and the category of
continuous representations of $G_{K_\infty}$ on $\F$-vector spaces,
and we have
$\dim_{\F}T_\infty(M)=\rank_{(k\otimes_{\Fp}\F)((u))}M$. We also write
$T_\infty$ for the induced functor from Breuil--Kisin modules to
$G_{K_\infty}$-representations given by
$\gM\mapsto T_\infty(\gM[1/u])$. Similarly, if~$\gM$ is the
Breuil--Kisin module underlying a quasi-Breuil module~$\cM$, we write
$T_\infty(\cM)$ for~$T_\infty(\gM)$.

Similarly, there is a functor~$T$ from the category of Breuil modules
of height at most~$h$ with $\F$-coefficients to the category of
continuous representations of $G_{K}$ on $\F$-vector spaces
defined
by
\[T(\cM):=\Hom_{k[u]/u^{ep},\varphi,N}(\cM,\widehat{A})^\vee,\]where
$\widehat{A}:=\widehat{A}_{\operatorname{st}}\otimes_Sk[u]/u^{ep}$ is
defined for example in~\cite[(A.3.1)]{MR3705291}. Again we have
$\dim_{\F}T(\cM)=\rank_{\overline{S}_{\F}}\cM$. Furthermore, by~\cite[Prop.\
A.3.2]{MR3705291} the forgetful functor from Breuil modules to
quasi-Breuil modules induces an
isomorphism \[T(\cM)|_{G_{K_\infty}}\isoto T_\infty(\cM).\]

From now on all of our Breuil modules will be crystalline and have
height~$(p-2)$. We write $\crFBrMod$ for the category of crystalline
Breuil modules of height~$(p-2)$ with~$\F$-coefficients, and
$\qFBrMod$ for the category of quasi-Breuil modules of height~$(p-2)$
with $\F$-coefficients, which we identify with the category of
Breuil--Kisin modules of height at most~$(p-2)$ with
$\F$-coefficients. We say that a complex \numequation\label{eqn: ses
  of BrMod}0\to\cM_1\to\cM\to\cM_2\to 0\end{equation} in $\crFBrMod$
or $\qFBrMod$ is exact if it induces exact sequences of
$\overline{S}_{\F}$-modules $0\to\cM_1\to\cM\to\cM_2\to 0$
and \[0\to\cM_1^{p-2}\to\cM^{p-2}\to\cM_2^{p-2}\to 0.\] It is easily
checked that a complex of quasi-Breuil modules is exact if and only if
the corresponding complex of Breuil--Kisin modules is exact (as a
complex of $\gS_{\F}$-modules).

If~$\cM$ is an object of $\crFBrMod$, then an $\overline{S}_{\F}$-submodule
$\mathcal{N}\subseteq\cM$ is a \emph{Breuil submodule of~$\cM$} if it is a
direct summand of~$\cM$ as a $k[u]/u^{ep}$-module, and we furthermore
have $N(\mathcal{N})\subseteq\mathcal{N}$ and
$\varphi(\mathcal{N}\cap\cM^r)\subseteq\mathcal{N}$. Then~$\mathcal{N}$ inherits the structure
of a crystalline Breuil module from~$\cM$, as does the quotient
$\cM/\mathcal{N}$, and by~\cite[Lem.\ 2.3.2]{MR3705291}, the complex of
crystalline Breuil modules \[0\to\mathcal{N}\to\cM\to\cM/\mathcal{N}\to 0\] is exact;
and conversely if~\eqref{eqn: ses of BrMod} is exact, then $\cM_1$ is
a Breuil submodule of~$\cM$.

\begin{thm}
  \label{thm: properties of crystalline Breuil modules}\leavevmode
  \begin{enumerate}
  \item The categories $\crFBrMod$ and $\qFBrMod$ are  exact categories
    in the sense of~\cite{MR2655184}, and the functors~$T$
    and~$T_\infty$ are exact.
  \item For any object~$\cM$ of $\crFBrMod$, there is an order
    preserving bijection~$\Theta$ between the Breuil submodules
    of~$\cM$ and the $G_K$-subrepresentations of~$T(\cM)$, taking
    $\mathcal{N}$ to the image of $T(\mathcal{N})\into T(\cM)$. Furthermore if
    $\cM_1\subseteq\cM_2$ are Breuil submodules of~$\cM$, then
    $\Theta(\cM_2)/\Theta(\cM_1)\cong T(\cM_2/\cM_1)$.
  \end{enumerate}
\end{thm}
\begin{proof}The statement for quasi-Breuil modules is~\cite[Thm.\
  2.1.2]{MR2951749}.  The rest of the theorem for
  not-necessarily-crystalline Breuil modules is~\cite[Prop.\ 2.3.4,
  2.3.5]{MR3705291}. The case of crystalline Breuil modules follows
  formally, because (as noted above) a Breuil submodule of a
  crystalline Breuil module is automatically crystalline (as is the
  corresponding quotient submodule). Alternatively, it is
  straightforward to check that the proofs of ~\cite[Prop.\ 2.3.4,
  2.3.5]{MR3705291} go through unchanged, once one notes that the
  duality on Breuil modules~\cite[Defn.\ 3.2.8]{MR3079258} by
  definition preserves
  the subcategory of crystalline Breuil modules.
\end{proof}

We now show that any Galois representation obtained
as the reduction mod~$p$ of a lattice in a crystalline representation
with Hodge--Tate weights in the range~$[0,p-2]$ comes from a
crystalline Breuil module. This is essentially an immediate
consequence of the main theorem of Liu's paper~\cite{MR2388556}, which
proves an equivalence of categories between $G_K$-stable lattices inside
 semistable representations with Hodge--Tate weights in the range
 $[0,p-2]$ and strongly divisible modules. From this one can easily deduce
 an equivalence of categories between $G_K$-stable lattices inside
 crystalline representations with Hodge--Tate weights in the range
 $[0,p-2]$ and an appropriate category of ``crystalline strongly
 divisible lattices'', but since we do not need this, we avoid
 recalling the definitions of strongly divisible lattices and leave it
 to the interested reader.%

 Recall that by the results of~\cite{MR2263197}, if  $\rho:G_K\to\GL_n(\cO)$ is a lattice in a crystalline
 representation with non-negative Hodge--Tate weights,  there is a  Breuil--Kisin
 module with $\cO$-coefficients ~$\gM_{\cO}$ associated to $\rho|_{G_{K_\infty}}$
 (see e.g.\ \cite[Thm.\ 3.2(3),
 Prop.\ 3.4(3)]{MR3164985} for a precise reference
 allowing~$\cO$-coefficients).  
\begin{thm}
  \label{thm: reductions of crystalline reps are crystalline Breuil
    modules}Let $\rho:G_K\to\GL_n(\cO)$ be a lattice in a crystalline
  representation with Hodge--Tate weights in~$[0,h]$ for some
  integer~$0\le h\le p-2$, and write
  $\rhobar:G_K\to\GL_n(\F)$ for its reduction modulo~$\m_{\cO}$. Then
  there is a crystalline Breuil module~$\cM$ with $\F$-coefficients
  such that $\rhobar\cong T(\cM)$. Furthermore, the underlying
  Breuil--Kisin module of~$\cM$ has height at most~$h$, and is the
  reduction modulo $\m_\cO$ of the Breuil--Kisin module~$\gM_{\cO}$
  with~$\cO$-coefficients associated to
  $\rho|_{G_{K_\infty}}$.
\end{thm}
\begin{proof}
  Since crystalline representations are in particular semistable, it
  is immediate from~\cite[Prop.\ 3.1.4, Lem.\ 3.2.2]{MR3079258} that
  there is a not-necessarily crystalline Breuil module~$\cM$ with
  $\F$-coefficients such that $\rhobar\cong T(\cM)$, whose underlying
  Breuil--Kisin module has height at most~$h$ (note that
  our~$h$ is the integer~$r$ in the statement of \cite[Prop.\ 3.1.4]{MR3079258}). We claim that
  the Breuil module provided by these results is necessarily
  crystalline. To see this, note first that (in the case at hand, with no
  descent data) \cite[Prop.\ 3.1.4]{MR3079258} is a trivial
  consequence of the main result~\cite[Thm.\ 2.3.5]{MR2388556} of
  Liu's paper~\cite{MR2388556}, and gives an equivalence of categories
  between $G_K$-stable $\cO$-lattices inside semistable
  $E$-representations of~$G_K$ with Hodge--Tate weights in the range
  $[0,p-2]$ and strongly divisible modules with~$\cO$-coefficients. In
  particular, there is a strongly divisible module
  with~$\cO$-coefficients $\widehat{\cM}$ corresponding to~$\rho$. 

  We do not recall the notion of a strongly divisible module here, but
  we note that they are by definition modules over a coefficient ring
  $S_{\cO}$, equipped with a Frobenius, a filtration and a monodromy
  operator~$N$, and by \cite[Lem.\ 3.2.2]{MR3079258}, the Breuil
  module~$\cM$ is obtained from the strongly divisible
  module~$\widehat{\cM}$ by tensoring over~$S_{\cO}$
  with~$\overline{S}_{\F}$.  By the
  commutative diagram at the end of~\cite[\S 3.4]{MR2388556}, the
  strongly divisible module~$\widehat{\cM}$ has underlying
  Breuil--Kisin module~$\gM_{\cO}$ (via the fully faithful functor
  of~\cite[Cor.\ 3.3.2]{MR2388556}). It follows immediately that the
  underlying Breuil--Kisin module of~$\cM$ is the reduction
  modulo~$\m_{\cO}$ of~$\gM_{\cO}$, as claimed.

   It remains to show that if~$\rho$ is
   crystalline, the monodromy operator~$N$ on $\cM$ vanishes mod $u$. This follows immediately from the
   compatibility between ~$\widehat{\cM}$ and the weakly
   admissible module~$D$ associated to~$\rho$, for which see
   ~\cite[\S 3.2]{MR2388556}. %
\end{proof}

\subsection{Extensions of rank one Breuil modules}\label{subsec:
  extensions rank one Breuil modules}
In this section we make a computation of the possible extensions of
rank one Breuil modules, and prove the crucial Lemma~\ref{lem:criteria for missing extensions}, which gives a constraint on the Breuil modules which can
  witness sufficiently generic extensions of characters.  A key input to the proof of this Lemma is
  Lemma \ref{lem: Breuil module extensions}, which constrains the shapes of
  extensions of rank one Breuil modules. To prove Lemma~\ref{lem:criteria for
    missing extensions}, we simply write down an explicit extension of
  characters (after restriction to $G_{K_{\infty}}$) and show that it cannot
  arise from a Breuil module satisfying these constraints. These
  calculations are elementary, but are complicated in the case of a
  general field~$K/\Qp$, and the reader may find it helpful to firstly
  work through the case that~$K/\Qp$ is totally ramified, where the
  calculations simplify dramatically; if furthermore~$n=2$, then the monodromy
  condition is automatic and the calculations simplify further to a basic
  exercise with Breuil--Kisin modules.

Let $\sigmabar_0:k\into\F$ be a fixed embedding. Inductively define
$\sigmabar_1,\dots,\sigmabar_{f-1}$ by
$\sigmabar_{i+1}=\sigmabar_i\circ\varphi^{-1}$, where~$\varphi$ is the
arithmetic Frobenius on~$k$; we will often consider the
numbering to be cyclic, so that $\sigmabar_f=\sigmabar_0$.
There are idempotents $\epsilon_i\in
k\otimes_{\Fp}\F$ such that if $M$ is any $k\otimes_{\Fp}\F$-module,
then $M=\bigoplus_i M_i$, where $M_i := \epsilon_iM$ is the subset of $M$ consisting of
elements $m$ for which $(x\otimes 1)m=(1\otimes\sigmabar_i(x))m$ for all
$x\in k$. Note that $(\varphi\otimes 1)(\epsilon_i) = \epsilon_{i+1}$ for all $i$.  

As explained above, we are free to work with coefficients in~$\Fpbar$
rather than~$\F$, and for convenience we do so throughout this
section. (To be precise, this means that we apply the definitions
above with $\overline{S}_{\F}$ replaced by
$(k\otimes_{\Fp}\Fpbar)[u]/u^{ep}$.) It will be clear to the reader
that the coefficients do not intervene in any way in the calculations,
and we could equally well work with coefficients in any finite
extension of~$\F$.
\begin{defn}%
  \label{defn:rank-one}
   Let $s_0,\ldots,s_{f-1}$ be non-negative integers, and let $a \in
   \Fpbar^{\times}$.  Let $\gM(\underline{s};a)$ be the  rank one
   Breuil--Kisin module with~$\Fpbar$ coefficients such that
   $\gM(\underline{s}; a)_i$ is generated by $e_i$ with
\[\varphi(e_{i-1}) = (a)_{i} u^{s_{i}} e_{i}.\]
Here and below, $(a)_0 = a$  and $(a)_i = 1$ if ~$i\ne 0$.
\end{defn}

By~\cite[Lem.\ 6.2]{MR3164985}, any rank one Breuil--Kisin module %
is isomorphic to (exactly) one of the form
$\gM(\underline{s}; a)$.
\begin{defn}
  \label{defn:alpha}
 Set $\alpha_i(\gM(\underline{s};a)) := \frac{1}{p^f-1}
 \sum_{j=1}^{f} p^{f-j} s_{j+i}$.
\end{defn}

 By~\cite[Lem.\ 5.1.2]{MR3324938},  there exists a nonzero
  map $\gM(\underline{s};a) \to \gM(\underline{t};b)$ if and only if $\alpha_i(\gM(\underline{s};a)) -
  \alpha_i(\gM(\underline{t};b)) \in \Z_{\ge 0}$ for all $i$, and
  $a=b$. We now show that each rank one Breuil--Kisin module of height at
most~$(p-2)$ %
underlies  a unique rank one (crystalline) Breuil module. (In particular, all rank one
  Breuil modules are crystalline.) %

\begin{lem}\label{lem: classification of rank one Breuil modules}
  If each~$s_i\in[0,e(p-2)]$ then the rank one Breuil--Kisin module
  $\gM(\underline{s}; a)$ underlies a unique height~$(p-2)$
  Breuil module $\cM=\cM(\underline{s}; a)$ with  %
   \[\cM_j^{p-2}=\langle u^{e(p-2)-s_j}(1\otimes e_{j-1})\rangle,\] \[\varphi(
    u^{e(p-2)-s_j}(1\otimes e_{j-1})) = (a)_{j} (1\otimes
    e_{j}),\] \[N((1\otimes e_{j-1}))=0.\]  %
\end{lem}%
\begin{proof}
We begin by noting that since  
 $(\varphi^*\gM(\underline{s}; a))_i$ is generated
by~$(1\otimes e_{i-1})$, the quasi-Breuil module~$\cM=\varphi^*\gM(\underline{s}; a)/u^{ep}$
corresponding to $\gM(\underline{s}; a)$ has the given form. %
  It is easy to see that
  taking ~$N((1\otimes e_{j-1}))=0$ gives~$\cM$
  the structure of a Breuil module. To see that this is the only
  possibility, write   $N((1\otimes e_{j-1}))=\nu_j(1\otimes e_{j-1})$. Then
  we have \[N(u^{e(p-2)-s_j}(1\otimes e_{j-1}))=u^{e(p-2)-s_j}\left(s_j-e(p-2)+\nu_j\right)(1\otimes e_{j-1})\in \cM_j^{p-2}, \]
  so that
  \[\varphi(u^eN(u^{e(p-2)-s_j}(1\otimes e_{j-1})))\in u^{ep}\varphi(\cM_j)=0,\] and the equation
  $\varphi(u^eN(u^{e(p-2)-s_j}(1\otimes e_{j-1})))=cN\varphi(u^{e(p-2)-s_j}(1\otimes e_{j-1}))$ gives
  $\nu_{j+1}=0$ for each~$j$, as required.
\end{proof}

The extensions of rank one Breuil--Kisin modules are computed as
follows. 

\begin{prop}
  \label{prop:kisin-module-extensions}
  Let $\gM$ be an extension of $\gM(\underline{s};a)$ by
  $\gM(\underline{t};b)$.  Then we can choose bases $e_i,f_i$ of
  the $\gM_i$ so that $\varphi$ has the form \numeqnarray\label{eqn:
    basis for extension of BK modules}
  \varphi(e_{i-1}) & = &(b)_i u^{t_i} e_i \\
  \varphi(f_{i-1}) & = &(a)_i u^{s_i} f_i + y_i e_i \nonumber
  \end{eqnarray}
with $y_i \in \F\llb u \rrb$ a polynomial with $\deg(y_i) < s_i$,
except that when there is a nonzero map $\gM(\underline{s};a)\to \gM(\underline{t};b)$ we must
also allow $y_j$ to have a term of degree $s_j + \alpha_j(\gM(\underline{s};a))-\alpha_j(\gM(\underline{t};b))
$ for any one choice of $j$.
\end{prop}
\begin{proof}
This is~\cite[Prop.\ 5.1.3]{MR3324938}.
\end{proof}
\begin{rem}
  \label{rem: we think these are different extensions}\begin{enumerate}
  	\item In our application to `generic' $\rhobar$, we could avoid considering the special case where there is a nonzero map $\gM(\underline{s};a)\to \gM(\underline{t};b)$ (for example by ensuring that $a\ne b$), but we have included it for completeness.
  	\item While this is
  not claimed in~\cite[Prop.\ 5.1.3]{MR3324938}, we expect that it is
  possible to show that distinct choices of the~$y_i$
  in~\ref{prop:kisin-module-extensions} give distinct extensions of
  Breuil--Kisin modules.
  \end{enumerate}
\end{rem}

We now compute a constraint on extension classes of rank 1 Breuil
modules. 
\begin{lem}\label{lem: Breuil module extensions}%
  Let~$\cM$ be a crystalline Breuil module which is an extension of $\cM(\underline{s};a)$ by
$\cM(\underline{t};b)$, with underlying Breuil--Kisin
module~$\gM$ as in Proposition~\ref{prop:kisin-module-extensions}. %

For each~$i$ we set
\numequation\label{eqn: definition of ni}
  n_i:=\frac{1}{p^f-1}\sum_{j=1}^fp^{f-j}(s_{j+i-1}-t_{j+i-1}-e)\in\Q.
  \end{equation}
Then the  $y_j$ in Proposition~\ref{prop:kisin-module-extensions} cannot have any terms of degree $l< s_j-e +\max(n_{j+1},1)$ with
$l\not\equiv t_j\pmod{p}$.
\end{lem}%
\begin{proof}
The quasi-Breuil module corresponding to~$\gM$ has  $\cM_j$  generated by $(1\otimes e_{j-1}),(1\otimes f_{j-1})$ and   %
$\cM_j^{p-2}$ generated by $E_{j}:= u^{e(p-2)-t_j}(1\otimes e_{j-1})$ and \[F_{j}:= u^{e(p-2)-s_j}(1\otimes f_{j-1})-u^{e(p-2)-s_j-t_j}(b^{-1})_jy_j(1\otimes e_{j-1}).\] The map $\varphi: \cM^{p-2} \to \cM$ is given by
\[\varphi(E_j)  = (b)_{j} (1\otimes e_{j}), \varphi(F_j) = (a)_{j} (1\otimes f_{j}).\]
(Note that~$\gM$ must have height at most $(p-2)$, since it underlies
a Breuil module, so~$y_j$ is indeed divisible by $u^{s_j+t_j-e(p-2)}$.)

We have $N(1\otimes e_{j-1})=0$, and we write $N(1\otimes
f_{j-1})=\mu_j(1\otimes e_{j-1})$ with $\mu_j\in u\Fpbar[u]$ (since $\cM$ is crystalline),
where  for each~$j$ we must have %
  \[u^eN(F_j)\in
    \cM^{p-2}_j\] by the second property of $N$ required in the definition of a Breuil module. Given this, the third property of $N$ gives the commutation relation
  \begin{align*}
       \varphi (u^eN(F_j)) &=
       cN\varphi(F_j)
     \\&=cN((a)_{j} (1\otimes f_{j}))\\&=c(a)_j\mu_{j+1}(1\otimes e_{j}).
      \end{align*}
We have %
\begin{align*}
N(F_j) & = (s_j-e(p-2))u^{e(p-2)-s_j}(1\otimes f_{j-1})+u^{e(p-2)-s_j}\mu_j(1\otimes e_{j-1})\\&\quad +N(-u^{e(p-2)-s_j-t_j}(b^{-1})_jy_j)(1\otimes e_{j-1})\\
& = (s_j-e(p-2))(u^{e(p-2)-s_j}(1\otimes f_{j-1})-u^{e(p-2)-s_j-t_j}(b^{-1})_jy_j(1\otimes e_{j-1})) \\
& \quad + (u^{e(p-2)-s_j}\mu_j+N(-u^{e(p-2)-s_j-t_j}(b^{-1})_jy_j))(1\otimes e_{j-1})\\ & \quad +((s_j-e(p-2))u^{e(p-2)-s_j-t_j}(b^{-1})_jy_j)(1\otimes e_{j-1}),
\end{align*}
  so we need the quantity %
  \[
    u^e(u^{e(p-2)-s_j}\mu_j+N(-u^{e(p-2)-s_j-t_j}(b^{-1})_jy_j)+(s_j-e(p-2))u^{e(p-2)-s_j-t_j}(b^{-1})_jy_j)\]
  to
  be divisible by $u^{e(p-2)-t_j}$;
  assuming this holds, the commutation relation with~$\varphi$
  reads
  \begin{multline*}
 c(a)_{j}\mu_{j+1}= (b)_{j}u^{p(e-e(p-2)+t_j)}\varphi\biggl(u^{e(p-2)-s_j}\mu_j\\+N(-u^{e(p-2)-s_j-t_j}(b^{-1})_jy_j)+(s_j-e(p-2))u^{e(p-2)-s_j-t_j}(b^{-1})_jy_j\biggr).
  \end{multline*}
  
  In
  particular, since $c = \varphi(d) \in \im(\varphi)$, we see that $\mu_{j+1}\in\im\varphi$. Writing
  $\mu_{j+1}=\varphi(\mu'_{j+1})$ and rearranging, %
we obtain %
\numequation\label{eqn: phi commutation rank 2}
u^{e-s_j+t_j}\left((b)_{j}\varphi(\mu'_j)-u^{-t_j}(t_jy_j+N(y_j))\right)=d(a)_{j}\mu'_{j+1}. %
\end{equation}
  (Strictly speaking, since $\varphi(u^e)=u^{ep}=0$, this is
  only an equation modulo~$u^e$; but it is easily checked that all
  terms have degree less than~$e$, so it holds literally.)

Examining the left hand side of~\eqref{eqn: phi commutation
  rank 2}, we note that
there can be no cancellation between the terms in $\varphi(\mu'_j)$
and  $u^{-t_j}(t_jy_j+N(y_j))$, as the exponents of~$u$ in $\varphi(\mu'_j)$ are
all divisible by~$p$, while none of the  exponents of~$u$ in
$u^{-t_j}(t_jy_j+N(y_j))$ are divisible by~$p$ (the terms in $t_jy_j$ with exponent $\equiv  t_j \mod{p}$ cancel with terms in $N(y_j)$). Let~$d_j\ge 1$ be
the $u$-adic valuation of~$\mu'_j$  (setting  $d_j = e$
if $\mu_j'$ is divisible by~$u^e$). Then, since $d$ is a $u$-adic unit, ~\eqref{eqn: phi commutation
  rank 2} gives us the inequality \numequation\label{eqn: valuation
  inequality to get ni}pd_j-(s_j-t_j-e)\ge d_{j+1}.\end{equation}
(To see this, note that if the left hand side is at least~$e$, there is nothing
to prove; and if~$d_j=e$, then since $s_j-t_j-e\le e(p-3)$, the left hand side is at least
$3e>e$. Otherwise the term~$u^{e-s_j+t_j}(b)_{j}\varphi(\mu'_j)$ means
that the left hand side of~\eqref{eqn: phi commutation
  rank 2} has a term of degree~$pd_j-(s_j-t_j-e)$, because
of the lack of cancellation.)
Multiplying the inequalities~\eqref{eqn: valuation
  inequality to get ni}
by suitable powers of~$p$ and summing, we have \[\sum_{j=1}^fp^{f-j}(pd_{j+i-1}-d_{j+i})\ge
  \sum_{j=1}^fp^{f-j}(s_{j+i-1}-t_{j+i-1}-e), \]which simplifies
to $d_i\ge n_i$, where~$n_i$ is as in~\eqref{eqn: definition of
  ni}. Since~$\cM$ is crystalline by assumption, we also have~$d_i\ge 1$, so that $d_i\ge\max(1,n_i)$ for all~$i$.

Returning to~\eqref{eqn: phi commutation
  rank 2}, since the right hand side has valuation~$d_{j+1}$, the lack
of cancellation implies  that the term  $u^{e-s_j}(t_jy_j+N(y_j))$ on
the left hand side is
divisible by~$u^{\max(1,n_{j+1})}$;  equivalently, the terms in~$y_j$ of degree
less than  $s_j-e+\max(1,n_{j+1})$ and not congruent to~$t_j$ modulo~$p$
vanish, as claimed.
\end{proof}
\begin{rem}
  \label{rem:explicit Baer sum}It follows easily from the definitions
  that if we have two extensions of  $\cM(\underline{s};a)$ by
$\cM(\underline{t};b)$ as in Lemma~\ref{lem:
    Breuil module extensions}, then their Baer sum corresponds to the extension
  obtained by summing the~$y_j$ (and has~$N$ given by summing the~$\mu_j$).
\end{rem}

In the following arguments it will be useful to note that by the
definition of the~$n_j$ in~\eqref{eqn: definition of ni}, we have \numequation\label{eqn: relation of
  nj} n_j+(s_{j-1}-t_{j-1}-e)=pn_{j-1}\end{equation} for all~$j$.

\begin{lem}
  \label{lem: the r i} Write \numequation\label{eqn: ri}r_i:=s_i-t_i-e+\lfloor n_{i+1}\rfloor-p\lfloor n_i\rfloor
    +1.\end{equation}Then $r_i\in [1,p]$ for each~$i$.
\end{lem}
\begin{proof}
Using~\eqref{eqn: relation of
  nj}, we have $r_i-1=p(n_i-\lfloor n_i\rfloor)-(n_{i+1}-\lfloor
n_{i+1}\rfloor)$ and since for any real number $x$ we have $(x-\lfloor
x\rfloor) \in [0,1)$ we have $r_i-1\in (-1,p)$. Since~$r_i$ is an integer, the result follows.
\end{proof}

Write
\[\Ext^1_{\crBrMod}(\cM(\underline{s};a),\cM(\underline{t};b))\]
for the~$\Ext^1$ group computed in the exact category~$\crFpbarBrMod$.
Write \[\chibar_1:=T(\cM(\underline{t};b)),\]
\[\chibar_2:=T(\cM(\underline{s};a)).\] Then the restriction maps
\[\Ext^1_{\crBrMod}(\cM(\underline{s};a),\cM(\underline{t};b))\xrightarrow{\,\resK\,}
  \Ext^1_{G_K}(\chibar_2,\chibar_1)\xrightarrow{\,\resKinfty\,}\Ext^1_{G_{K_\infty}}(\chibar_2,\chibar_1)\]
are
homomorphisms of $\Fpbar$-vector spaces. Regarding
elements of $\Ext^1_{G_{K_\infty}}(\chibar_2,\chibar_1)$ as \'etale
$\varphi$-modules, we have the following description of the image of
the restriction map~$\resKinfty$. (In our key Lemma~\ref{lem:criteria for missing extensions}, we will show that the composition $\resKinfty\circ\resK$ has smaller image.)

\begin{lem}%
  \label{lem: classes that show up for G_K representations}
  \leavevmode
  \begin{enumerate}
  \item \label{resKinfty injective} The restriction map~$\resKinfty$ is injective unless
$\chibar_1\chibar_2^{-1}=\varepsilonbar$, in which case its kernel is
$1$-dimensional, and is generated by the tr\`es ramifi\'ee line given
by the Kummer extension
corresponding to the chosen uniformizer~$\pi$ of~$K$.
\item \label{item: dim image resKinfty}The image of~$\resKinfty$ has dimension $[K:\Qp]$, unless $\chibar_1=\chibar_2$, in
which case it has dimension~$[K:\Qp]+1$.
  \item The \'etale $\varphi$-modules~ $M$ in the image of~$\resKinfty$ are precisely those for which we can choose a basis $e_i,f_i$ of~ $M_i$ so that $\varphi$ has
 the form
   \numeqnarray\label{eqn: the Mi in image of GK}
    \varphi(e_{i-1}) & = &(b)_i u^{t_i} e_i \\
    \varphi(f_{i-1}) & = &(a)_i u^{s_i} f_i + y_i e_i \nonumber
  \end{eqnarray} %
where $y_i \in \Fpbar[u,u^{-1}]$ %
has nonzero terms 
only in degrees $[s_i+\lfloor
n_{i+1}\rfloor-e+1,\dots,s_i+\lfloor n_{i+1}\rfloor]$;
except that when $\chibar_1=\chibar_2$ we 
also allow $y_i$ to have a term of degree \[s_i +
  \frac{1}{p^f-1}\sum_{j=1}^fp^{f-j}(s_{j+i}-t_{j+i})\] (necessarily an integer in this case) for any one
choice of $i$.%
  \end{enumerate}
\end{lem}
\begin{proof}
The first part is~\cite[Lem.\ 5.4.2]{MR3324938}. The second part then
follows from the usual computation of the dimension of
$\Ext^1_{G_K}(\chibar_2,\chibar_1)$ via Tate's Euler characteristic
formula and local duality.
  
The final part can presumably be proved in an elementary way, but for
convenience we explain how to deduce it from the results of
\cite{MR3324938} on the Breuil--Kisin modules associated to certain
crystalline representations with small Hodge--Tate weights. This was
first explained in~\cite[Thm.\ 3.3.2]{MR3705280} in the case~$K/\Qp$
unramified (which employed the earlier results~\cite{MR3164985}), and
in general~ \cite[Thm.\ 4.2]{MR4504652}, under the assumption that
$\chibar_1\chibar_2^{-1}\ne\varepsilonbar$. However the comparison to
the notation used in~\cite[Thm.\ 4.2]{MR4504652} is not immediate, and
we need to treat the case~$\chibar_1\chibar_2^{-1}=\varepsilonbar$, so
for the convenience of the reader we explain in our notation how to
extract the claim from~\cite{MR3324938}. %

It is easy to check that the \'etale
$\varphi$-modules in~\eqref{eqn: the Mi in image of GK} span a space
of the dimension computed in part~(2). In particular, in the case when $\chibar_1 = \chibar_2$, considering the change of basis adding on a suitable multiple of $e_i$ to $f_i$ shows that the additional $\varphi$-modules in that case do not depend on the choice of $i$ for which $y_i$ is allowed to have an extra term.  It now suffices to show that all
of the possibilities in ~\eqref{eqn: the Mi in image of GK} do indeed
arise from $G_K$-representations.  We can and do twist so that that~$t_i=0$ for
all~$i$. (This has the effect of replacing~$s_i$ by~$s_i-t_i$, and
leaving~$n_i$ unchanged, so the general statement follows immediately
by twisting back.) Then our strategy is to show that our
\'{e}tale $\varphi$-modules arise from the reductions of certain crystalline
representations. In fact, we will see that they arise from the reductions of crystalline extensions of $p$-adic characters.  %

We make the change of variables
$f_i'=u^{-\lfloor n_{i+1}\rfloor}f_i$, and
write~$y_i':=u^{-p\lfloor n_{i}\rfloor}y_i$. 
Then we have
\[ \varphi(e_{i-1})  = (b)_i e_i,\]
\[\varphi(f_{i-1}')= (a)_i u^{r_i+e-1}
  f_i'+y'_ie_i.\] %
By~\eqref{eqn: ri} and our assumption that $t_i=0$, we have
\[r_i+p\lfloor n_i\rfloor=r_i+t_i+p\lfloor n_i\rfloor=s_i-e+\lfloor n_{i+1}\rfloor +1,\] so we
need to show that every choice of~$y'_i$ having nonzero terms in
degrees  $[r_i,r_i+e-1]$ occurs (together with
the additional term in the statement in the case
that~$\chibar_1=\chibar_2$). If we make a further change of
variables to replace $f'_i$ with $f'_i+z_ie_i$ for all~$i$, with
$z_i\in\Fpbar$, then we may exchange the terms in~$y_i'$ of degree
$r_i+e-1$ with terms in~$y_{i-1}$ of degree~$0$ (cf.\ \eqref{eqn:
  change of variables from BK to etale}), so it suffices in turn to
show that every choice of~$y_i'$ having nonzero terms in degrees $0$
and $[r_i,r_i+e-2]$ occurs in the image of~$\resKinfty$ (again, together with
the additional term in the statement in the case
that~$\chibar_1=\chibar_2$).

Recall~\cite[Defn.\ 2.3.1]{MR3324938} that a  pseudo-Barsotti--Tate
representation of weight~$\{r_i\}$ is a 2-dimensional crystalline
representation whose labelled Hodge--Tate weights are~$\{0,1\}$,
except at a chosen set of~$f$ embeddings lifting the embeddings
$\sigma_i:k\into\Fpbar$, where they are~$\{0,r_i\}$. By~\cite[Defn.\
4.1.3]{MR3324938}, these are the representations which have
$\sigma_{r-1,0}:=\otimes_{i}\Sym^{r_i-1}k^2\otimes_{k,\sigma_i}\Fpbar$
as a Serre weight.

Now consider~\cite[Thm.\ 5.1.5]{MR3324938}, taking the~$t_i$ there to
be zero, the~$x_i$ to be~$e-1$, and the~$s_i$ there to be
our~$r_i+e-1$ (which are not necessarily equal to our~$s_i$ -- we
apologize for this temporary notation). Note that %
with this choice,  the Breuil--Kisin modules
spanned by our basis $e_i,f'_i$ are precisely the extensions of Breuil--Kisin modules
in \cite[Thm.\ 5.1.5]{MR3324938}, for the rank one
Breuil--Kisin modules %
which are the minimal and
maximal models of $\chibar_1,\chibar_2$ as in the statement
of~\cite[Prop.\ 5.3.4]{MR3324938}.  %
So by ~\cite[Prop.\ 5.3.4, Thm.\ 5.1.5]{MR3324938}, %
if~$\psi\in\Ext^1_{G_K}(\chibar_2,\chibar_1)$ comes
from the reduction of a pseudo-Barsotti--Tate representations of weight~$\{r_i\}$, then~$\resKinfty(\psi)$
is given by an \'etale $\varphi$-module as in~\eqref{eqn: the Mi in
  image of GK}.  It therefore suffices to show that these classes
~$\resKinfty(\psi)$ %
span the
image of~$\resKinfty$.

To see this, we consider the reductions of reducible crystalline
representations. As in the proof of \cite[Thm.\ 5.4.1]{MR3324938}, we
choose crystalline characters $\chi_{1,\min}, \chi_{2,\max}$
 which lift $\chibar_1, \chibar_2$ respectively. More precisely, these
 characters  are
determined (up to unramified twist, which we do not specify) by their
Hodge--Tate weights, which (recalling that $t_i=0$ for all~$i$) we can
and do choose so that~$\chi_{2,\max}$ is unramified, and so that any
crystalline extension of~$\chi_{2,\max}$ by~$\chi_{1,\min}$ is
pseudo-Barsotti--Tate of weight~$\{r_i\}$. %

The space of crystalline extensions of~ $\chi_{2,\max}$
by~$\chi_{1,\min}$ is identified with the Galois cohomology group
$H^1_{f}(G_K,\chi_{1,\min}\chi_{2,\max}^{-1})$, and as in the proof
of~\cite[Thm.\ 5.4.1]{MR3324938}, one immediately computes that the
dimension of the image of the reduction map \numequation\label{eqn:
  redn of H1f}H^1_{f}(G_K,\chi_{1,\min}\chi_{2,\max}^{-1})\to
H^1(G_K,\chibar_1\chibar_2^{-1})\end{equation} is~$[K:\Qp]$,
unless~$\chibar_1=\chibar_2$ in which case it is~$[K:\Qp]+1$; so by
part~\eqref{item: dim image resKinfty}, this image has the same
dimension as the image of~$\resKinfty$.

In particular we see that we are done if the restriction
of~$\resKinfty$ to the image of~\eqref{eqn: redn of H1f} is
injective. If $\chibar_1\chibar_2^{-1}\ne\varepsilonbar$ then this is
automatic by part~\eqref{resKinfty injective},  so we may suppose that
$\chibar_1\chibar_2^{-1}=\varepsilonbar$.  If some~$r_i\ne p$, then
by~\cite[Lem.\ A.4]{caraiani2022local} the image
of~\eqref{eqn: redn of H1f} is contained in the peu ramifi\'ee
subspace, so %
we  again conclude by part~\eqref{resKinfty injective}. %
Finally if~$r_i=p$ for all~$i$, then as in the proof of~\cite[Thm.\
6.1.18]{MR3324938}, the union of the images of~\eqref{eqn: redn of
  H1f} as $\chi_{1,\min},\chi_{2,\max}$ range over their twists by
unramified characters with trivial reduction is all of
$H^1(G_K,\chibar_1\chibar_2^{-1})$, so we are done.%
\end{proof}

  \begin{lem}%
    \label{lem: we also win here}Suppose that
    $\sum_{j=1}^f(s_j-t_j-e)<0$. Then either there  exists an~$i$ with
    $\lfloor n_{i+1}\rfloor=-1$ and $r_i\ne p$, or there exists an~$i$ with
    $\lfloor n_{i+1}\rfloor\le -2$.
  \end{lem}
  \begin{proof}Summing~\eqref{eqn: ri} over all~$j$, we
    have \[\sum_{j=1}^f(s_j-t_j-e)=\sum_{j=1}^f\left((p-1)\lfloor
      n_{j+1}\rfloor+(r_j-1)\right). \]If this is negative, there exists
    an~$i$ with  \[(p-1)\lfloor n_{i+1}\rfloor+(r_i-1)< 0.\]
    Since $r_i\ge 1$, we must have  $\lfloor n_{i+1}\rfloor <0$. If
    $\lfloor n_{i+1}\rfloor=-1$ then we have $1-p+r_i-1<0$, so
    $r_i<p$, as required.
  \end{proof}
  
\begin{lem}%
  \label{lem:criteria for missing extensions}Suppose that
  $\sum_{j=1}^f(s_j-t_j-e)<0$. Then the restriction map \[\Ext^1_{\crBrMod}(\cM(\underline{s};a),\cM(\underline{t};b))\xrightarrow{\,\resK\,}
  \Ext^1_{G_K}(\chibar_2,\chibar_1)\] is not surjective.
\end{lem}%
\begin{proof}%
  It suffices to show that $\im(\resKinfty\circ\resK)$ is a proper
  subspace of~$\im(\resKinfty)$. %
  Viewing classes in
  $\Ext^1_{G_{K_\infty}}(\chibar_2,\chibar_1)$ as \'etale
  $\varphi$-modules, it therefore suffices to exhibit an \'etale 
  $\varphi$-module as in the statement of %
  Lemma~\ref{lem: classes that show up for G_K representations}~(3) which
  is not in the image of~$\resKinfty\circ\resK$.

By Lemma~\ref{lem: we also win here}, we may assume that for some~$i$
we either have  $\lfloor
n_{i+1}\rfloor=-1$ and $r_i\ne p$, or we have $\lfloor n_{i+1}\rfloor\le
-2$. If~$r_i\ne p$ then we set $x_i=s_i+\lfloor
  n_{i+1}\rfloor-e+1$, %
  while if  $r_i=p$ (so that $\lfloor n_{i+1}\rfloor\le
-2$) we set $x_i=s_i+\lfloor
n_{i+1}\rfloor-e+2$. %
It follows from~\eqref{eqn: ri}  that $s_i+\lfloor
  n_{i+1}\rfloor-e+1\equiv t_i+r_i\pmod{p}$, so we have
\numequation\label{eqn: xi not ti mod p}x_i\not\equiv
t_i\pmod{p}.\end{equation} %
We claim that we also have
\numequation\label{eqn: xi
  inequality for image restriction} s_i+\lfloor
  n_{i+1}\rfloor-e+1\le x_i\le s_i+\lfloor
  n_{i+1}\rfloor.\end{equation}Indeed by definition we have $x_i=s_i+\lfloor
n_{i+1}\rfloor-e+1$ or $x_i=s_i+\lfloor
n_{i+1}\rfloor-e+2$, so the lower bound is immediate, and the upper
bound is also automatic unless~$e=1$. If~$e=1$, we need to rule out
the possibility that we are in the case $x_i=s_i+\lfloor
n_{i+1}\rfloor-e+2$. In this case we assumed that $\lfloor n_{i+1}\rfloor\le
-2$, so in particular $n_{i+1}<-1$; but since we have $s_j-t_j-e\ge
-e(p-1)$ for all~$j$, it follows from~\eqref{eqn: definition of
  ni} that we have $n_j\ge -e$ for all~$j$, and in particular if
~$e=1$ we  have~$n_{i+1}\ge -1$, as required.

We also have $x_i\le s_i-e$ (because if $\lfloor n_{i+1}\rfloor=-1$
then $x_i=s_i+\lfloor n_{i+1}\rfloor-e+1=s_i-e$, and otherwise
$\lfloor n_{i+1}\rfloor\le -2$ and we have
$x_i=s_i+\lfloor n_{i+1}\rfloor-e+2\le s_i-e$), i.e.\
\numequation\label{eqn: xi inequality} x_i < s_i-e+1 = 
s_i-e+\max(n_{i+1},1).\end{equation} (We have written the inequality in this form so that we can apply Lemma~\ref{lem: Breuil module extensions}.)
Set~$y'_i=u^{x_i}$ and~$y'_j=0$ for
all~$j\ne i$. By~\eqref{eqn: xi
inequality for image restriction} and Lemma~\ref{lem: classes that show up for G_K
  representations}, it suffices to show that the  \'etale
$\varphi$-module~$M$ arising from taking the~$y_j$ in~\eqref{eqn: the Mi
  in image of GK} to be our~$y'_j$ 
is not of the form~$\gM[1/u]$ for any Breuil--Kisin module~$\gM$ satisfying the constraints of
Lemma~\ref{lem: Breuil module extensions}. 

Suppose on the contrary that~$\gM$ as in~\eqref{eqn: basis for
  extension of BK modules} has $\gM[1/u]\cong M$. This means that there is a change of variables
$e'_j=e_j$, $f_j'=f_j+\lambda_j e_j$ with $\lambda_j\in \Fpbar((u))$
having the property that for all~$j$, we have
\[\varphi(f'_{j-1})= (a)_j u^{s_j} f'_j + y_j e_j.\] Equivalently, for
each~$j$ we must have 
\numequation\label{eqn: change of variables
  from BK to etale}y_j=y'_j+ (b)_j u^{t_j}\varphi(\lambda_{j-1})
-(a)_j u^{s_j}\lambda_j.\end{equation}

Recall that we chose
$y'_i=u^{x_i}$, where~$x_i$ satisfies~\eqref{eqn: xi not ti mod p} and~\eqref{eqn: xi inequality}, so  the coefficient of~ $u^{x_i}$ in $y_i$ is
zero by Lemma~\ref{lem: Breuil module
  extensions}. The coefficient of~$u^{x_i}$
in $u^{t_i}\varphi(\lambda_{i-1})$ is also zero (again by~\eqref{eqn: xi not ti mod p}), so it follows
from~\eqref{eqn: change of variables from BK to etale} with~$j=i$ that
the coefficient of~$u^{x_i}$ in~$u^{s_i}\lambda_i$ is nonzero. Thus
$\lambda_i$ has a term of degree $x_i-s_i$. By~\eqref{eqn: xi
  inequality} we have $x_i-s_i\le -e$.

We claim that this implies that every~$\lambda_j$ has a term of degree
at most $-e$. To see this we rewrite ~\eqref{eqn: change of variables
  from BK to etale} for~$j$ replaced by~$j+1$ in the form
\numequation\label{eqn: rewritten for lambda valuation}(a)_{j+1} u^{s_{j+1}}\lambda_{j+1}=(y'_{j+1}-y_{j+1})+ (b)_{j+1} u^{t_{j+1}}\varphi(\lambda_{j}).\end{equation}
If~$j+1\ne i$ then $y'_{j+1}-y_{j+1}\in\Fpbar[[u]]$, so if $\lambda_{j}$ has a term of degree
at most $-e$, then  $u^{t_{j+1}}\varphi(\lambda_{j})$ has a term of
degree at most $t_{j+1}-ep$, which must cancel with a term in
$u^{s_{j+1}}\lambda_{j+1}$. Thus~$\lambda_{j+1}$ has a term of degree at most
$t_{j+1}-ep-s_{j+1}\le e(p-2)-ep=-2e<-e$, so the claim follows from induction
(beginning with the case~$j=i$).

Now~$v_j\le -e$ denote the $u$-adic
valuation of~$\lambda_j$. Then $u^{t_{j+1}}\varphi(\lambda_{j})$  has a nonzero term of degree
$t_{j+1}+pv_j$, which again must cancel with a term in
$u^{s_{j+1}}\lambda_{j+1}$. (Indeed, we have $t_{j+1}+pv_j\le
e(p-2)-ep<0$, and the only possible term in any $y'_{j+1}-y_{j+1}$ of negative
degree is the term~$u^{x_i}$ in~$y'_i$, which cannot cancel with a
term of degree $t_{i}+pv_{i-1}$ by~\eqref{eqn: xi not ti mod p}). We
therefore have \[v_{j+1}\le t_{j+1}-s_{j+1}+pv_j\le
  t_{j+1}+pv_j\le e(p-2)+pv_j,\] i.e.\ \numequation\label{eqn: lower
  bounds on lambda from phi relation}pv_j-v_{j+1}\ge -e(p-2).\end{equation} Summing
these inequalities multiplied by appropriate powers of~$p$, we have
 \[\sum_{j=1}^fp^{f-j}(pv_{j+i-1}-v_{j+i})\ge
  -e(p-2)(p^f-1)/(p-1), \] so that 
 $v_j\ge
-e(p-2)/(p-1)>-e$ for each~$j$. Since we already saw that $v_j\le -e$
for all~$j$, we have a contradiction, and we are done.
\end{proof}

\begin{defn}\label{defn: generic extension class}
Let $\chibar_1,\chibar_2:G_K\to\Fpbartimes$ be two characters. We say that an element of $\Ext^1_{G_K}(\chibar_2,\chibar_1)$ is generic if it is not in the image of the restriction map \[\Ext^1_{\crBrMod}(\cM(\underline{s};a),\cM(\underline{t};b))\xrightarrow{\,\resK\,}
  \Ext^1_{G_K}(\chibar_2,\chibar_1)\]
for any rank 1 Breuil modules $\cM(\underline{s};a)$ and $\cM(\underline{t};b))$ with $T(\cM(\underline{t};b))=\chibar_1$,
$T(\cM(\underline{s};a))=\chibar_2$ and $\sum_{j=1}^f(s_j-t_j-e)<0$.
\end{defn}
\begin{rem}\label{rem: non generic is finite proper union}
  Note that by Lemma \ref{lem:criteria for missing extensions}, the
  generic extensions in $\Ext^1_{G_K}(\chibar_2,\chibar_1)$ are the
  complement of the union of finitely many proper subspaces.
\end{rem}

\begin{rem}
  \label{rem: ad hoc working with Fpbar points}Definition~\ref{defn:
    generic extension class} may seem a little ad hoc, but it is
  closely related to the condition of being a generic $\Fpbar$-point
  on an irreducible component of the 2-dimensional Emerton--Gee stack
  (which we recall in Section~\ref{subsec:
  special fibre EG stacks}). %
To make this precise, we
would need to work simultaneously with arbitrary unramified twists of the
characters~$\chibar_1,\chibar_2$. While it is clear that the arguments
above are uniform across such unramified twists, and we could presumably
formulate and prove our results in the context of stacks of Breuil
modules (and Breuil--Kisin modules), there does not seem to be any
benefit in doing so. Indeed, while working with $\Fpbar$-points
occasionally leads to slightly clumsy formulations, we view it as a
feature of the structural results proved in~\cite{emertongeepicture}
(see e.g.\ Theorem~\ref{thm:Xdred is algebraic}) that we can prove
statements about families of Galois representations (e.g.\ lifting
rings) by only thinking about representations valued in~$\Fpbar$.
\end{rem}

\begin{rem}\label{rem: extension class is close to sharp}While it may
  be possible to use other integral $p$-adic Hodge theories (e.g.\
  $(\varphi,\Ghat)$-modules) to prove a version of Lemma~\ref{lem:criteria for missing extensions} which could apply to the reductions of crystalline
  representations in a greater range of Hodge--Tate weights than
  $[0,p-2]$, it is unlikely that it can be significantly improved. %
  Indeed already for $K=\Qp$, there are irreducible 2-dimensional
  crystalline representations of $G_{\Qp}$ with Hodge--Tate weights
  $0,p+2$ whose corresponding mod~$p$ Breuil--Kisin modules are of the
  form $\begin{pmatrix}bu^p&x\\0&au^2\end{pmatrix}$ where
  $a,b\in\Fpbartimes$ and $x\in\Fpbar$ are arbitrary, and consequently
  give all extensions of the corresponding characters of~$G_{\Qp}$
  when~$a\ne b$. (In addition, it is not clear to us whether the
  analogue of Lemma~\ref{lem: Breuil module extensions} holds for
  $(\varphi,\Ghat)$-modules, even in height $[0,p-2]$, although we
  have not seriously pursued this question.) %
\end{rem}

\subsection{Generic weight 0 crystalline representations}%
         \label{subsec: generic weight 0}                       %
In this subsection and the next, in order to be compatible with the notation of \cite{emertongeepicture}, we work with $d$-dimensional rather than $n$-dimensional
representations.                            
\begin{defn}\label{defn: generic rhobar}
We say that a representation $\rhobar:G_K\to\GL_d(\FF)$ is generic if it has the form
 \begin{equation*}\rhobar\cong \begin{pmatrix}
      \chibar_d &*&\dots &*\\
      0& \chibar_{d-1}&\dots &*\\
      \vdots&& \ddots &\vdots\\
      0&\dots&0& \chibar_1\\
    \end{pmatrix}  \end{equation*}
    and for $i=1,\ldots,d-1$, the off diagonal extension class in $\Ext^1_{G_K}(\chibar_i,\chibar_{i+1})$ is generic in the sense of Definition \ref{defn: generic extension class}.
\end{defn}
\begin{thm}\label{thm: rho generic ordinary}
Suppose $p>d$ and let $\rho:G_K\to\GL_d(\cO)$ be a weight 0 crystalline representation such that $\rhobar$ is generic in the sense of Definition \ref{defn: generic rhobar}.  Then
\begin{equation*}
\rho\simeq\begin{pmatrix}
    \ur_{\lambda_d} &*&\dots &*\\
      0& \ur_{\lambda_{d-1}}\varepsilon^{-1}&\dots &*\\
      \vdots&& \ddots &\vdots\\
      0&\dots&0& \ur_{\lambda_1}\varepsilon^{1-d}\\
\end{pmatrix}
\end{equation*}
for some $\lambda_1,\ldots,\lambda_d\in\cO^\times$.
\end{thm}

\begin{proof}
The proof will be by induction on $d$.  The base case $d=1$ is trivial.  For the inductive step, we claim that $\rho$ fits in an exact sequence
$$0\to\rho'\to\rho\to\ur_{\lambda_1}\varepsilon^{1-d}\to 0.$$
Admitting this for the moment, $\rho'$ is a $d-1$ dimensional crystalline representation of weight 0, and $\rhobar'$ is generic (since $\rhobar$ has a unique $d-1$ dimensional subrepresentation, which is generic). We conclude by induction on $d$.

We now prove the key claim above.  As the Hodge--Tate weights of $\rho$ are contained in the interval $[0,d-1]\subseteq[0,p-2]$, by Theorem~\ref{thm: reductions of crystalline reps are crystalline
    Breuil modules} there is a crystalline Breuil module~$\cM$ of
  rank~$d$ with $\rhobar\cong T(\cM)$, whose underlying Breuil--Kisin
  module has height at most~$(d-1)$. By Theorem~\ref{thm: properties of crystalline Breuil
    modules}~(2), the unique maximal filtration on~$\rhobar$ determines
  a filtration $0=\cM^0\subset\cM^1\subset\dots\subset\cM^d=\cM$ by
  crystalline Breuil submodules.  Write $\cM^i/\cM^{i-1}\simeq \cM(\underline{s(i)};a_i)$ in the notation of Lemma~\ref{lem: classification of rank
    one Breuil modules}.
    
    It follows from Lemma \ref{lem:criteria for missing extensions} and the definition of genericity that for each $1\leq i\leq d-1$ we have
   \[\sum_{j=1}^f(s(i+1)_j-s(i)_j-e)\ge 0.\]  Summing
these inequalities over~$i$, we obtain
  \[\sum_{j=1}^f(s(d)_j-s(1)_j)\ge ef(d-1).\] Since the underlying Breuil--Kisin
  module of~$\cM$ has height at most~$(d-1)$, we have $s(i)_j\leq e(d-1)$ for all $i,j$, and hence $s(d)_j-s(1)_j\leq e(d-1)$ for all $j$.  Since we also have the reverse inequality summed over $f$ this implies that $s(d)_j-s(1)_j=e(d-1)$ for all $j$, and hence $s(1)_j=0$ and $s(d)_j=e(d-1)$ for all $j$.
  
Now let $\gM/\gS_{\cO}$ be the Breuil--Kisin module associated to $\rho$, and $\overline{\gM}=\gM\otimes_{\cO}\FF$.  This is the Breuil--Kisin module underlying $\mathcal{M}$ by Theorem \ref{thm: reductions of crystalline reps are crystalline
    Breuil modules}. Since $s(d)_j = e(d-1)$ for all $j$, we have shown that $\overline{\gM}$ has a
  rank $1$ quotient $\overline{\gM}\to\gS_{\FF}\cdot v$ where
  $\varphi(\overline{v})=\overline{\lambda} u^{e(d-1)}\overline{v}$ for some
  $\overline{\lambda}\in (k\otimes\FF)^\times$.  It follows from
  ~\cite[Prop.\ 1.2.11]{kis04} (or rather its obvious
  generalization from height $1$ to height $(d-1$) Breuil--Kisin
  modules) %
  that this lifts to a quotient $\gM\to\gS_{\cO}\cdot v$
  where $\varphi(v)=\lambda E(u)^{d-1}v$ for some $\lambda\in
  (W(k)\otimes\cO)^\times$. Indeed, using height $(d-1)$ duality \cite[\S3.1]{Liu07}, we need to lift a rank one `multiplicative' submodule of $\overline{\gM}^*$ to ${\gM}^*$, where multiplicative means that the linearization of $\varphi$ is an isomorphism. As in~\cite[Prop.\ 1.2.11]{kis04}, we have a maximal multiplicative submodule ${\gM}^{*,m}$ of ${\gM}^*$ which lifts the maximal multiplicative submodule of $\overline{\gM}^*$ and therefore has rank at least one. Since $\rho$ is weight $0$ crystalline, its maximal unramified subrepresentation has dimension at most one. It follows that ${\gM}^{*,m}$ has rank one and is the desired lift.
  
  Finally it follows from the full
  faithfulness of the functor from lattices in crystalline
  representations to Breuil--Kisin modules (see~\cite[Prop.\
  1.3.15]{MR2263197}, or for the precise statement we are using here~\cite[Thm.\ 1.2.1]{kisin-abelian}) that there is a nonzero map
  $\rho\to\ur_{\lambda_1}\varepsilon^{1-d}$.
\end{proof}

\begin{cor}\label{cor: rhobar generic ordinary}
Let $\overline{\rho}:G_K\to\GL_d(\FF)$ be a generic representation.  Suppose that $\overline{\rho}$ has a crystalline lift of weight 0.  Then $\overline{\rho}$ has the form
\begin{equation*}\rhobar\cong\begin{pmatrix}
    \ur_{\overline{\lambda}_d} &*&\dots &*\\
      0& \ur_{\overline{\lambda}_{d-1}}\varepsilonbar^{-1}&\dots &*\\
      \vdots&& \ddots &\vdots\\
      0&\dots&0& \ur_{\overline{\lambda}_1}\varepsilonbar^{1-d}\\
    \end{pmatrix}
    \end{equation*}
    and moreover the off-diagonal extensions are peu ramifi\'ee.
    \end{cor}
    
    \begin{proof}
    The first statement is immediate from Theorem \ref{thm: rho generic ordinary}, while the claim about the extensions follows from the fact that the  reduction of a crystalline representation
    $$\begin{pmatrix}\ur_{\lambda_2}&*\\0&\ur_{\lambda_1}\varepsilon^{-1}\end{pmatrix}$$
    is peu ramifi\'ee (e.g.~by~\cite[Lem.\ A.4]{caraiani2022local}; the reduction of such a representation is finite flat, hence peu ramifi\'{e}e).
    \end{proof}

\subsection{Recollections on Emerton--Gee stacks}\label{subsec:
  special fibre EG stacks}We now recall some of the main results
of~\cite{emertongeepicture}, and prove a slight extension of them. We
use the notation of~\cite{emertongeepicture}, and in particular we continue to work with $d$-dimensional rather than $n$-dimensional
representations.

As above, we let $E/\Qp$ be a finite extension containing the Galois closure of~$K$, with ring of
integers~$\cO$, uniformizer~$\varpi$, and residue
field~$\cO/\varpi=\F$. %
The stack~$\cX_d$
over~$\Spf\cO$ is defined in~\cite[Defn.\
3.2.1]{emertongeepicture}. It is a stack of
$(\varphi,\Gamma)$-modules, but if $\F'/\F$ is a finite extension (or
if~$\F'=\Fpbar$), then the groupoid of points $x\in\cX_d(\F')$ is
canonically equivalent to the groupoid of Galois representations
$\rhobar: G_K \to \GL_d(\F')$ \cite[\S 3.6.1]{emertongeepicture}, and
we use this identification without comment below.  The stack~$\cX_d$
is a Noetherian formal algebraic stack \cite[Cor.\
5.5.18]{emertongeepicture}, and it admits closed substacks cut out by
(potentially) crystalline or semistable conditions. In particular
there is a closed substack~$\cX_d^{\crys,0}$ of~$\cX_d$ corresponding
to crystalline representations of weight~$0$, which has the
following properties.
\begin{prop}\leavevmode  \label{prop: crystalline weight 0 substack}
  \begin{enumerate}
  \item $\cX_d^{\crys,0}$ is a $p$-adic formal algebraic stack, which
    is flat over~$\Spf\cO$ and of finite type. In particular, the
    special fibre
    $\overline{\cX}_{d}^{\crys,0}:=\cX_{d}^{\crys,0}\times_{\Spf\cO}\Spec
    \F$ is an algebraic stack.
  \item  If~$\Acirc$ is a finite
    flat~$\cO$-algebra, then $\cX_{d}^{\crys,0}(\Acirc)$ \ is
    the subgroupoid of~$\cX_{d}(\Acirc)$ consisting of
    $G_K$-representations which after inverting~$p$ are crystalline of
    weight~$0$.%
     \item The special fibre $\overline{\cX}_{d}^{\crys,0}:=\cX_{d}^{\crys,0}\times_{\Spf\cO}\Spec \F$   is equidimensional of dimension~$[K:\Q_p] d(d-1)/2$. 
  \item For any finite extension $\F'$ of $\F$ and any point  $x:\Spec \F'\to\cXbar_{d}^{\crys,0}$,
  there
is a versal morphism $\Spf
R^{\crys,0,\cO'}_{\rhobar}\to\cX_{d}^{\crys,0}$ at~$x$, where
$\rhobar: G_K \to \GL_d(\F')$ is the representation corresponding to~$x$,
$\cO':=W(\F')\otimes_{W(\F)}\cO$,  and $R^{\crys,0,\cO'}_{\rhobar}$
is the weight~$0$ crystalline lifting ring. %
   \end{enumerate}

\end{prop}
\begin{proof}
We define $\cX_d^{\crys,0}$ to be the
stack~$\cX_{K,d}^{\crys,\underline{\lambda},\tau}$ of~\cite[Defn.\
4.8.8]{emertongeepicture}, taking~$\underline{\lambda}$ to be given by
~$\lambda_{\sigma,i}=d-i$ %
for all~$\sigma,i$, and~$\tau$ to be trivial. Then the first two claims are~\cite[Thm.\ 4.8.12]{emertongeepicture},
the third is~\cite[Thm.\
4.8.14]{emertongeepicture}, and the final claim is
\cite[Prop.\ 4.8.10]{emertongeepicture}.
\end{proof}

We now recall some definitions from~\cite[\S 5.5]{emertongeepicture}. By a \emph{Serre
weight} $\underline{k}$ we mean a tuple of \index{Serre weight} \index{\underline{k}}
integers~$\{k_{\sigmabar,i}\}_{\sigmabar:k\into\Fpbar,1\le i\le d}$
with the properties that \begin{itemize}
\item $p-1\ge k_{\sigmabar,i}-k_{\sigmabar,i+1}\ge 0$ for each $1\le i\le
  d-1$, and
\item $p-1\ge k_{\sigmabar,d}\ge 0$, and not every~$k_{\sigmabar,d}$
  is equal to~$p-1$.
\end{itemize}
For each ~$\sigmabar:k\into\F$,
we define the fundamental character~$\omega_{\sigmabar}$ to be
the composite
\[\omega_{\sigmabar}:I_K \xrightarrow{}%
W_K^{\operatorname{ab}} \xrightarrow{\,\Art_K^{-1}\,} \cO_K^{\times}\xrightarrow{}  k^{\times} \xrightarrow{\,\sigmabar\,}
\F^{\times}.\] %
As in~\cite[\S
5.5]{emertongeepicture}, for each Serre weight~$\underline{k}$ we
choose characters~$\omega_{\underline{k},i}:G_K\to\F^\times$
($i=1,\dots,d$) with
\[\omega_{\underline{k},i}|_{I_K}=\prod_{\sigmabar:k\into\F}\omega_{\sigmabar}^{-k_{\sigmabar,i}},\]in
such a way that if $k_{\sigmabar,i}-k_{\sigmabar,i+1}=p-1$ for
all~$\sigmabar$, then
$\omega_{\underline{k},i}=\omega_{\underline{k},i+1}$. (In~\cite[\S
5.5]{emertongeepicture} it was erroneously claimed that we could
impose further constraints on the~$\omega_{\underline{k},i}$, but as
explained in~\cite{EGaddenda}, these properties are all that we
require.) For any~$\nu\in\Fpbar$ we write
$\ur_{\nu}:G_K\to\Fpbartimes$ for the unramified character taking a
geometric Frobenius to~$\lambda$.

We say that a
representation~$\rhobar:G_K\to\GL_d(\Fpbar)$ is \emph{maximally
  nonsplit of niveau~$1$} \index{maximally
  nonsplit of niveau~$1$} if it has a unique filtration by~$G_K$-stable
$\Fpbar$-subspaces such that all of the graded pieces are one-dimensional
representations of~$G_K$.   We assign a unique Serre weight~$\underline{k}$ to
each such~$\rhobar$ in the following way:
  we say that~$\rhobar$
  is of weight~$\underline{k}$ if and only we can write %
 \numequation\label{eqn:weight k version of maximally nonsplit structure}\rhobar\cong \begin{pmatrix}
      \ur_{\nu_d}\omega_{\underline{k},d} &*&\dots &*\\
      0& \ur_{\nu_{d-1}}\varepsilonbar^{-1}\omega_{\underline{k},d-1}&\dots &*\\
      \vdots&& \ddots &\vdots\\
      0&\dots&0& \ur_{\nu_1}\varepsilonbar^{1-d}\omega_{\underline{k},1}\\
    \end{pmatrix};  \end{equation}
this uniquely determines~$\underline{k}$, except that  if
 $\omega_{\underline{k},i}=\omega_{\underline{k},i+1}$  then we need
 to say whether
   $k_{\sigmabar,i}-k_{\sigmabar,i+1}=p-1$ for all $\sigmabar$ or
   $k_{\sigmabar,i}-k_{\sigmabar,i+1}=0$ for all $\sigmabar$. We distinguish these
   possibilities as follows:  if
 $\omega_{\underline{k},i}=\omega_{\underline{k},i+1}$, then we set $k_{\sigmabar,d-i}-k_{\sigmabar,d+1-i}=p-1$ for all $\sigmabar$  if
    and only if
    $\nu_i=\nu_{i+1}$ and the element
    of~ \[\Ext^1_{G_K}(\ur_{\nu_i}\varepsilonbar^{i-d}\omega_{\underline{k},i},
  \ur_{\nu_{i+1}}\varepsilonbar^{i+1-d}\omega_{\underline{k},i+1})=H^1(G_K,\varepsilonbar)\]
    determined by~$\rhobar$ is
    tr\`es ramifi\'ee. %

    Let $(\Gm)^d_{\underline{k}}$ denote the closed subgroup scheme of
    $(\Gm)^d$ parameterizing tuples~$(x_1,\dots,x_d)$ for
    which~$x_i=x_{i+1}$
    whenever~$k_{\sigmabar,i}-k_{\sigmabar,i+1}=p-1$ for
    all~$\sigmabar$. By the definition that we just made, if~$\rhobar$
    is maximally nonsplit of niveau~$1$ and weight~$\underline{k}$,
    then the
    tuple~$(\nu_1,\dots,\nu_d)$ is an  $\Fpbar$-point of
    $(\Gm)^d_{\underline{k}}$ (where  the $\nu_i$ are as in
    in~\eqref{eqn:weight k version of maximally nonsplit structure}).

We have the following slight variant on~\cite[Thm.\ 5.5.12]{emertongeepicture}.
\begin{thm}
	\label{thm:Xdred is algebraic}\leavevmode
        \begin{enumerate}
        \item 	The Ind-algebraic stack $\cX_{d,\red}$
is an algebraic stack, of finite presentation over $\F$.
\item  $\cX_{d,\red}$ is
        equidimensional of dimension
	$[K:\Q_p] d(d-1)/2$. 
      \item The irreducible components of $\cX_{d,\red}$
        are indexed by the Serre weights~$\underline{k}$. More
        precisely, for each~$\underline{k}$ there is an irreducible
        component~$\cX_{d,\red}^{\underline{k}}$ containing a dense open
  substack~$\cU^{\underline{k}}$, all of whose $\Fpbar$-points are 
  maximally nonsplit of niveau one and
  weight~$\underline{k}$; and the~$\cX_{d,\red}^{\underline{k}}$
  exhaust the irreducible components of~$\cX_{d,\red}$.
  \item\label{item: we can get generic Ext classes in rhobar}
     There is an open subscheme~
  $T$ of $(\Gm)^d_{\underline{k}}$ such that for
  all~$(t_1,\dots,t_d)\in T(\Fpbar)$, there is an $\Fpbar$-point
  of $\cU^{\underline{k}}$ corresponding to a
  representation~\eqref{eqn:weight k version of maximally nonsplit
    structure} with $\nu_i=t_i$ for all~$i$, and which is generic in the sense of Definition~\ref{defn: generic rhobar}.
  \end{enumerate}
\end{thm}
\begin{proof}Everything except for part~\eqref{item: we can get
    generic Ext classes in rhobar} is part of~\cite[Thm.\ 5.5.12,
  Thm.\ 6.5.1]{emertongeepicture}. 
Part~\eqref{item: we can get generic Ext
    classes in rhobar} follows from the version of \cite[Thm.\
  5.5.12]{emertongeepicture} proved in~\cite{EGaddenda}, as we explain
  below.
  
  We begin by taking $T$ to be an open contained in the image of the eigenvalue morphism $\cU^{\underline{k}} \to (\Gm)^d_{\underline{k}}$, and then further shrink it so that for any~$m<n$ and $(t_1,\dots,t_d)\in
  T(\Fpbar)$ either:
  \begin{itemize}
  \item we have $k_{\sigmabar,i}-k_{\sigmabar,i+1}=p-1$ for
  all~$\sigmabar$ and all~$m\le i<n$, or 
  \item we have $(\ur_{t_n}\varepsilonbar^{n-d}\omega_{\underline{k},n})/(\ur_{t_m}\varepsilonbar^{m-d}\omega_{\underline{k},m})\not=\varepsilonbar$.
  \end{itemize}

  We then fix~$(t_1,\dots,t_d)\in
  T(\Fpbar)$, and regard each $\Ext^1_{G_K}(\ur_{t_i}\varepsilonbar^{i-d}\omega_{\underline{k},i},
     \ur_{t_{i+1}}\varepsilonbar^{i+1-d}\omega_{\underline{k},i+1})$
     as an affine space over~$\Fpbar$, and as
     in~\cite{EGaddenda} we define
\[\Ext^1_{(t_1,\dots,t_d),\underline{k}}\subseteq\prod_{i=1}^{d-1}\Ext^1_{G_K}(\ur_{t_i}\varepsilonbar^{i-d}\omega_{\underline{k},i},
     \ur_{t_{i+1}}\varepsilonbar^{i+1-d}\omega_{\underline{k},i+1})\]
   to be the closed subvariety of tuples of extension classes $(\psi_1,\dots,\psi_{d-1})$ determined by the condition that for
   each~$i=1,\dots,d-2$, the cup product $\psi_i\cup\psi_{i+1}$ vanishes.

The version of \cite[Thm.\ 5.5.12]{emertongeepicture}
  proved in~\cite{EGaddenda} states in particular that for a dense %
  Zariski open subset $U$ of $\Ext^1_{(t_1,\dots,t_d),\underline{k}}$, the
  corresponding extension classes are realized by
  some~$\rhobar\in\cU^{\underline{k}}(\Fpbar)$; so it suffices to show that $U$ contains a point
  $(\psi_1,\dots,\psi_{d-1})$ with each~$\psi_i$ generic. As the
  locus of generic classes in $\Ext^1_{(t_1,\dots,t_d),\underline{k}}$
  is open, and $U$ is dense, it suffices in turn to exhibit a single generic
  class in $\Ext^1_{(t_1,\dots,t_d),\underline{k}}$. %

          To do this, first note that
          $\psi_i\cup\psi_{i+1}$ is an element of
          the~$\Ext^2$ group
     \[\Ext^2_{G_K}(\ur_{t_{i}}\varepsilonbar^{i-d}\omega_{\underline{k},i},
       \ur_{t_{i+2}}\varepsilonbar^{i+2-d}\omega_{\underline{k},i+2}).\]
     This group vanishes unless
     $(\ur_{t_{i+2}}\varepsilonbar^{i+2-d}\omega_{\underline{k},i+2})/(\ur_{t_i}\varepsilonbar^{i-d}\omega_{\underline{k},i})=\varepsilonbar$,
     which by our choice of $T$ can only occur when
     $k_{\sigmabar,i}-k_{\sigmabar,i+1}=k_{\sigmabar,i+1}-k_{\sigmabar,i+2}=p-1$
     for all $\sigmabar$ and $\varepsilonbar=1$.  Thus if
     $\varepsilonbar\not=1$, we can just choose each~ $\psi_i$ to be any generic extension class, and the cup product condition is automatically satisfied.

     We assume from now on that $\varepsilonbar=1$ and fix a maximal interval $m<n$ such that $k_{\sigmabar,i}-k_{\sigmabar,i+1}=p-1$ for all~$\sigmabar$ and all~$m\le i< n$.  The
     characters~$\ur_{t_i}\varepsilonbar^{i-d}\omega_{\underline{k},i}$
     for $m\le i\leq n$ are all equal, and we write~$\chibar$ for their
     common value. 
     The cup product pairing is a perfect %
          pairing \[\Ext^1_{G_K}(\chibar,\chibar)\times
            \Ext^1_{G_K}(\chibar,\chibar)\to
            \Ext^2_{G_K}(\chibar,\chibar)=H^2(G_K,\varepsilonbar)\simeq
            \FF.\] The generic classes are the complement of the union of finitely many proper subspaces $L_j \subset \Ext^1_{G_K}(\chibar,\chibar)$, with annihilators $L_j^\perp$ under the pairing. Pick a generic class $\psi_m$ which is not in any $L_j^\perp$. Then the annihilator $\langle \psi_m \rangle^\perp$ cannot be contained in $\bigcup_j L_j \cup \bigcup_j L_j^\perp$ (otherwise $\langle \psi_m \rangle^\perp$ is contained in one of the $L_j$ or $L_j^\perp$ which implies that $L_j$ or $L_j^\perp$ is contained in $\langle \psi_m \rangle$). So we can find a generic class $\psi_{m+1} \in \langle \psi_m \rangle^\perp$ which is also not in any $L_j^\perp$. Repeating, we can find a sequence $\psi_m,\psi_{m+1},\ldots \psi_{n-1}$ of generic classes such that $\psi_i\cup\psi_{i+1} = 0$ for $m\le i < n$.     
   \end{proof}

\subsection{Generic reducedness}\label{subsec: generic reducedness}
We now compute the underlying cycle of the weight~$0$ crystalline
stack, and deduce our main result on generic reducedness
(Theorem~\ref{thm: special fibre weight 0 crystalline def ring
  generically reduced}).

This underlying cycle is defined as follows. By Theorem~\ref{thm:Xdred
  is algebraic}, the special fibre $\overline{\cX}_{d}^{\crys,0}$  is
a closed substack of the
special fibre~$\cXbar_d$,  and its irreducible components (with the
induced reduced substack structure) are therefore closed substacks of
the algebraic stack~$\cXbar_{d,\red}$ (see~\cite[\href{https://stacks.math.columbia.edu/tag/0DR4}{Tag 0DR4}]{stacks-project}
for the theory of irreducible components of algebraic stacks and their
multiplicities). Furthermore, $\overline{\cX}_{d}^{\crys,0}$
and~$\cXbar_{d,\red}$ are both
 algebraic stacks over~$\F$ which
are equidimensional of dimension~$[K:\Qp]d(d-1)/2$. It follows that
the irreducible components of~  $\overline{\cX}_{d}^{\crys,0}$
 are irreducible components
of~$\cXbar_{d,\red}$, and are therefore of the
form~$\cXbar_{d,\red}^{\underline{k}}$ for some Serre
weight~$\underline{k}$.

For each~$\underline{k}$, we
write~$\mu_{\underline{k}}(\cXbar^{\crys,0}_d)$ for the
multiplicity of~$\cXbar_{d,\red}^{\underline{k}}$ as a component
of~$\cXbar^{\crys,0}_d$.  We
write~$Z_{\crys,0}=Z(\cXbar^{\crys,0}_d)$ for
the corresponding cycle, i.e.\ for the formal sum
\numequation\label{eqn: cris HS multiplicity
  stack}Z_{\crys,0}=\sum_{\underline{k}}\mu_{\underline{k}}(\cXbar^{\crys,0}_d)\cdot\cXbar_{d,\red}^{\underline{k}}, \end{equation}
which we regard as an element of the finitely generated free abelian
group~$\Z[\cX_{d,\red}]$ whose generators are the irreducible
components~$\cXbar_{d,\red}^{\underline{k}}$.

\begin{thm}%
  \label{thm: BM cycle weight 0}Suppose that $p>d$. Then we have an
  equality of
  cycles \[Z_{\crys,0}=\cXbar_{d,\red}^{\underline{0}},\] where
  $\underline{0}$ is the Serre weight $\underline{k}$ with
  $k_{\sigma,i}=0$ for all $\sigma,i$.
  In particular, the  special fibre $\overline{\cX}_{d}^{\crys,0}$ is
  generically reduced. 
\end{thm}
\begin{proof}
  Suppose (in the notation of Theorem~\ref{thm:Xdred is
    algebraic}~(3)) that~$\cXbar_{d,\red}^{\underline{k}}$ is an
  irreducible component of~$\cX_{d,\red}$ contained in the
  $\overline{\cX}_{d}^{\crys,0}$. We begin by showing
  that~$\underline{k}=\underline{0}$.  By Theorem \ref{thm:Xdred is algebraic} (\ref{item: we can get generic Ext classes in rhobar}) after possibly enlarging $\F$, we can pick a point $x:\Spec\F\to\cU^{\underline{k}}$ so that the corresponding representation $\rhobar:G_K\to\GL_d(\F)$ is generic in the sense of Definition \ref{defn: generic rhobar} (it is also maximally nonsplit of niveau one and weight $\underline{k}$, since it comes from a point of $\cU^{\underline{k}}$).  Since~$x$ is in~$\overline{\cX}_{d}^{\crys,0}$, $\rhobar$ has a crystalline lift of weight 0.  We can now apply Corollary \ref{cor: rhobar generic ordinary} to conclude that $\underline{k}=0$.

  We have now shown that the support of $Z_{\crys,0}$ is indeed
  $\cXbar_{d,\red}^{\underline{0}}$, i.e.\ that the underlying reduced
  substack of $\overline{\cX}_{d}^{\crys,0}$ is equal to
  $\cXbar_{d,\red}^{\underline{0}}$, and it remains to determine the generic multiplicity. To do this, we modify our choice of
  point~$x$ as follows: by definition, we have
  $(\Gm)^d_{\underline{0}}=(\Gm)^d$, so we can and do choose our
  point~$x$ such that if $i\not=j$, then
  \numequation\label{eqn: generic eigenvalues for ordinary smooth}(\ur_{\nu_i}\varepsilonbar^{i-d})/(\ur_{\nu_j}\varepsilonbar^{j-d})\not=1,\varepsilonbar.\end{equation}%
  We will show that $\overline{\cX}_{d}^{\crys,0}$ is
  reduced in some open neighbourhood of~$x$. Since the reduced locus
  is open, and $\overline{\cX}_{d}^{\crys,0}$ is irreducible, this
  implies that $\overline{\cX}_{d}^{\crys,0}$ is generically reduced.

  We claim that the crystalline lifting ring
  $R^{\crys,0,\cO}_{\rhobar}$ is formally smooth, where $\rhobar$ corresponds to our chosen point $x$. Indeed, by Theorem
  \ref{thm: rho generic ordinary}, crystalline lifts of~$\rhobar$ of
  weight~$0$ are ordinary, and so $R^{\crys,0,\cO}_{\rhobar}$ is the
  weight 0 ordinary lifting ring of $\rhobar$. Since~$\rhobar$ is
  maximally nonsplit (i.e.\ has a unique filtration with rank~$1$
  graded pieces) and satisfies~\eqref{eqn: generic eigenvalues for
    ordinary smooth}, the deformation problem represented
  by~$R^{\crys,0,\cO}_{\rhobar}$ coincides with the one considered
  in~\cite[2.4.2]{cht} (taking~$F_{\tilde{v}}$ there to be our~$K$, $n$ to
  be our~$d$, and~$\chi_{v,i}$ to be~$\varepsilon^{-i}$), and the formal
  smoothness is~\cite[Lem.\ 2.4.7]{cht}.

  By Theorem~\ref{prop: crystalline weight 0
    substack} we have a versal morphism
  $\Spec R^{\crys,0,\cO}_{\rhobar}/\varpi\to \overline{\cX}_{d}^{\crys,0}$ at~$x$, where
  $R^{\crys,0,\cO}_{\rhobar}/\varpi$ is formally smooth and in
  particular reduced. By
  ~\cite[\href{https://stacks.math.columbia.edu/tag/0DR0}{Tag
    0DR0}]{stacks-project} we may find a smooth morphism $V\to \overline{\cX}_{d}^{\crys,0}$
  with source a finite type $\cO/\varpi$-scheme, and a point $v\in V$
  with residue field~$\F$, %
   such that there is an isomorphism
  $\widehat{\cO}_{V,v}\cong R^{\crys,0,\cO}_{\rhobar}/\varpi$,
  compatible with the given morphism to~$\overline{\cX}_{d}^{\crys,0}$. By
  \cite[\href{https://stacks.math.columbia.edu/tag/00MC}{Tag
    00MC}]{stacks-project} and
  \cite[\href{https://stacks.math.columbia.edu/tag/033F}{Tag
    033F}]{stacks-project}, the local ring ${\cO}_{V,v}$ is
  reduced. Since being reduced is an open condition, we see that~$V$
  is reduced in an open neighbourhood of~$v$; and since it is also a
  smooth local condition
  (see~\cite[\href{https://stacks.math.columbia.edu/tag/04YH}{Tag
    04YH}]{stacks-project}) it follows
  that~$\overline{\cX}_{d}^{\crys,0}$ is reduced in an open
  neighbourhood of~$x$, and we are done.
 \end{proof}

\begin{rem}
  \label{rem: this is a BM cycle}%
  Since the algebraic representation of~$\GL_d$ of highest
  weight~$0$ is the trivial representation, Theorem~\ref{thm: BM cycle
    weight 0} shows that if~$p>d$, the cycle~$Z_{\underline{0}}$ in
  the geometric Breuil--M\'ezard conjecture \cite[Conj.\
  8.2.2]{emertongeepicture} is necessarily equal to
  $\cXbar_{d,\red}^{\underline{0}}$. As far as we are aware, this is
  the only instance in which such a cycle has been computed for~$d>2$
  and~$K/\Qp$ arbitrary.
\end{rem}

\begin{thm}%
  \label{thm: special fibre weight 0 crystalline def ring generically
    reduced}Suppose that $p>d$, that $K/\Qp$ is a finite extension,
  and that $E/\Qp$ is a finite extension containing the Galois closure
  of~$K$, with ring of integers~$\cO$ and residue field~$\F$.

  Then for any $\rhobar: G_K \to \GL_d(\F)$, the special fibre $
  R^{\crys,0,\cO}_{\rhobar}/\varpi $ of the weight~
  $0$ crystalline lifting ring is generically reduced.
\end{thm}
\begin{proof} We follow the proof of~\cite[Thm.\
  4.6]{caraiani2022geometric}. By Proposition~\ref{prop: crystalline
    weight 0 substack}, we have a versal morphism
  $\Spf R^{\crys,0,\cO}_{\rhobar}/\varpi
  \to\cXbar_d^{\crys,0}$ at the $\F$-point of~$\cX_{d,\red}$
  corresponding to
  ~$\rhobar$. By~\cite[\href{https://stacks.math.columbia.edu/tag/0DR0}{Tag
    0DR0}]{stacks-project} we may find a smooth morphism
  $V\to \cXbar_d^{\crys,0}$ with source a finite type $\F$-scheme, and
  a point $v\in V$ with residue field~$\F$, such that there is an
  isomorphism $\widehat{\cO}_{V,v}\cong R^{\crys,0,\cO}_{\rhobar}/\varpi$,
  compatible with the given morphism to~$\cXbar_d^{\crys,0}$.

  By Theorem~\ref{thm: BM cycle weight 0}, there is a dense open
  substack~$\cU$ of $\cXbar_d^{\crys,0}$ such that~$\cU$ is
  reduced. Since being reduced is a smooth local property, the
  pullback of~$\cU$ to~$V$ is a reduced open subscheme of~$V$; and
  this pullback is furthermore dense in~$V$, because the formation
  of the scheme-theoretic image of~$\cU$ in ~ $\cXbar_d^{\crys,0}$
  commutes with flat base change
  \cite[\href{https://stacks.math.columbia.edu/tag/0CMK}{Tag
    0CMK}]{stacks-project}.
Thus~$V$ is generically reduced, and the
  complete local rings of~$V$ at finite type points are generically
  reduced by ~\cite[Lem.\ 4.5]{caraiani2022geometric}. In particular
  $ R^{\crys,0,\cO}_{\rhobar}/\varpi\cong \widehat{\cO}_{V,v}$ is
  generically reduced, as required.
\end{proof}

\begin{rem}
  \label{rem: comparison to CEGS}The case~$d=2$ of Theorem~\ref{thm:
    special fibre weight 0 crystalline def ring generically reduced}
  is a special case of~\cite[Thm.\ 4.6]{caraiani2022geometric}. In
  both cases the statement is deduced from the corresponding statement
  for the stack~$\cXbar_d^{\crys,0}$, and indeed in the case~$d=2$,
  Theorem~\ref{thm: BM cycle weight 0} is a special case
  of~\cite[Thm.\ 7.1, 7.6]{caraiani2022geometric} (although the
  generic reducedness statement is proved earlier in~\cite[Prop.\
  4.1]{caraiani2022geometric}). 

  The argument that we use to prove Theorem~\ref{thm: BM cycle weight
    0} is necessarily rather different from the proof of~\cite[Thm.\
  4.6]{caraiani2022geometric}, which was written
  before~\cite{emertongeepicture}, and in particular could not use the
  structure of generic points on the irreducible components
  of~$\cX_{2,\red}$. Instead, the proof
  in~\cite{caraiani2022geometric} uses the Kisin resolution of
  ~$\cX_2^{\crys,0}$ (originally defined for lifting rings
  in~\cite{kis04}). By results on local models for Shimura varieties,
  this Kisin resolution has reduced special fibre, and the arguments
  in~\cite{caraiani2022geometric} show that the map from the Kisin
  resolution is an isomorphism on dense open substacks of the source
  and target. In dimension greater than~$2$ we do not know of a
  candidate Kisin resolution for which we could expect to argue in
  this way.

  The result~\cite[Thm.\ 4.6]{caraiani2022geometric} is more general
  than Theorem~\ref{thm: special fibre weight 0 crystalline def ring
    generically reduced} (as always, in the special case $d=2$),
  because it also proves the analogous statement for the potentially
  crystalline lifting ring of weight~$0$ and any tame type. The
  Breuil--M\'ezard conjecture implies that the analogous statement
  necessarily fails for $d\ge 4$ (even if~$K=\Qp$), because the
  reductions modulo~$p$ of the corresponding inertial types contain
  Serre weights with multiplicities greater than~$1$, even for generic
  choices of type (see~\cite[Rem.\ 8.1.4]{MR4549091}). Similarly,
  Theorem~\ref{thm: BM cycle weight 0} is best possible in the sense
  that for any parallel Serre weight~$\underline{k}$ greater than~$0$
  (i.e.\ $k_{\sigma,i}$ is independent of~$\sigma$, and
  the~$k_{\sigma,i}$ are not all equal), the
  stack~$\cX_d^{\crys,\underline{k}}$ cannot have generically reduced
  special fibre once~$K$ is sufficiently ramified. (While the
  Breuil--M\'ezard conjecture is not known, standard arguments with
  Taylor--Wiles patching give the expected lower bounds for
  Breuil--M\'ezard multiplicities, so it is presumably possible to
  prove unconditionally that the special fibres of the corresponding
  stacks are not generically reduced.)  \end{rem}

\begin{rem}\label{rem: generic reduced might well be fine for small p
    but we can't prove it}
  Despite
  Remark~\ref{rem: extension class is close to sharp}, it seems
  plausible to us that Theorem~\ref{thm: BM cycle weight 0} should
  also hold if~$p\le d$, but any proof will necessarily be more
  complicated, and presumably cannot rely only on an analysis of the
  successive extension classes of characters of the kind that we have
  made here. %
\end{rem}

\section{An automorphy lifting theorem in weight~$0$}

\subsection{Preliminaries}
Our goal in this section is to state and prove
Theorem~\ref{thm:main_automorphy_lifting_theorem}, which is an automorphy lifting
theorem for $n$-dimensional crystalline weight~ $0$ $p$-adic
representations of~$G_F$, where~$F$ is an imaginary CM
field in which~$p$ is arbitrarily ramified. The key innovations
that allow us to prove this theorem are the local-global compatibility
result of~\cite{caraiani-newton} and the generic reducedness result
that we proved in Theorem~\ref{thm: special fibre weight 0 crystalline def ring generically
    reduced}. Given these ingredients, the proof is very close to
  those of~ \cite[Theorem
6.1.1]{10author} and \cite[Theorem 1.2]{miagkov-thorne}, and we
refer to those papers for some of the details of the arguments, and
for any unfamiliar terminology.

We begin by introducing some terminology and notation we will need for the statement and proof. 

\subsubsection{Galois preliminaries}
Fix a continuous irreducible representation $\rhobar: G_F \to \GL_n(\Fpbar)$ for a
number field $F$. We fix a coefficient field $E/\Qp$ such that
$\rhobar(G_F) \subset \GL_n(\FF)$. %

We will use the notion of a \emph{decomposed generic} representation $\rhobar$, defined in \cite[Definition 4.3.1]{10author}. We will also use the notion of an \emph{adequate subgroup} of $\GL_n(\FF)$, see for example \cite[Definition 1.1.1]{miagkov-thorne}.

Let $v$ be a finite place of $F$. As in \cite[\S6.2.1]{10author}, a
\emph{local deformation problem} is a $\widehat{\PGL}_n$-stable
subfunctor of the lifting functor $\cD_v^\square:= \cD_{\rhobar|_{G_{F_v}}}^{\square,\cO}$, (pro-)representable by a quotient $R_v$ of the lifting ring $R_v^\square$. The following local deformation problems will be relevant:

\begin{itemize}
	\item the lifting functor itself, $\cD_v^\square$,
	\item for $v|p$, weight $0$ crystalline lifts $\cD_v^{\crys,\underline{0}}$, represented by $R^{\crys,\underline{0},\cO}_{\rhobar|_{G_{F_v}}}$,
	\item the local deformation problem $\cD_v^\chi$ defined in \cite[\S6.2.15]{10author}. In this case, we assume that $q_v\equiv 1 \mod p$, that $\rhobar|_{G_{F_v}}$ is trivial, that $p > n$, and we have a tuple $(\chi_{i})_{i = 1, \dots, n}$ of characters $\chi_{i} : \cO_{F_v}^\times \to \cO^\times$ which are trivial modulo $\varpi$. Then $\cD_v^\chi$ classifies lifts $\rho \colon G_{F_v} \rightarrow \GL_n(A)$ such that
	\[
	\mathrm{char}_{\rho(\sigma)}(X) = \prod_{i=1}^n (X - 
	\chi_i(\mathrm{Art}_{F_v}^{-1}(\sigma)))
	\]
	for all $\sigma \in I_{F_v}$. 
\end{itemize}

Let $S$ be a finite set of finite places of $F$ containing the $p$-adic places $S_p$ and
all places at which $\rhobar$ is ramified. Then we use the notion of a
\emph{global deformation problem} from \cite[Definition
6.2.2]{10author}. We will be able to restrict to the case where
$\Lambda_v = \cO_v$ for all $v \in S$, so our global deformation
problems will be tuples $\cS= (\rhobar,S,\{\cO\}_{v \in
  S},\{\cD_v\}_{v \in S})$. Each $\cD_v$ is a local deformation
problem, representable by a quotient $R_v$ of $R_v^\square$. There is
an associated functor $\cD_\cS$ of deformations of $\rhobar$
satisfying the local condition $\cD_v$ for each $v \in S$. It is
representable by $R_{\cS}$. More generally, if $T \subset S$, we have
a functor~$\cD_{\cS}^{T}$ of $T$-framed deformations, which is
representatble by~ $R_{\cS}^T$. The $T$-framed global deformation ring $R_{\cS}^T$ receives a natural $\cO$-algebra map from $R_{\cS}^{T,\loc} := \widehat{\otimes}_{v\in T,\cO} R_v$.

\subsubsection{Automorphic preliminaries}
Now we assume that $F$ is an imaginary CM number field. On the automorphic side, we will be interested in cuspidal automorphic representations of $\GL_n(\A_F)$ which are \emph{regular algebraic of weight $0$}. This means that the infinitesimal character of $\pi_\infty$ matches the infinitesimal character of the trivial representation of $\GL_n(F_\infty)$. These automorphic representations contribute to the cohomology groups with trivial coefficients of locally symmetric spaces.

Let $X_\infty = \GL_n(F_\infty)/\R_{> 0}K_\infty$ be the symmetric space, with $K_\infty$ a maximal compact subgroup of $\GL_n(F_\infty)$ (since $F$ is totally imaginary, $K_\infty$ is connected). Suppose we have a \emph{good subgroup} $K \subset \GL_n(\A_F^\infty)$. In other words, $K$ is neat, compact, open, and factorizes as  $K = \prod_v K_v$ for compact open subgroups $K_v \subset \GL_n(F_v)$. Then we can define a smooth manifold \[X_K = \GL_n(F) \backslash \left(X_\infty \times \GL_n(\A_F^\infty)/K \right).\]

Fix a finite set of finite places $S$ of $F$ containing $S_p$, with $K_v = \GL_n(\cO_v)$ for $v \notin S$. We factorize $K=K_S K^S$. We have an abstract Hecke algebra $\mathcal{H}(\GL_n(\A_F^{\infty,S}),K^S)$ with coefficients in $\Z$, a tensor product of spherical Hecke algebras over finite places $v \notin S$. 

Suppose that $V$ is a finite $\cO$-module with an action of $G(F)\times K_S$. Then, as explained in \cite[\S2.1.2]{10author}, $V$ descends to a local system of $\cO$-modules $\cV$ on $X_K$, and we have a natural Hecke action 
\[\mathcal{H}(\GL_n(\A_F^{\infty,S}),K^S)\otimes_{\Z}\cO \to \End_{\mathbf{D}(\cO)}(R\Gamma(X_K,\cV)).\] The image of this $\cO$-algebra map is a finite $\cO$-algebra denoted by $\TT^S(K,\cV)$. If $\m$ is a maximal ideal of $\TT^S(K,\cV)$, it has an associated semisimple Galois representation \[\rhobar_\m: G_{F,S'} \to \GL_n(k(\m))\] for a suitable set of places $S'$ containing $S$ \cite[Theorem 2.3.5]{10author}. For $v \notin S'$, the characteristic polynomial of $\rhobar_\m(\Frob_v)$ equals the image of \[\begin{split} P_v(X) = X^n&-T_{v, 1}X^{n-1} + \dots + (-1)^iq_v^{i(i-1)/2}T_{v, i}X^{n-i}+\dots \\ & + q_v^{n(n-1)/2}T_{v,n} \in \mathcal{H}(\GL_n(F_v),\GL_n(\cO_{F_v}))[X].
\end{split}\] in the residue field $k(\m)$. We write $T_{v, i} \in \mathcal{H}(\GL_n(F_v), \GL_n(\cO_{F_v}))$ for the double coset operator
\[ T_{v, i} = [ \GL_n(\cO_{F_v}) \mathrm{diag}(\varpi_v, \dots, \varpi_v, 1, \dots, 1) \GL_n(\cO_{F_v})], \]
where $\varpi_v$ appears $i$ times on the diagonal.

When $\rhobar_\m$ is absolutely irreducible, the cohomology groups $H^i(X_K,\cO)_\m \otimes_{\cO} E$ can be described in terms of cuspidal automorphic representations which are regular algebraic of weight 0 \cite[Theorem 2.4.10]{10author}.
 
\subsection{An automorphy lifting theorem}The rest of this section is
devoted to the proof of the following theorem, which is a version of \cite[Theorem 6.1.1]{10author} and \cite[Theorem 1.2]{miagkov-thorne} allowing arbitrary ramification at primes dividing $p$, at the price of restricting to weight~$0$ automorphic representations. 

\begin{theorem}\label{thm:main_automorphy_lifting_theorem}
	Let $F$ be an imaginary CM or totally real field
and let $p>n$ be a prime. Suppose given a continuous representation $\rho : G_F \to \GL_n(\Qpbar)$ satisfying the following conditions:
	\begin{enumerate}
		\item $\rho$ is unramified almost everywhere.
		\item For each place $v | p$ of $F$, the
                  representation $\rho|_{G_{F_v}}$ is crystalline of
                  weight~$0$, i.e.\ with Hodge--Tate weights $HT_{\tau}(\rho)=\{0,1,2,\ldots,n-1\}$ for each $\tau:F_v\hookrightarrow\Qpbar$. %
		\item\label{part: ALT cond dgi} $\overline{\rho}$ is absolutely irreducible and decomposed generic. The image of $\overline{\rho}|_{G_{F(\zeta_p)}}$ is adequate \emph{(}as a subgroup of $\GL_n(\FF)$, for sufficiently large $\FF$\emph{)}. 
		\item\label{pa4fl}  \label{part:scalartoremark} There exists $\sigma \in G_F - G_{F(\zeta_p)}$ such that $\overline{\rho}(\sigma)$ is a scalar. 
		\item There exists a cuspidal automorphic representation $\pi$ of $\GL_n(\A_F)$ satisfying the following conditions:
		\begin{enumerate}
			\item $\pi$ is regular algebraic of weight $0$.
			\item There exists an isomorphism $\iota : \Qpbar \to \C$ such that $\overline{\rho} \cong \overline{r_\iota(\pi)}$.
			\item\label{assm:connects at p} If $v | p$ is a place of $F$, then $\pi_v$ is unramified and $r_\iota(\pi)|_{G_{F_v}} \sim \rho|_{G_{F_v}}$ \emph{(}``connects to'', in the sense of \cite[\S1.4]{BLGGT}\emph{)}.
		\end{enumerate}
	\end{enumerate}
	Then $\rho$ is automorphic: there exists a cuspidal automorphic representation $\Pi$ of $\GL_n(\A_F)$ of weight $\lambda$ such that $\rho \cong r_\iota(\Pi)$. Moreover, if $v$ is a finite place of $F$ and either $v | p$ or both $\rho$ and $\pi$ are unramified at $v$, then $\Pi_v$ is unramified.
\end{theorem}
\begin{rem}
In assumption (\ref{assm:connects at p}), we are using \cite[Theorem
4.3.1]{caraiani-newton} which shows that $r_\iota(\pi)|_{G_{F_v}}$ is
crystalline with the same labelled Hodge--Tate weights as
$\rho|_{G_{F_v}}$. Choose a $p$-adic coefficient field $E$ which
contains the Galois closure of $F$ and such that $\rhobar(G_F)\subset
\GL_n(\FF)$. Then assumption~\eqref{assm:connects at p} is that $r_\iota(\pi)|_{G_{F_v}}$ and $\rho|_{G_{F_v}}$ define points on the same irreducible component of the weight $0$ crystalline lifting ring $R^{\crys,\underline{0},\cO}_{\rhobar|_{G_{F_v}}}\otimes_{\cO}\Qpbar$.  
\end{rem}
We begin by imposing some additional assumptions, under which we can
use the Calegari--Geraghty version of the Taylor--Wiles--Kisin
patching method to prove
an automorphy lifting theorem. We then deduce
Theorem~\ref{thm:main_automorphy_lifting_theorem} by a standard base
change argument. We refer the reader to~\cite{10author} for any
unfamiliar notation.

We let $F$ be an imaginary CM field with maximal totally real
subfield~$F^+$ and complex conjugation $c \in \Gal(F/F^+)$. We fix an integer $n \ge 1$, %
an odd prime $p > n$ and an isomorphism $\iota : \Qpbar \cong \C$. We let $\pi$ be a
cuspidal automorphic representation of $\GL_n(\A_F)$, which is regular
algebraic of weight $0$. We suppose we have a finite set $S$ of finite
places of $F$, containing the set~$S_p$ of places of~$F$ above $p$, and a (possibly empty)
subset $R \subset (S\setminus S_p)$.%

Then we assume that the following conditions are satisfied:
\begin{enumerate}
	\item If $l$ is a prime lying below an element of $S$, or which is ramified in $F$, then $F$ contains an imaginary quadratic field in which $l$ splits. In particular, each place of $S$ is split over $F^+$ and the extension $F / F^+$ is everywhere unramified. 
	\item For each $v \in S_p$, let $\overline{v}$ denote the place of $F^+$ lying below $v$. Then there exists a place $\overline{v}' \neq \overline{v}$ of $F^+$ such that $\overline{v}' | p$ and 
	\[ \sum_{\overline{v}'' \neq \overline{v}, \overline{v}'} [ F^+_{\overline{v}''} : \Qp ] > \frac{1}{2} [ F^+ : \Qp ]. \]
	\item The residual representation $\overline{r_\iota(\pi)}$ is absolutely irreducible and decomposed generic, and $\overline{r_\iota(\pi)}|_{G_{F(\zeta_p)}}$ has adequate image.
	\item If $v$ is a place of $F$ lying above $p$, then $\pi_v$ is unramified.
	\item If $v \in R$, then $\pi_v^{\Iw_v} \neq 0$, $q_v \equiv 1 \text{ mod }p$ and $\overline{r_\iota(\pi)}|_{G_{F_v}}$ is trivial.
	\item If $v \in S - (R \cup S_p)$, then $\pi_v$ is unramified, $v\notin R^c$, and $H^2(F_v, \ad \overline{r_\iota(\pi)}) = 0$. %
	\item $S-(R \cup S_p)$ contains at least two places with distinct residue characteristics.
	\item If $v \not\in S$ is a finite place of $F$, then $\pi_v$ is unramified. 
\end{enumerate}
We define an open compact subgroup $K = \prod_v K_v$ of $\GL_n(\widehat{\cO}_F)$ as follows:
\begin{itemize}
	\item If $v \not\in S$, or $v \in S_p$, then $K_v = \GL_n(\cO_{F_v})$.
	\item If $v \in R$, then $K_v = \Iw_v$.
	\item If $v \in S - (R \cup S_p)$, then $K_v = \Iw_{v,1}$ is the pro-$v$ Iwahori subroup of $\GL_n(\cO_{F_v})$.
\end{itemize} 

By \cite[Theorem 2.4.10]{10author}, we can find a coefficient field $E \subset \Qpbar$ and a maximal ideal $\ffrm \subset \TT^S(K, \cO)$ such that $\overline{\rho}_\ffrm \cong \overline{r_\iota(\pi)}$. After possibly enlarging $E$, we can and do assume that the residue field of $\m$ is equal to $\FF$, the residue field of $E$.
For each tuple $(\chi_{v, i})_{v \in R, i = 1, \dots, n}$ of characters $\chi_{v, i} : k(v)^\times \to \cO^\times$ which are trivial modulo $\varpi$, we define a global deformation problem  
\[ \cS_\chi = (\overline{\rho}_\ffrm, S, \{ \cO \}_{v \in S}, \{ \cD_v^{\crys,\underline{0}} \}_{v \in S_p} \cup \{ \cD_v^\chi \}_{v \in R} \cup \{ \cD_v^\square \}_{v \in S - (R \cup S_p)}). \]
We will assume that either $\chi_{v,i} = 1$ for all $v \in R$ and all
$1\le i\le n$, or that for each $v \in R$ the $\chi_{v,i}$ are
pairwise distinct. %

Extending $\cO$ if necessary, we may assume that all irreducible components of our local lifting rings and their special fibres are geometrically irreducible. We fix representatives $\rho_{\cS_\chi}$ of the universal deformations which are identified modulo $\varpi$ (via the identifications $R_{\cS_\chi} / \varpi \cong R_{\cS_1} / \varpi$). We define an $\cO[K_S]$-module $\cO(\chi^{-1})$, where $K_S$ acts by the composition of $\chi^{-1}$ with the projection \[K_S \to K_R = \prod_{v \in R} \Iw_v \to \prod_{v \in R} (k(v)^\times)^n.\]
\begin{prop}\label{prop:existence_of_Hecke_Galois_with_LGC}
	There exists an integer $\delta \geq 1$, depending only on $n$ and $[F : \Q]$, an ideal $J \subset \TT^S( R \Gamma(X_K, \cV_\lambda(\chi^{-1})))_\ffrm$ such that $J^\delta = 0$, and a continuous surjective homomorphism
	\[ f_{\cS_\chi} : R_{\cS_\chi}  \to \TT^S( R \Gamma(X_K, \cO(\chi^{-1})))_\ffrm / J \]
	such that for each finite place $v \not \in S$ of $F$, the characteristic polynomial of $f_{\cS_\chi} \circ \rho_{\cS_\chi}(\Frob_v)$ equals the image of $P_v(X)$ in $\TT^S( R \Gamma(X_K, \cO(\chi^{-1})))_\ffrm / J$.
\end{prop}
\begin{proof}
	This is a version of \cite[Proposition 6.5.3]{10author}, using \cite[Theorem 4.2.15]{caraiani-newton} to verify that we satisfy the crystalline weight $0$ condition at $v \in S_p$.
\end{proof}

This proposition means that it makes sense to talk about the support of  $H^\ast(X_K, \cO)_\ffrm$ over $ R_{\cS_{1}}$, since $f_{\cS_1}$ realizes $\Spec(\TT^S(K,\cO)_\m)$ as a closed subset of $\Spec(R_{\cS_1})$. 

Here are the essential properties of the (completed tensor products of) local deformation rings in our situation:
\begin{lem}\label{lem:local def ring props}Fix a tuple $\chi = (\chi_{v, i})_{v \in R, i = 1, \dots, n}$ of characters $\chi_{v, i} : k(v)^\times \to \cO^\times$ which are trivial modulo $\varpi$. We assume that either $\chi_{v,i} = 1$ for all $v \in R$ and all
	$1\le i\le n$, or that for each $v \in R$ the $\chi_{v,i}$ are
	pairwise distinct.
	\begin{enumerate}
		\item $R_{\cS_{\chi}}^{S,\loc}$ is equidimensional of dimension $1+n^2|S| + \frac{n(n-1)}{2}[F:\Q]$ and every generic point has characteristic $0$.
		\item\label{uniquegen} Each generic point of $\Spec R_{\cS_{\chi}}^{S,\loc}/\varpi$ is the specialization of a unique generic point of $\Spec R_{\cS_{\chi}}^{S,\loc}$.
		\item\label{irredcompsatp} Assume that $\chi_{v,1},\ldots,\chi_{v,n}$ are pairwise distinct for each $v\in R$. 
		Then the natural map $\Spec R_{\cS_{\chi}}^{S,\loc}\to\Spec R_{\cS_{\chi}}^{S_p,\loc} = \Spec R_{\cS_{1}}^{S_p,\loc}$ induces a bijection on irreducible components. 
		\item\label{regular generic fibre} Each characteristic zero point of $\Spec R_{\cS_{1}}^{S_p,\loc}$ lies on a unique irreducible component.
		\item Assume that $\chi_{v,1},\ldots,\chi_{v,n}$ are
                  pairwise distinct for each $v\in R$, and let $C$ be
                  an irreducible component of $\Spec
                  R_{\cS_{1}}^{S,\loc}$. Write  $C_p$ for the
                  image of~$C$ in
                   $\Spec R_{\cS_{1}}^{S_p,\loc}$ \emph{(}so that~$C_p$ is an
                   irreducible component of $\Spec
                   R_{\cS_{1}}^{S_p,\loc}$\emph{)}. Then the generic points
                   of $C\cap \Spec R_{\cS_{1}}^{S,\loc}/\varpi$
                   generalize \emph{(}via the equality $\Spec
                   R_{\cS_{1}}^{S,\loc}/\varpi= \Spec
                   R_{\cS_{\chi}}^{S,\loc}/\varpi$\emph{)} to the generic
                   point of $\Spec R_{\cS_{\chi}}^{S,\loc}$
                   corresponding  to $C_p$ via the bijection of part~\eqref{irredcompsatp}. \emph{(}By part (\ref{uniquegen}), each of these points has a unique generalization.\emph{)}
	\end{enumerate} 
\end{lem}
\begin{proof} We begin by noting that \cite[Lemma 3.3]{blght} allows
  us to describe the set of irreducible components of
  $\Spec(R_{\cS_{\chi}}^{S,\loc})$ (respectively, its special fibre)
  as the product over $v \in S$ of the sets of irreducible components
  of the local deformation rings (respectively, their special
  fibres). (Here we use that the irreducible components of the local
  deformation rings that we consider are all in characteristic zero)
	
The first part follows from \cite[Lemma 6.2.25]{10author} (we
        have a different deformation condition at $p$, but
        $R^{\crys,\underline{0},\cO}_{\rhobar|_{G_{F_v}}}$ is
        $\cO$-flat by definition and equidimensional of dimension
        $1+n^2 + \frac{n(n-1)}{2}[F_v : \Qp]$ by \cite[Theorem
        3.3.4]{kisindefrings}).

        For the second part, for each $v \in S$ and local deformation ring $R_v$ we need to check that the generic points of $\Spec R_v/\varpi$ have unique generalizations to $\Spec R_v$. For $v|p$, this follows from Theorem~\ref{thm: special fibre weight 0 crystalline def ring generically
        	reduced}%
              --- see \cite[Lemma 5.3.3]{caraiani-newton} for the
              argument that generically reduced special fibre implies
              unique generalizations of its generic points, and note that we are assuming  $p > n$. For $v \in
              R$, the property we need follows from \cite[Props.~6.2.16,
              6.2.17]{10author}. For $v \in S - (R\cup S_p)$, $R_v =
              R_v^{\square}$ is formally smooth over $\cO$.

              The third part follows from the irreducibility of the
              local deformation rings for $v \in S-S_p$ when the $\chi_{v,i}$ are pairwise distinct
              \cite[Prop.~6.2.17]{10author}, and the fourth part  from
              the regularity of
              $R_{\cS_{1}}^{S_p,\loc}\left[1/p\right]$ \cite[Theorem
              3.3.8]{kisindefrings}.

           For the final part, by the third part is enough to note
           that as we saw above, it follows from 
           Theorem~\ref{thm: special fibre weight 0 crystalline def
             ring generically reduced} that the generic points
           of the special fibre of $\Spec R_{\cS_{1}}^{S_p,\loc}/\varpi
           = \Spec R_{\cS_{\chi}}^{S_p,\loc}/\varpi$ uniquely
           generalize to generic points of~$\Spec R_{\cS_{1}}^{S_p,\loc}
           = \Spec R_{\cS_{\chi}}^{S_p,\loc}$.
\end{proof}

\begin{theorem}\label{thm:alt_after_bc}
	Suppose we are given two homomorphisms $f_1, f_2: R_{\cS_1}\to \cO$ with associated liftings $\rho_1, \rho_2: G_{F,S} \to \GL_n(\cO)$. Suppose $\ker(f_1) \in \Supp_{R_{\cS_1}}(H^\ast(X_K, \cO)_\ffrm)$ and $\rho_1|_{G_{F_v}} \sim \rho_2|_{G_{F_v}}$ for each $v \in S_p$.  Then $\ker(f_2) \in \Supp_{R_{\cS_1}}(H^\ast(X_K, \cO)_\ffrm)$.
\end{theorem}
\begin{proof}
We patch, as in \cite[\S6.5]{10author} and \cite[\S8]{miagkov-thorne}, replacing the Fontaine--Laffaille local condition at $v \in S_p$ with the crystalline weight $0$ condition. Once again, we use \cite[Theorem 4.2.15]{caraiani-newton} to ensure that we have the necessary maps from deformation rings with these local conditions to our Hecke algebras. We record the output of this patching process, complete details of which can be found in \cite[\S6.4--6.5]{10author}. Fix a tuple $\chi = (\chi_{v, i})_{v \in R, i = 1, \dots, n}$ of characters $\chi_{v, i} : k(v)^\times \to \cO^\times$ which are trivial modulo $\varpi$, and with $\chi_{v,1},\ldots,\chi_{v,n}$ pairwise distinct for each $v\in R$. Patching will provide us with the following:

\begin{enumerate}
	\item A power series ring $S_\infty=\cO\llb X_1,\cdots,X_r\rrb$ with augmentation ideal $\mathfrak{a}_\infty = (X_1, \dots, X_r)$. 
	\item\label{ihara_complex_iso} Perfect complexes $C_\infty, 
	C'_\infty$ of $S_\infty$-modules, an isomorphism 
	\[ C_\infty \otimes^\bL_{S_\infty} S_{\infty} / \varpi \cong C'_\infty \otimes^\bL_{S_\infty} S_{\infty} / \varpi \]
	in $\mathbf{D}(S_\infty / \varpi)$ and an isomorphism \[C_\infty \otimes^\bL_{S_\infty} S_{\infty} / \mathfrak{a}_{\infty} \cong R \Hom_\cO( R \Gamma(X_K, \cO)_\m, \cO)[-d]\] in $\mathbf{D}(\cO)$.
	\item Two $S_\infty$-subalgebras
	\[ T_\infty\subset 
	\End_{\bD(S_\infty)}(C_\infty) \]
	and
	\[ T'_\infty \subset 
	\End_{\bD(S_\infty)}(C'_\infty), \]
	which have the same image in
	\[ \End_{\bD(S_\infty/\varpi)}(C_\infty \otimes^\bL_{S_\infty} S_{\infty} / \varpi ) = \End_{\bD(S_\infty/\varpi)}(C'_\infty \otimes^\bL_{S_\infty} S_{\infty} / \varpi ), \]
	where these endomorphism algebras are identified using the fixed isomorphism in (\ref{ihara_complex_iso}). Call this common image 
	$\overline{T}_\infty$. Note that $T_\infty$ and $T'_\infty$ are finite 
	$S_\infty$-algebras. The map \[T_\infty \to \End_{\mathbf{D}(\cO)}(C_\infty \otimes^\bL_{S_\infty} S_{\infty} / \mathfrak{a}_{\infty}) = \End_{\mathbf{D}(\cO)}(R \Gamma(X_K, \cO)_\m)^{op}\] factors through a map $T_\infty \to \TT(K,\cO)_\m$.
	\item Two Noetherian complete local $S_\infty$-algebras $R_\infty$ and $R'_\infty$, which are power series algebras over $R_{\cS_{1}}^{S,\loc}$ and $R_{\cS_{\chi}}^{S,\loc}$ respectively. We have a surjective $R_{\cS_{1}}^{S,\loc}$-algebra map $R_\infty \to R_{\cS_{1}}$, which factors through an $\cO$-algebra map $R_\infty/\mathfrak{a}_{\infty} \to R_{\cS_{1}}$. We also have surjections $R_\infty\onto T_\infty/I_\infty$, $R'_\infty\onto 
	T'_\infty/I'_\infty$, where $I_\infty$ and $I'_\infty$ are nilpotent ideals. We write $\overline{I}_\infty$ and $\overline{I}'_\infty$ for the image of these 
	ideals in $\overline{T}_\infty$. These maps fit into a commutative diagram
\[ \xymatrix{ R_\infty  \ar[d] \ar[r] & R_{\cS_1}\ar[d] \\
	T_\infty/I_\infty \ar[r] & \TT(K,\cO)_\m^{\red}.}\]
	\item An isomorphism $R_\infty/\varpi\cong R'_\infty/\varpi$ 
	compatible with the $S_\infty$-algebra structure and the actions 
	(induced from $T_\infty$ and $T'_\infty$) on 
	\[ H^*( C_\infty \otimes^\bL_{S_\infty} S_{\infty} / \varpi)/(\overline{I}_\infty+\overline{I}'_\infty) = H^*( C'_\infty \otimes^\bL_{S_\infty} S_{\infty} / \varpi)/(\overline{I}_\infty+\overline{I}'_\infty), \]
	where these cohomology groups are identified using the fixed isomorphism. 
	\item Integers $q_0 \in \Z$ and $l_0 \in \Z_{\geq 0}$ such that \[ H^\ast(X_K,E)_\m\neq 0, \]
	and these groups are non-zero only for degrees in the interval $[q_0,q_0+l_0]$. Moreover, $\dim R_\infty=\dim R'_\infty=\dim S_\infty -l_0$.
\end{enumerate}
With that out of the way, we  let $x \in \Spec R_\infty$ be the automorphic point coming from $\ker(f_1)$. By the first part of \cite[Proposition 5.4.2]{caraiani-newton}, there is an irreducible component $C_a$ of $\Spec R_\infty$, containing $x$, with $C_a \subset \Spec T_\infty$.  Let $C$ be any irreducible component of $\Spec R_\infty$ containing $\ker(f_2)$. Since $\rho_1|_{G_{F_v}} \sim \rho_2|_{G_{F_v}}$ for each $v \in S_p$, $C$ and $C_a$ map to the same irreducible component of $\Spec R_{\cS_1}^{S_p,\loc}$ (we are using part (\ref*{regular generic fibre}) of Lemma~\ref{lem:local def ring props} here, which says that each characteristic $0$ point lies in a unique irreducible component of $R_{\cS_1}^{S_p,\loc}$). By Lemma \ref{lem:local def ring props}, the generic points of $C \cap \Spec R_\infty/\varpi$ and $C_a \cap \Spec R_\infty/\varpi$ all generalize to the same irreducible component of $\Spec R_\infty'$. We can apply the second part of \cite[Proposition 5.4.2]{caraiani-newton} to deduce that $C \subset \Spec \TT_\infty$, and therefore $\ker(f_2)$ is in the support of $H^*(C_\infty)$. It follows as in \cite[Corollary 6.3.9]{10author} (see also \cite[Corollary 5.4.3]{caraiani-newton}) that $\ker(f_2)$ is in the support of $H^\ast(X_K,\cO)_\m$, as desired.
\end{proof}

\begin{proof}[Proof of Theorem \ref{thm:main_automorphy_lifting_theorem}]
  This is immediate from
  Theorem \ref{thm:alt_after_bc} via a standard base change
  argument identical to the one found in \cite[\S6.5.12]{10author}.
\end{proof}

\section{The Dwork family}\label{ssec:dwork}

\subsection{Definitions}
We begin by introducing the Dwork motives we need
to consider. For our purposes, we need to consider
the non-self dual motives (with coefficients) studied
in~\cite{Qian,QianPotential} rather than the self-dual
(generalized) symplectic motives previously considered
in~\cite{HSBT,blght}.

Let $n > 2$ and $N > 100 n + 100$ be integers, with $N$ odd and $(N, n) = 1$. Let $\zeta_N \in \overline{\Q}$ be a primitive $N^\text{th}$ root of unity. Let $R_0 = \Z[\zeta_N, N^{-1}]$, $T_0 = \Spec R_0[ t, (1-t^N)^{-1}]$, and let $Z \subset \mathbf{P}^{N-1}_{T_0}$ be the family of smooth hypersurfaces of degree $N$ and dimension $N - 2$ defined by the equation
\[ X_1^N + \dots + X_N^N = N t X_1 \dots X_N. \]
We write $\pi : Z \to T_0$ for the natural projection. Let $\mu_N$ denote the group of $N^\text{th}$ roots of unity in $\Z[\zeta_N]^\times$. Then the group $H = \mu_N^N / \Delta(\mu_N)$ acts on $ \mathbf{P}^{N-1}$ by multiplication of coordinates, and the subgroup
\[ H_0 = \ker( \prod : H \to \mu_N ) \]
preserves $Z$. The action of $H_0$ extends to an action of $H$ on the central fibre $Z_0$ (which is a Fermat hypersurface). 

Let $M = \Q(e^{2 \pi i / N}) \subset \C$, and set
\[ X = \Hom(H, M^\times), \]
\[ X_0 = \Hom(H_0, M^\times). \]
A choice of embedding $\tau : \Q(\zeta_N) \to \C$ determines an isomorphism 
\[ f_\tau : X \cong \ker\left( \sum : (\Z / N \Z)^N \to \Z / N \Z) \right), \]
 but we do not fix a preferred choice. We do choose a character $\underline{\chi} \in (\chi_1, \dots, \chi_N) \in X$ with the following properties:
\begin{itemize}
\item The trivial character of $\mu_N$ occurs $n+1$ times among $\chi_1, \dots, \chi_N$, and each other character appears at most once. 
\item Let $\rho_1, \dots, \rho_n$ be the $n$ distinct non-trivial characters $\mu_N \to M^\times$ which do not appear in $\chi_1, \dots, \chi_N$. Then the stabilizer of the set $\{ \rho_1, \dots, \rho_n \}$ in $\Gal(M / \Q)$ is trivial.
\end{itemize}
The existence of such $\underline{\chi}$ is established in \cite[Lem.\ 3.1]{QianPotential}, as a consequence of the assumption $N > 100 n + 100$. The precise choice is not important. 

For any place $\lambda$ of $M$ of characteristic $l$, we define $\cV_\lambda = (\pi[1/l]_\ast \cO_{M_\lambda})^{H_0 = \underline{\chi}|_{H_0}}$. It is a lisse sheaf of finite free $\cO_{M_\lambda}$-modules on $T_0[1/l]$. If $k$ is a perfect field which is an $R_0[1/l]$-algebra, and $t \in T_0(k)$, then we write $V_{t, \lambda} = \cV_{\lambda, \overline{t}}$ for the stalk at a geometric point lying above $t$; it is an $\cO_{M_\lambda}[G_k]$-module, finite free as $\cO_{M_\lambda}$-module. 

Katz \cite{KatzExponential, KatzDwork} defines hypergeometric sheaves on $T_1 = \Spec R_0[t, t^{-1}, (1-t)^{-1}]$. We give the definition just in the case of interest. Let $j : T_1 \to \mathbf{G}_{m, R_0}$ be the natural open immersion, and let $f : T_1 \to \mathbf{G}_{m, R_0}$ be the map induced by $t \mapsto 1-t$. Fixing again a place $\lambda$ of $M$ of characteristic $l$, let $\cL_i$ denote the rank~ 1 lisse $M_\lambda$-sheaf  on $\mathbf{G}_{m, R_0}[1/l]$ associated to $\rho_i$ and the $\mu_N$-torsor $\GG_m\xrightarrow{(\cdot)^N}\GG_m$, and let $\mathcal{F}_i = j[1/l]_! f[1/l]^\ast \cL_i$. We set
\[ \cE_\lambda = j[1/l]^\ast (\mathcal{F}_1 \ast_! \mathcal{F}_2 \ast_! \dots \ast_! \mathcal{F}_n)[n-1], \]
where $\ast_!$ denotes multiplicative convolution with compact support. 

\subsection{Basic properties and good ordinary points}
Associated to specializations of~$\cE_\lambda$ are compatible systems
of Galois representations. In this section, we establish some of their basic properties.
Most importantly, we prove (in Proposition~\ref{ordinarypoints}) the existence of many specializations which have crystalline
ordinary reduction. (In~\cite{Qian}, Qian proves that specializations sufficiently
close to~$t = \infty$ are semistable ordinary, but that is not sufficient for our purposes where
we need to work with crystalline representations.)

\begin{theorem}\label{thm_hypergeometric_sheaf} \leavevmode
\begin{enumerate}
\item $\cE_\lambda$ is a lisse $M_\lambda$-sheaf on $T_1[1/l]$ of rank $n$. The sheaf $\cE_\lambda \otimes_{M_\lambda} \overline{M}_\lambda$ is geometrically irreducible. Moreover, $\cE_\lambda$ is pure of weight $n-1$ and there is an isomorphism $\det \cE_\lambda \cong M_\lambda(n(1-n)/2)$. 
\item Let $k$ be an $R_0[1/l]$-algebra which is a finite field of cardinality $q$, and let $x \in T_1(k)$. Then we have
\numequation\label{eqn: trace on E} \tr( \Frob_k \mid \cE_{\lambda, \overline{x}} ) = (-1)^{n-1} \sum_{\substack{x_1, \dots, x_n \in k \\ \prod_{i=1}^n x_i = x}} \prod_{i=1}^n \rho_i( ( 1-x_i )^{(q-1)/N} ) \end{equation}
where we identify $\mu_N = k^\times[N]$ and extend $\rho_i$ by $\rho_i(0) = 0$.
\item There exists a \emph{(}unique\emph{)} continuous character 
$$\Psi_\lambda : \pi_1(\Spec R_0[1/l]) \to \cO_{M_\lambda}^\times$$
 with the following property: let $T_0' = T_0[1/t]$, let $j' : T_0' \to T_0$ be the natural open immersion, and let $g : T_0' \to T_1$ be the map induced by $t \mapsto t^{-N}$. Then there is an isomorphism of $M_\lambda$-sheaves on $T_0'[1/l]$:
\numequation\label{eqn: V to E iso} (j'[1/l])^\ast \cV_\lambda \otimes_{\cO_{M_\lambda}} M_\lambda \cong g^\ast \cE_\lambda \otimes_{M_\lambda} M_\lambda(\Psi_\lambda). \end{equation}
\end{enumerate}
\end{theorem}
\begin{proof}
The construction and properties of $\cE_\lambda$ are summarized in \cite{KatzDwork} (where it is the sheaf denoted $\mathcal{H}^{can}( \mathbf{1} (n \text{ times}), \{ \rho_i \})$) and given in detail in \cite[Ch. 8]{KatzExponential}. See also \cite[\S A.1.6]{DrinfeldKedlaya}. The computation of the determinant follows from \cite[Theorem 8.12.2(1a)]{KatzExponential}, noting that $\prod_{i=1}^n \rho_i = \mathbf{1}$. Property (2) follows from the definition. The existence of $\Psi_\lambda$ as in (3) follows from \cite[Theorem 5.3]{KatzDwork} (we have also used the relation $[-1]^\ast \mathcal{H}^{can}( \mathbf{1} (n \text{ times}), \{ \rho_i \}) \cong \mathcal{H}^{can}( \{ \rho^{-1}_i \}, \mathbf{1} (n \text{ times}))$). The uniqueness follows from geometric irreducibility and Schur's lemma. 
\end{proof}
\begin{lemma}
There exists a Hecke character $\Psi : \Q(\zeta_N)^\times \backslash
\mathbf{A}_{\Q(\zeta_N)}^\times \to \C^\times$ of type $A_0$,
unramified  away from $N$, and with field of definition contained inside $M$, such that $\Psi_\lambda$ is associated to $\Psi$. In other words, $\Psi_\lambda$ is `independent of $\lambda$'. Moreover, $\Psi$ is defined over $M$ and we have $\Psi c(\Psi) = | \cdot |^{N-n}$. \end{lemma}
\begin{proof}
We may argue as in \cite[Question 5.5]{KatzDwork} to see that almost all Frobenius traces of $\Psi_\lambda$ lie in $M$, and are independent of $\lambda$. The existence of the character $\Psi$, of type $A_0$, follows from the main result of \cite{henniart}. There are $c$-linear isomorphisms $\cV_{c(\lambda)} \cong \cV_\lambda^\vee(1-N)$ and $\cE_{c(\lambda)} \cong \cE_\lambda^\vee(1-n)$, so the final part is again a consequence of Schur's lemma. 
\end{proof}
For any place $\lambda$ of $M$ of characteristic $l$, we define
$\cW_\lambda = \cV_\lambda \otimes_{\cO_{M_\lambda}}
\cO_{M_\lambda}(\Psi_\lambda^{-1})$, a lisse sheaf of finite free
$\cO_{M_\lambda}$-modules on $T_0[1/l]$. Thus $\cW_\lambda
\otimes_{\cO_{M_\lambda}} M_\lambda$ is a lisse $M_\lambda$-sheaf of rank
$n$ which is pure of weight $n-1$, geometrically irreducible, and of
determinant $M_\lambda(n(1-n)/2)$. If $k$ is a perfect field which is
an $R_0[1/l]$-algebra, and $t \in T_0(k)$, then we write $W_{t,
  \lambda} = \cW_{\lambda, \overline{t}}$ for the stalk at a geometric
point lying above $t$; it is an $\cO_{M_\lambda}[G_k]$-module, finite
free as $\cO_{M_\lambda}$-module. The local systems $\cW_\lambda$ are
the ones we will use in building the moduli spaces used in the
Moret-Bailly argument in~\S\ref{subsec_proof_of_pot_aut}.
 In particular, let us write $\overline{\cW}_\lambda = \cW_\lambda \otimes_{\cO_{M_\lambda}} k(\lambda)$ and define $\overline{W}_{t, \lambda}$ similarly. 
\begin{prop}\label{prop_independence_of_l} Let $F / \Q(\zeta_N)$ be a number field.
\begin{enumerate}
\item Let $v$ be a finite place of $F$ of characteristic $l$, and let
  $\lambda$ be a place of $M$ of characteristic not equal to $l$. If $l \nmid N$ and $t \in T_0(\cO_{F_v})$, then $W_{t, \lambda}$ is unramified, and the polynomial $Q_v(X) = \det(X - \Frob_v \mid W_{t, \lambda})$ has coefficients in $\cO_M[X]$ and is independent of $\lambda$. 
\item Let $v$ be a finite place of $F$ of characteristic $l$, and let $\lambda$ be a place of $M$ of the same characteristic. Let $t \in T_0(F_v)$. Then $W_{t, \lambda}$ is de Rham and for any embedding $\tau : F_v \to \overline{M}_\lambda$, we have $\mathrm{HT}_\tau(W_{t, \lambda}) = \{ 0, 1, \dots, n-1 \}$. If $l \nmid N$ and $t \in T_0(\cO_{F_v})$, then $W_{t, \lambda}$ is crystalline and the characteristic polynomial of $\Frob_v$ on $\mathrm{WD}(W_{t, \lambda})$ equals $Q_v(X)$. In particular, $W_{t, \lambda}$ is ordinary if and only if the roots of $Q_v(X)$ in $\overline{M}_\lambda$ have $l$-adic valuations $0, [k(v) : \F_l], \dots, (n-1) [k(v) : \F_l]$. 
\item Let $t \in T_0(F)$, and let $S$ be the set of finite places $v$ of $F$ such that either $v | N$, or $v \nmid N$ and $t \not\in T_0(\cO_{F_v}) \subset T_0(F_v)$. Then
\[ ( M, S, \{ Q_v(X) \}_{v \not \in S}, \{ W_{t, \lambda}^{ss} \}_\lambda, \{ \{0, 1, \dots, n-1\}\}_\tau  ) \]
is a weakly compatible system of $l$-adic representations of $G_F$ over $M$ of rank $n$, pure of weight $n-1$, in the sense of \cite[\S 5.1]{BLGGT}.
\end{enumerate} 
\end{prop}
\begin{proof}
For the first part, we note that $Z_t$ is smooth and proper over $\cO_{F_v}$, so by smooth proper base change $H^\ast_{\text{\'et}}(\overline{Z}_{t, F_v}, \cO_{M_\lambda})$ is unramified and there is an isomorphism
\[ H^{N-2}_{\text{\'et}}(Z_{t, \overline{F}_v}, \cO_{M_\lambda}) \cong H^{N-2}_{\text{\'et}}(Z_{\overline{t}, \overline{k(v)}}, \cO_{M_\lambda}), \]
where $\overline{t}$ denotes the image of $t$ in $T_0(k(v))$. \cite[Theorem 2(2)]{MR332791}  shows that for any $h \in H$ the characteristic polynomial of $h \cdot \Frob_v$ on this group has coefficients in $\cO_M$ and is independent of the choice of $\lambda \nmid l$, and this implies that $Q_v(X)$ also has coefficients in $\cO_M[X]$ and is independent of $\lambda$. 

For the second part, note that $W_{t, \lambda}$ is de Rham because it is a subquotient of $H_{\text{\'et}}^{N-2}(Z_t, M_\lambda) \otimes M_\lambda(\Psi_\lambda^{-1})$, which is de Rham. To compute the Hodge--Tate weights, we use \cite[Lemma 3.10]{QianPotential}, which implies that there is an integer $M_\tau$ such that $\mathrm{HT}_\tau(V_{t, \lambda}) = \{ M_\tau, M_\tau + 1, \dots, M_\tau + (n-1) \}$. Since $W_{t, \lambda}$ is a twist of $V_{t, \lambda}$, there is an integer $M'_\tau$ such that $\mathrm{HT}_\tau(W_{t, \lambda}) = \{ M'_\tau, M'_\tau + 1, \dots, M'_\tau + (n-1) \}$. Looking at determinants shows that $n M'_\tau + n(n-1)/2 = n(n-1)/2$, hence $M'_\tau = 0$. 

If further $l \nmid N$ and $t \in T_0(\cO_{F_v})$ then again $Z_t$ is smooth and proper, so $H_{\text{\'et}}^{N-2}(Z_t, M_\lambda)$ is crystalline, and also $M_\lambda(\Psi_\lambda^{-1})$ is crystalline, hence $W_{t, \lambda}$ is crystalline. The crystalline comparison theorem implies that there is an isomorphism 
\[ D_{cris}( H_{\text{\'et}}^{N-2}(Z_{t, \overline{F}_v}, \Q_l) ) \cong H^{N-2}_{cris}(Z_{\overline{t}} / F_{v, 0}), \]
respecting the action of Frobenius $\phi_v$ and $H_0$ on each side. Here $F_{v, 0}$ denotes the maximal absolutely unramified subfield of $F_v$.  Choosing an embedding $\sigma_0 : F_{v, 0} \to \overline{M}_\lambda$, there is an isomorphism
 \[ D_{cris}( H_{\text{\'et}}^{N-2}(Z_{t, \overline{F}_v}, \Q_l) ) \otimes_{F_{v, 0}, \sigma_0} \overline{M}_\lambda \cong H^{N-2}_{cris}(Z_{\overline{t}} / F_{v, 0}) \otimes_{F_{v, 0}, \sigma_0} \overline{M}_\lambda, \]
 equivariant for the $\overline{M}_\lambda$-linear action of $\phi_v^{[k(v) : \F_l]}$. By definition, $\mathrm{WD}(W_{t, \lambda})$ is the unramified representation of $W_{F_v}$ over $\overline{M}_\lambda$ afforded by the $\Psi_\lambda^{-1}$-twist of the $\underline{\chi}|_{H_0}$-isotypic subspace of the left-hand side. We therefore need to check that the characteristic polynomial of $\phi_v^{[k(v) : \F_l]}$ on the $\Psi_\lambda^{-1}$-twist of the $\underline{\chi}|_{H_0}$-isotypic subspace of the right-hand side equals $Q_v(X)$. This follows again from \cite[Theorem 2(2)]{MR332791} (applicable here by the main result of \cite{MR904940}). The characterization of ordinary representations follows from \cite[Lemma 2.32]{ger}. 
 
 The third part follows from the first two parts and the definition of
 a weakly compatible system. 
\end{proof}
We now apply the results of Drinfeld--Kedlaya \cite{DrinfeldKedlaya}
to deduce that the~$W_{t,\lambda}$ are ordinary for generic choices of~$t$.
\begin{prop} \label{ordinarypoints}
Let $v$ be a place of $\Q(\zeta_N)$ of characteristic $l \nmid N$, and let $\lambda$ be a place of $M$ of the same characteristic. Then there exists a non-empty Zariski open subset $U(v; \lambda) \subset T_{0, k(v)}$ with the following property: for any finite extension $F_w / \Q(\zeta_N)_v$ and any $t \in T_0(\cO_{F_w})$ such that $\overline{t} = t \text{ mod }(\varpi_w) \in U(v; \lambda)(k(w))$, $W_{t, \lambda}$ is a crystalline ordinary representation of $G_{F_w}$. 
\end{prop}
\begin{proof}
Fix an auxiliary place $\mu$ of $M$ of characteristic not $l$. If $k /
k(v)$ is a finite extension of cardinality $q$ and $x \in T_0(k)$, we
write $Q_x(X) \in \cO_M[X]$, for the characteristic polynomial of $\Frob_x$ on $W_{x, \mu}$. Let $s_1(x) \geq s_2(x) \geq \dots \geq s_n(x)$ denote $[k : \F_l]^{-1}$ times the $l$-adic valuations of the roots of $Q_x(X)$ in $\overline{M}_\lambda$. Observe that these normalized slopes $s_i(x)$ do not change if $k$ is replaced by a larger extension (leaving the point $x$ unchanged). By Proposition \ref{prop_independence_of_l}, it suffices to show the existence of a non-empty Zariski open subset $U \subset T_{0, k(v)}$ such that if $k / k(v)$ is a finite extension and $x \in U(k)$, then $s_i(x) = n-i$ for each $i = 1, \dots, n$. 

By \cite[Theorem 1.3.3]{DrinfeldKedlaya}, we can find a non-empty Zariski open subset $V \subset T_{0, k(v)}$ such that the numbers $s_i(x)$ are constant for $x \in V(k)$, and moreover such that $s_i(x) \leq s_{i+1}(x) + 1$ for each $i = 1, \dots, n-1$. To complete the proof, it suffices to show that there is a non-empty Zariski open subset $U \subset V$ such that if $x \in U(k)$, then $s_n(x) = 0$. Indeed, consideration of determinants shows that $s_1(x) + s_2(x) + \dots + s_n(x) = n(n-1)/2$ for all $x \in T_0(k)$. If $x \in V(k)$ and $s_n(x) = 0$ then $s_i(x) \leq n-i$ for each $i = 1, \dots, n$, hence $s_1(x) + \dots + s_n(x) \leq n(n-1)/2$, with equality if and only if $s_i(x) = n-i$ for each $i = 1, \dots, n$. 

Finally, by~\eqref{eqn: V to E iso}, it is enough to show the analogous statement for the pullback of $\cE_\mu$ to $T_{1, k(v)}$. We will prove this by showing that there is a non-empty Zariski open subset $U_1 \subset T_{1, k(v)}$ such that if $x \in T_{1, k(v)}(k)$, then  $\tr( \Frob_x \mid \cE_{\mu, \overline{x}}) \not\equiv 0 \text{ mod }\lambda$, or in other words (using~\eqref{eqn: trace on E}) that
\[ \sum_{\substack{x_1, \dots, x_n \in k \\ \prod_{i=1}^n x_i = x}} \prod_{i=1}^n \rho_i( ( 1-x_i )^{(q-1)/N} ) \not\equiv 0 \text{ mod }\lambda. \]
We will produce this set $U_1$ by a computation following \cite[\S A.3]{DrinfeldKedlaya}. Let $\tau : \Q(\zeta_N) \to M$ be an isomorphism identifying the place $v$ with the place $\lambda$. The character $k(v)^\times \to k(v)^\times$, $z \mapsto \tau^{-1} \rho_i( z^{(q_v-1)/N} )$, is given by the formula $z \mapsto z^{c_i}$ for some integer $c_i$ with $1 \leq c_i \leq q_v-2$. If $q = q_v^d$ then we find that the pre-image under $\tau$ of the left-hand side of the displayed equation is given by 
\[ \sum_{\substack{x_1, \dots, x_n \in k \\ \prod_{i=1}^n x_i = x}} \prod_{i=1}^n \mathbf{N}_{k / k(v)} (1 - x_i)^{c_i} = \sum_{\substack{x_1, \dots, x_n \in k \\ \prod_{i=1}^n x_i = x}} \prod_{i=1}^n  (1 - x_i)^{\widetilde{c}_i}, \]
where $\widetilde{c}_i = c_i \cdot (q-1) / (q_v-1) < q - 1$. This we can in turn compute as
\begin{multline*}   \sum_{\substack{x_1, \dots, x_n \in k \\ \prod_{i=1}^n x_i = x}} \sum_{\substack{ 0 \leq r_i \leq \widetilde{c}_i \\ i = 1, \dots, n}} \prod_{i=1}^n \binom{\widetilde{c}_i}{r_i} (-x_i)^{r_i} \\
 =   \sum_{\substack{ 0 \leq r_i \leq \widetilde{c}_i \\ i = 1, \dots, n}} (-1)^{r_1 + \dots + r_n} \left(\prod_{i=1}^n \binom{\widetilde{c}_i}{r_i}\right) \sum_{x_1, \dots, x_{n-1} \in k^\times}   \left( \prod_{i=1}^{n-1} x_i^{r_i - r_n} \right) x^{r_n}. \end{multline*} 
 We now use that if $r \in \Z$ and $r \not\equiv 0 \text{ mod }(q-1)$, then $\sum_{z \in k^\times} z^r = 0$. This implies that the inner sum vanishes except if $r_i = r_n$ for each $i = 1, \dots, n-1$. We obtain  (noting that there are only finitely many non-zero terms in the sum on the right-hand side):
 \[ \sum_{0 \leq r \leq \min(\widetilde{c}_i)} (-1)^{nr}  \prod_{i=1}^n \binom{\widetilde{c}_i}{r} (-1)^{n-1} x^r = (-1)^{n-1} \sum_{r \geq 0} (-1)^{nr}   \prod_{i=1}^n \binom{\widetilde{c}_i}{r} x^r. \]
 Define a polynomial $u(T) \in k(v)[T]$:
 \[ u(T) = \sum_{r \geq 0} (-1)^{nr}   \prod_{i=1}^n \binom{c_i}{r} T^r. \]
 We claim that there is an equality 
 \[ \sum_{r \geq 0} (-1)^{nr}   \prod_{i=1}^n \binom{\widetilde{c}_i}{r} x^r = \mathbf{N}_{k / k(v)}( u(x) ). \]
 This will complete the proof: indeed, it will imply that we can take $U_1 = T_{1, k(v)}[1/u]$ (noting that $u$ is non-zero, since its constant term is 1, so $U_1$ is indeed non-empty, and also $u$ is independent of the field extension $k/k(v)$). To show this, we expand
\begin{multline*}  \mathbf{N}_{k / k(v)}( u(x) ) = \left( \sum_{r \geq 0} (-1)^{nr}   \prod_{i=1}^n \binom{c_i}{r} x^r \right)^{1 + q_v + \dots + q_v^{d-1}} 
\\ = \sum_{r_0, \dots, r_{d-1} \geq 0} (-1)^{n(r_0 + \dots + r_{d-1} q_v^{d-1})} \prod_{i=1}^n \prod_{j=0}^{d-1} \binom{c_i}{r_j} x^{r_0 + r_1 q_v + \dots + r_d q_v^{d-1}}.
\end{multline*}   
We now observe that a given tuple $(r_0, \dots, r_{d-1})$ can contribute a non-zero summand only if $r_j < q_v-1$ for each $j$, so each value of $r = r_0 + r_1 q_v + \dots + r_{d-1} q_v^{d-1}$ is represented at most once. Furthermore, since $(1+X)^{\widetilde{c}_i} = \prod_{j=0}^{d-1} (1+X^{q_v^j})^{c_i}$ in $\F_l[X]$, we have in this case a congruence
\[ \binom{\widetilde{c}_i}{r} \equiv \prod_{j=0}^{d-1} \binom{c_i}{r_j} \text{ mod }l, \]
showing that $\mathbf{N}_{k / k(v)}( u(x) )$ indeed equals
\[ \sum_{r \geq 0} (-1)^{nr} \prod_{i=1}^n \binom{\widetilde{c}_i}{r} x^r, \]
 as desired. 
\end{proof}

\subsection{Basics on unitary groups over finite fields}
In order to discuss the possible (residual) images of the Galois
representations associated to our Dwork family, we recall here some basic facts
about unitary groups over finite fields which will be used in the sequel. (Nothing here
is original but we include it for convenience of exposition.)

Let~$l/k$ be a quadratic extension of finite fields.
 Let $p$ be a prime and let~$M \in M_n(l)$. Let~$M^t$ denote the transpose of~$M$ and~$M^{c}$
the conjugate of~$M$ by the generator of~$\Gal(l/k)$. We define the adjoint~$M^{\dagger}$
of~$M$ to be~$M^{\dagger}:= (M^{c})^t = (M^t)^{c}$.
  Note that~$(AB)^{\dagger} = B^{\dagger} A^{\dagger}$. 
  We recall:
 
 \begin{df} \label{unitary} The unitary group~$\GU_n(l)$
 is the subgroup of matrices~$M \in \GL_n(l)$ satisfying~$M^{\dagger} M = \lambda \in k^{\times}$.
  Let~$\nu$ be the
  multiplier character~$\nu: \GU_n(l) \rightarrow k^{\times}$ sending~$M$ to~$M^{\dagger} M$ and let~$\SU_n(l)$ denote the kernel of~$\nu$.
 \end{df}

  If~$V$ is a representation of a finite group~$G$ over~$l$, let~$V^{c}:=V\otimes_{l,c}l$ denote
 the representation obtained by conjugating the coefficients by the
 generator of~$\Gal(l/k)$. If $x \in V$, we set $x^c := x\otimes 1 \in V^c$.
 If~$x \in l$, we write either~$cx$ or~$x^c$ for the conjugate of~$x$ by~$c \in \Gal(l/k)$. 
 
 \begin{df}\label{hermetian} If~$l/k$ is a quadratic extension of finite fields and~$\Gal(l/k) \simeq \langle c \rangle$,
 then a Hermitian form on a vector space~$V$ over~$l$ is an $l$-valued pairing on~$V$ which is~$k$-bilinear, satisfies
 $\langle  ax,y \rangle = a \langle x,y\rangle$ and~$\langle x,ay \rangle = ca \langle x,y\rangle$ for $a \in l$,
 and moreover satisfies~$\langle y,x \rangle = \langle x,y \rangle^c  =
c \langle x,y \rangle$.
 \end{df}
 
\begin{rem} \label{minushermetian}
Scaling the pairing by an element~$\eta \in l$ such that~$c \eta = - \eta$, all
the conditions remain true except that now~$\langle x,y \rangle = - c \langle x,y \rangle$.
\end{rem}

 The basic fact concerning  unitary groups over finite fields is that there is essentially
 only one non-degenerate Hermitian form. In practice, it will be useful to formulate
 this in the following lemma.
 
 \begin{lemma} \label{selfdual} Let~$V$ be a vector space over~$l$ with an absolutely
 irreducible representation of a group~$G$.
 Suppose that there is an $l$-linear isomorphism
 \begin{equation}
 \label{twistselfdual}
 V^{\vee} \simeq V^{c} \otimes \chi^{-1}
\end{equation}
 for some multiplier character~$\chi: G \rightarrow k^{\times}$. Then, after a suitable choice of basis for~$V$,
  the corresponding map~$G \rightarrow \GL_n(l)$ has image in~$\GU_n(l)$.
 \end{lemma}
 
 \begin{proof} An isomorphism~(\ref{twistselfdual}) is equivalent to the existence of a~$G$-equivariant
 non-degenerate bilinear pairing:
 $$\psi: V \times V^{c} \rightarrow \chi,$$
 where by abuse of notation we consider~$\chi$ as a~$1$-dimensional
 vector space over~$l$. By Schur's Lemma,  the isomorphism in~(\ref{twistselfdual})
 is unique up to scaling and thus~$\psi$ is also unique up to scaling. If we define~$\psi'$ to be the map:
  $$\psi'(x,y^{c}) = \psi(y,x^{c})^{c},$$
then~$\psi'$ is also a~$G$-equivariant bilinear map from~$V \times V^{c}$ to~$\chi$ and hence~$\psi'$
 is equal to~$\psi' = \lambda \psi$ for some~$\lambda \in l^\times$, that is,
 $$\psi(y,x^{c})^{c} = \psi(x,y^{c}) \cdot \lambda.$$ 
 Applying this twice, we get~$\psi(x,y^{c}) = \lambda \cdot \lambda^{c}\cdot  \psi(x,y^{c})$,
 and thus~$N_{l/k}(\lambda) = 1$. By Hilbert Theorem~$90$, it follows that~$\lambda = c \eta/\eta$
 for some~$\eta \in l$. Replacing~$\psi(x,y)$ by~$\psi(x,y)/\eta$, we deduce
 that~$\psi(x,y^{c}) = \psi(y,x^{c})^c$.
It follows that
 $$\langle x,y \rangle := \psi(x,y^{c})$$
defines a non-degenerate Hermitian form on $V$ in the sense
of Definition~\ref{hermetian}.
 Let~$A$ denote the matrix associated to this Hermitian form, so that~$A^{\dagger} = A$.
 Then~$G \subset \GU(V,A)$, that is, matrices $M$ such that~$M^{\dagger} A M = \lambda \cdot A$ for some~$\lambda \in k^{\times}$.
 But now we use the fact that there is a unique non-degenerate equivalence class
 of Hermitian forms
associated to~$l/k$, namely, they are all equivalent to~$A = I$ and so~$G \subset \GU_n(l)$.
(See, for example,~\cite[\S4]{MR656449}.)
  \end{proof}

\subsection{Moduli spaces and monodromy}\label{subsec_def_of_dwork_moduli_space_with_level_structure}
We shall now discuss a number of moduli spaces related to finding Dwork motives
with fixed residual representations, and compute the corresponding monodromy groups.
Since it will be important to find such motives whose~$p$-adic representations
are related to symmetric powers of ``niveau two'' representations, for our applications
we will have to take~$p \equiv -1 \bmod N$, and thus be in cases excluded by~\cite{QianPotential}.

\begin{df} \label{pqmoduli}
Let $l_1, l_2 \nmid 2N$ be distinct primes and let $\lambda_1, \lambda_2$ be places of $\Q(\zeta_N)$ of these characteristics. Suppose we are given the following data:
\begin{enumerate}
\item \label{arepresentation} A field $F / \Q(\zeta_N)$ and for each $i = 1, 2$ an \'etale sheaf $\overline{U}_{\lambda_i}$ on $\Spec F$ of $k(\lambda_i)$-modules of rank $n$.
\item \label{determinant} For each $i = 1, 2$ an isomorphism $\eta_i : \wedge^n \overline{\cW}_{\lambda_i} \to \wedge^n \overline{U}_{\lambda_i, T_{0, F}}$ of sheaves of $k(\lambda_i)$-modules.
\item \label{another} For each $i = 1, 2$, if $-1 \text{ mod } N \in \langle l_i \rangle \leq (\Z / N \Z)^\times$ (equivalently: if $c \in \Gal(k(\lambda_i) / \F_{l_i})$), then we fix in addition a perfect $\F_l$-bilinear morphism $\langle \cdot, \cdot \rangle_{\overline{U}_{\lambda_i}} : \overline{U}_{\lambda_i} \times \overline{U}_{\lambda_i} \to k(\lambda_i)(1-n)$  satisfying the following conditions:
\begin{enumerate}
\item For all $x, y \in \overline{U}_{\lambda_i}$, $a \in k(\lambda_i)$, we have $\langle a x, y \rangle = a \langle x, y \rangle$, $\langle x, a y \rangle = c(a) \langle x, y \rangle$.
\item For all $x, y \in \overline{U}_{\lambda_i}$, we have $\langle y, x \rangle = - c \langle x, y \rangle$ .
\end{enumerate}
That is, the pairing is Hermitian in the sense of Definition~\ref{hermetian} up to a scalar~$\eta$ with~$c \eta = - \eta$, 
see Remark~\ref{minushermetian}.
\end{enumerate}
Note that if $-1 \text{ mod } N \in \langle l_i \rangle$ then there is also a perfect $\F_l$-bilinear morphism $\langle \cdot, \cdot \rangle_{\overline{\cW}_{\lambda_i}} : \overline{\cW}_{\lambda_i} \times \overline{\cW}_{\lambda_i} \to k(\lambda_i)(1-n)$ satisfying the same two conditions, induced by taking Poincar\'e duality on the relative cohomology with $\cO_{M^+_\lambda}$-coefficients of the hypersurface $Z \to T_0$ and extending sesquilinearly to $\cO_{M_\lambda}$-coefficients. If $k / F$ is a field extension and $t \in T_0(k)$, then we write $\langle \cdot, \cdot \rangle_{\overline{W}_{t, \lambda_i}}$ for the induced perfect pairing on $\overline{W}_{t, \lambda_i}$. 

Given such data, let us write $\mathcal{F}(\{ \overline{U}_{\lambda_i} \})$ for the functor which sends a scheme $S \to T_{0, F}$ to the set of pairs of isomorphisms $\phi_i : \overline{\cW}_{\lambda_i, S} \to \overline{U}_{\lambda_i, S}$ $( i = 1, 2)$ satisfying the following conditions:
\begin{itemize}
\item For each $i = 1, 2$, $\wedge^n \phi_i = \eta_i$.
\item For each $i = 1, 2$, if $-1 \text{ mod } N \in \langle l_i \rangle$, then $\phi_i$ intertwines $\langle \cdot, \cdot \rangle_{\overline{\cW}_{\lambda_i}, S}$ and $\langle \cdot, \cdot \rangle_{\overline{U}_{\lambda_i}, S}$. 
\end{itemize}
Then $\mathcal{F}(\{ \overline{U}_{\lambda_i} \})$ is represented by a finite \'etale $T_{0, F}$-scheme $T(\{ \overline{U}_{\lambda_i} \})$. 
\end{df}

We need the following variant of~\cite[Proposition 3.8]{QianPotential}.

\begin{prop} \label{monodromy}
With notation as above, 
$T(\{ \overline{U}_{\lambda_i} \})$ is a geometrically irreducible smooth $F$-scheme.
  \end{prop}
\begin{proof}
We need to show that the geometric monodromy group $\pi_1(T_{0, \overline{\Q}})$ acts transitively on the fibres of $T(\{ \overline{U}_{\lambda_i} \})$ over $T_0$. 
The existence of the pairing $\langle \cdot, \cdot \rangle_{\overline{\cW}_{\lambda_i}}$ shows that if $-1 \text{ mod } N \in \langle l_i \rangle$, then the image of the geometric monodromy group acting on the geometric generic fibre of $\overline{\cW}_{\lambda_i}$ may be identified with a subgroup of $\mathrm{SU}_n(k(\lambda_i))$,
and otherwise it may be identified with a subgroup of~$\mathrm{SL}_n(k(\lambda_i))$.

 We claim that it is enough to show that equality holds in either of these cases. Indeed, let $H_i$ denote the image at each prime $l_i$ (which would then be either $\mathrm{SU}_n(k(\lambda_i))$ or $\SL_n(k(\lambda_i))$). 
 Since we are assuming~$l_1, l_2 \nmid 2$ and~$n > 2$, it follows that
the $H_i$ are perfect and their associated projective groups (i.e.\ the $H_i$ modulo their subgroups of scalar matrices) are simple (Lemma~\ref{steinberg}), and
moreover $H_1 \not\cong H_2$ (also by Lemma~\ref{steinberg}).
Goursat's lemma implies that the image of geometric monodromy acting on $\overline{\cW}_{\lambda_1} \times \overline{\cW}_{\lambda_2}$ must be $H_1 \times H_2$, completing the proof.

If~$-1 \text{ mod } N \not\in \langle l_i \rangle$, then the required statement follows 
from~\cite[Lemma 3.7]{QianPotential}.
Now suppose that $-1 \text{ mod } N \in \langle l_i \rangle$. In this case, we can follow the proof of \cite[Lemma 3.7]{QianPotential} (now allowing the case $-1 \text{ mod } N \in \langle l_i \rangle$, which is used there to exclude the possibility of image a special unitary group) to conclude that $H_i$ is isomorphic to a subgroup of $\mathrm{SU}_n(k(\lambda_i))$ which maps to $\mathrm{SU}_n(k(\lambda_i))$ or $\SL_n(k(\lambda_i))$ with image a normal subgroup of index dividing $N$. Since $N$ is coprime to $n$ by assumption, the only possibility is that this map is in fact an isomorphism and that $H_i \cong \mathrm{SU}_n(k(\lambda_i))$, as required. 
\end{proof}

\begin{rem} \label{pqmoduliremark} If~$l \equiv 1 \bmod N$, so that~$l$ splits completely in~$\Q(\zeta_N)$, and~$\lambda | l$, then
the data
 of an \'etale sheaf $\overline{U}_{\lambda}$ on $\Spec F$  satisfying
 conditions~(\ref{arepresentation}), (\ref{determinant}),
 and~(\ref{another}) of Definition~\ref{pqmoduli} is nothing more than a representation
 $$\rbar_l: G_F \rightarrow \GL_n(\F_l)$$
 with determinant~$\varepsilonbar^{-n(n-1)/2}$.
 If $l \equiv -1 \bmod N$, however,  then (in light of Lemma~\ref{selfdual})
 these conditions correspond to a representation
 $$\rbar_l: G_F \rightarrow \GU_n(\F_{l^2})$$
  with multiplier character~$\varepsilonbar^{1-n}$ (and determinant~$\varepsilonbar^{-n(n-1)/2}$).
  In practice, we shall only consider the moduli spaces~$T$ with primes~$l_1,l_2$ that are either~$\pm 1 \bmod N$. in which case we sometimes write  $T(\{ \overline{U}_{\lambda_i} \})$ as~$T(\rbar_{\lambda_1},\rbar_{\lambda_2})$,
  replacing the  \'etale sheaf with the corresponding representations, which are always assumed to satisfy 
  conditions~(\ref{arepresentation}), (\ref{determinant}),
 and~(\ref{another})  of Definition~\ref{pqmoduli}.
  We will also use the simpler variant $T(\rbar_{\lambda})$ where there is a single prime $l \equiv 1 \bmod N$, a choice of place $\lambda | l$, and a representation $\rbar_{\lambda}: G_F \rightarrow \GL_n(\F_l)$ of determinant $\varepsilonbar^{-n(n-1)/2}$. The $F$-scheme $T(\rbar_{\lambda})$ is geometrically irreducible.
   \end{rem}

The following lemma will be used to prove the existence of local points on $T(\{ \overline{U}_{\lambda_i} \})$ in certain cases. 
\begin{lemma}  \label{specializations}
Let $l > n$ be a prime such that $l \equiv -1 \text{ mod }N$, and let $v, \lambda$ be places of $\Q(\zeta_N)$, $M$, respectively, of residue characteristic $l$.
Let $\tau, c \tau : k(v) \to k(\lambda)$ be the two distinct isomorphisms, and let $\omega_\tau : I_{\Q(\zeta_N)_v} \to k(\lambda)^\times$ be the character $\tau \circ \Art_{\Q(\zeta_N)_v}^{-1}$. Then there is an isomorphism
\[ \overline{W}_{0, \lambda} \cong \bigoplus_{i=1}^n \omega_\tau^{i-1} \omega_{c \tau}^{n-i}. \]
\end{lemma}
\begin{proof}
The action of $H_0$ on $Z_0$ extends to an action of $H$, leading to a decomposition
\[ \overline{W}_{0, \lambda} = \bigoplus_{j = 1}^n \overline{W}_{0, \lambda, j}, \]
where $W_{0, \lambda, j}$ is the $\Psi_\lambda^{-1}$-twist of the $H$-eigenspace in $H^{N-1}(Z_{0, \overline{\Q}}, \cO_{M_\lambda})$ for the character $(\chi_1 \rho_j^{-1}, \dots, \chi_N \rho_j^{-1})$, and $\overline{W}_{0, \lambda, j} := W_{0, \lambda, j}\otimes_{\cO_{M_\lambda}}k(\lambda)$. Here we have used the computation of \cite[I.7.4]{MR654325}, which moreover shows that each summand here has rank 1 over $k(\lambda)$. Moreover, this decomposition is orthogonal with respect to $\langle \cdot, \cdot \rangle_{\overline{W}_{0, \lambda}}$, showing that $\overline{W}_{0, \lambda, j} \otimes_{k(\lambda), c} k(\lambda) \cong \overline{W}_{0, \lambda, j}^\vee \varepsilon^{1-n}$ as $k(\lambda)[G_{\Q(\zeta_N)}]$-modules. 

After permuting $\rho_1, \dots, \rho_n$, we can assume that $\mathrm{HT}_\tau(W_{0, \lambda, j}) = j-1$. Then we have $ \overline{W}_{0, \lambda, j} \cong k(\lambda)(\omega_\tau^{j-1} \omega_{c \tau}^{a_j})$ for some integers $a_j$ with $\{ a_1, \dots, a_n \} = \{ 0, \dots, n-1 \}$. The last sentence of the previous paragraph shows that we must in fact have $j-1 + a_j = n-1$, completing the proof.
\end{proof}

\subsection{A result of Moret-Bailly}\label{subsec: MB variant}
We will use the following variant of the extensions \cite[Theorem 3.1]{frankII}, \cite[Proposition 3.1.1]{BLGGT} of the main result of \cite{mb}. 
\begin{prop}\label{prop_variation_of_MB}
Let $F$ be an imaginary CM field, Galois over $\Q$, and let $T / F$ be a smooth, geometrically irreducible variety. Suppose given the following data:
\begin{enumerate}
\item A finite extension $F^{\avoid} / F$ and disjoint finite sets $S_0$ of rational primes.
\item For each $l \in S_0$ and each place $v | l$ of $F$, a Galois extension $L_v / F_v$. These have the property that if $\sigma \in G_{\Q_l}$ then $\sigma(L_v) = L_{\sigma(v)}$.
\item For each $l \in S_0$ and each place $v | l$ of $F$, a non-empty open subset $\Omega_v \subset T(L_v)$, invariant under the action of $\Gal(L_v / F_v)$.
\end{enumerate}
Then we can find a finite CM extension $F' / F$ and a point $P \in T(F')$ with the following properties:
\begin{enumerate}
\item $F' / \Q$ is Galois and $F' / F$ is linearly disjoint from $F^{\avoid} / F$.
\item For each $l \in S_0$ and each place $v | l$ of $F$ and $w | v$ of $F'$, there is an isomorphism $F'_w \cong L_v$ of $F_v$-algebras such that $P \in \Omega_v \subset T(F'_w) \cong T(L_v)$. 
\end{enumerate}
Suppose given further a finite group $G$ and a surjective homomorphism $f : \pi_1^{\text{\'et}}(T) \to G$. Then we can further choose $P$ so that the image of $f \circ P_\ast : G_{F'} \to G$ is surjective. 
\end{prop}
\begin{proof}
Without the last sentence, this is a special case of \cite[Proposition 3.1.1]{BLGGT}  (taking $K_0 = \Q$ in the notation there), noting that (as in \cite[Theorem 3.1]{frankII}) we can choose $F'$ to be of the form $F' = F E$ for a Galois, totally real extension $E / \Q$, and therefore in particular to be CM. 

To get the last sentence, it suffices to add further local conditions
at places of sufficiently large norm, ensuring using a Chebotarev density theorem for schemes of finite type over $\Z$ that the image of $f
\circ P_\ast$ meets every conjugacy class of $G$ (in close analogy
with the argument of  \cite[Proposition 3.2]{frankII} -- the
surjectivity is then a consequence of Jordan's theorem). To define the
necessary local conditions, we can spread $T$ out to a geometrically
irreducible scheme $\cT$, smooth and of finite type over $\cO_F$, %
such that $f$ factors through $\pi_1^{\text{\'et}}(\cT)$. Then \cite[Corollary 9.12]{MR2920749} shows that for any $X > 0$ and any conjugacy class $C \subset G$, we can find a finite place $v$ of $F$ of norm $q_v > X$ and a point $x \in \cT(k(v))$ such that the image of (arithmetic) Frobenius under $f \circ x_\ast$ lies in $C$. For each conjugacy class $C$ of $G$, we choose one such  place $v_C$ and point $x_C$ for each conjugacy class of $G$ and take $\Omega_{v_C}$ to be the pre-image of $x_C$ in $\cT(\cO_{F_{v_C}}) \subset T(F_{v_C})$.  We may assume that if $C \neq C'$ then $v_C$ and $v_{C'}$ have distinct residue characteristics $l_C \neq l_{C'}$, and then replace $S_0$ by $S_0 \cup \{ l_C \mid C \subset G \}$. Finally, if $v | l_C$ and $v \neq v_C$, we take $\Omega_v = T(F_v)$. Provided the norm $q_{v_C}$ is sufficiently large, these sets $\Omega_v$ will also be non-empty, as required. 
\end{proof}

\section{Preliminaries on deformation rings and Galois theory}

\subsection{Lemmas on components of Galois deformation rings}

We begin by defining a certain local representation which shall appear repeatedly in the sequel.

\begin{df} \label{df: defn of rho0}
For $n, m \in \Z_{\ge 1}$, let $\varepsilon_2, \varepsilon_2' : G_{\Q_{p^2}} \to \overline{\Z}_p^\times$ be the two Lubin--Tate characters trivial on $\Art_{\Q_{p^2}}(p)$, and let~$\rho_{n,m, 0}$ denote
the representation
\numequation\label{eqn: defn of rho0}\rho_{n,m,0} = \bigoplus_{i=1}^n \varepsilon^{m(n-i)}_2 (\varepsilon'_2)^{m(i-1)}:
G_{\Q_{p^2}} \rightarrow \GL_{n}(\overline{\Z}_p).\end{equation}
We assume that $p > nm$, so the representation~$\rho_{n,m,0}$ is Fontaine--Laffaille.
If the value of~$n$ is implicit, we often simply write~$\rho_0$ for~$\rho_{n,1,0}$. 
\end{df}

\begin{lemma} \label{prep} Let $K_0 / \Q_{p^2}$ be an unramified extension and let~$\rho: G_{K_0} \rightarrow \GL_{n}(\overline{\Z}_p)$
be any crystalline representation of Hodge--Tate weights $\{ 0, m, 2m, \dots, (n-1)m \}$ \emph{(}with respect to any embedding $K_0 \to \Qpbar$\emph{)}
such that~$\rhobar|_{I_{K_0}}=\rhobar_{n, m, 0}|_{I_{K_0}}$.
Then:
\begin{enumerate}
\item\label{item: finite unramified makes them the same} There is a  finite unramified 
extension~$K_1 / K_0$ such that $\rhobar  |_{G_{K_1}} = \rhobar_{n, m, 0} |_{G_{K_1}}$.
\item For any finite extension~$K/ K_0$ such that $\rhobar  |_{G_{K}} = \rhobar_{n, m, 0} |_{G_{K}}$, 
we have~$ \rho |_{G_K} \sim  \rho_{n, m, 0} |_{G_K}$ \emph{(}``connects to'', in the sense of \cite[\S1.4]{BLGGT}\emph{)}.
\end{enumerate}
\end{lemma}

\begin{proof} The first claim is clear. For the second, choose $K_1 / K_0$ minimal such that $\rhobar|_{G_{K_1}} = \rhobar_{n,m,0}|_{G_{K_1}}$. Since $p > nm$ and $K_1$ is unramified, the lifting ring $R^{\crys,\{ 0, \dots, (n-1)m \},\cO}_{\rhobar|_{G_{K_1}}}$ is formally smooth by Fontaine--Laffaille theory. It follows that~$ \rho |_{G_{K_1}} \sim  \rho_0|_{G_{K_1}}$,
and then these representations are still connected after passing to any further finite extension.
\end{proof}

\begin{lemma}\label{prep2}
Let $K$ be a finite extension of $\Q_p$.  Let $\rho_1,\rho_2$ be
ordinary, crystalline weight~$0$ representations of $G_K$ with
$\rhobar_1=\rhobar_2$ the trivial representation.  Then $\rho_1\sim\rho_2$.
\end{lemma}
\begin{proof}
This follows immediately from \cite[Lemma 3.14]{ger} --- the ordinary weight $0$ crystalline lifting ring of the trivial representation is irreducible.
\end{proof}

 \begin{lemma} \label{smooth} Let~$K$ be a finite extension of~$\Q_p$,
   and let~$\rho: G_{K} \rightarrow \GL_{n}(\Zpbar)$ be crystalline of
 weight~$0$. %
Then  there exists a constant~$c = c(K,\rho,n)$ 
with the following property:
\begin{itemize}\item if~$t: G_K \rightarrow \GL_{n}(\Zpbar)$ is
  crystalline of weight~$0$,
and~$t \equiv \rho|_{G_K} \bmod p^c$, then~$t \sim \rho|_{G_K}$.\end{itemize}
\end{lemma}

\begin{proof} Up to conjugation, the image of~$\rho$  lands in~$\GL_n(\OL_E)$
for some finite extension~$E/\Q_p$ with residue field~$k$.
Let~$R = R^{\crys,\underline{0}}_{\rhobar|_{G_{K}}} \otimes_{W(k)} \OL_E$
  denote the weight~$0$ crystalline lifting ring of~$\rhobar |_{G_K}: G_K \rightarrow \GL_n(k)$.
By assumption, $R$ has specializations corresponding to~$\rho$ and to~$t$.
We may choose a finite set of elements~$\{g_k: 1 \le k \le d\}$ of~$G_K$
such that \[R = \OL_E\llbracket X_{ijk} : 1 \le i,j \le n, 1\le k \le d
  \rrbracket/I\] for an ideal $I$, and the universal
lifting~$\rho^{\univ}:G_K \to \GL_n(R)$ of $\rhobar$
satisfies \[\rho^{\univ}(g_k) = \rho(g_k)  + [X_{ijk}]_{i,j=1}^{n},\]
so that $\pp = (X_{ijk})$ is
the dimension one prime associated to~$\rho$. The condition that~$t \equiv \rho|_{G_K} \bmod p^c$
is then equivalent to the condition that the corresponding homomorphism $t: R \to \Zpbar$ satisfies $v_p(t(X_{ijk})) \ge c$ for all $i,j,k$.

The generic fibre of~$R$ is formally smooth at~$\pp$ 
by~\cite[Theorem~3.3.8]{kisindefrings}, and so in particular there is a unique
minimal prime~$\PP$ of~$R[1/p]$ contained in the prime~$\pp$.
Suppose that~$\QQ$ is any minimal prime ideal of~$R[1/p]$
which is not contained in~$\pp$. Then~$\QQ$ contains an element~$P(X_{ijk}) \in R[1/p]$
which doesn't vanish at~$X_{ijk} = 0$ and hence has a non-zero constant term. After scaling
if necessary, we may assume that~$P \in R$. But now any
specialization of~$P$ with~$v_p(X_{ijk}) > v_p(P(0,0,\ldots,0))$ for every~$(i,j,k)$ will be non-zero,
and hence, if~$c > v_p(P(0,0,\ldots,0))$, then~$t$ cannot lie  on the irreducible component
corresponding to~$\QQ$.   Since~$R$ has only finitely many minimal prime ideals (it is Noetherian),  there exists a choice of~$c$ which guarantees that~$t$ lies on the component
corresponding to~$\PP$.
\end{proof}

\subsection{Lemmas on big image conditions}

In order to apply Theorem~\ref{thm:main_automorphy_lifting_theorem} to
a $p$-adic  representation of~$G_F$, one needs first to establish that the image of the residual representation
(and its restriction to~$G_{F(\zeta_p)}$) satisfies certain technical hypotheses,
in particular conditions~(\ref{part: ALT cond dgi}) and~(\ref{part:scalartoremark}). 
In this section, we prove some lemmas showing that a number of representations of a form
we shall encounter later have these properties. We first combine these conditions into the following
definition:

\begin{df} \label{tw} Say that a representation~$\sbar: G_F \rightarrow \GL_n(\Fbar_p)$
satisfies the Taylor--Wiles big image conditions if
the following hold:
\begin{enumerate}
\item  \label{generic} The representation~$\sbar$ is decomposed generic.
\item  \label{adequate}
The representation~$\sbar |_{G_{F(\zeta_p)}}$ has adequate image.
\item  \label{whatever}
There exists~$\sigma \in G_F  -  G_{F(\zeta_p)}$ such
that~$\sbar(\sigma)$ is  scalar.
\end{enumerate}
\end{df}

We have:

\begin{lemma} \label{restricttw}
Suppose that~$\sbar: G_F \rightarrow \GL_n(\Fbar_p)$ satisfies the Taylor--Wiles
big image conditions.  Suppose that~$F/\Q$ is Galois.
Let~$H/F$ be a finite extension whose Galois closure over~$\Q$  is linearly disjoint over~$F$
from the composite of~$F(\zeta_p)$ and the Galois closure over~$\Q$ of the fixed field of~$\ker(\sbar)$.
Then~$\sbar |_{G_H}$ satisfies the Taylor--Wiles
big image conditions.
\end{lemma}

\begin{proof} 
Let~$\wH$ be the Galois closure of~$H$ over~$\Q$. Since~$\sbar |_{G_{H}}$ satisfies
the Taylor--Wiles conditions if~$\sbar |_{G_{\wH}}$ does, we assume that~$H = \wH$ is Galois over~$\Q$.
The conditions ensure that the images of~$\sbar$ and~$\sbar |_{G_H}$
coincide, and also the images of~$\sbar |_{G_{F(\zeta_p)}}$ and~$\sbar |_{G_{H(\zeta_p)}}$
coincide. Thus condition~(\ref{adequate}) of Definition~\ref{tw} holds. 
Let~$M$ be the Galois closure of the fixed field of~$\ker(\sbar)$.
Then we have an isomorphism
$$\Gal(H \compositum M(\zeta_p)/F) \simeq \Gal(M(\zeta_p)/F) \times \Gal(H/F),$$
and so~$ \Gal(M(\zeta_p)/F) \simeq \Gal(H \compositum M(\zeta_p)/H)$ via the map~$\sigma \rightarrow (\sigma,1)$.
Moreover,
$$\Gal(H \compositum M(\zeta_p)/\Q) \subset \Gal(M(\zeta_p)/\Q) \times \Gal(H/\Q)$$
is the subgroup of elements whose projection to~$\Gal(F/\Q)\times\Gal(F/\Q)$ is the diagonal.
There exists a conjugacy class~$\langle \sigma \rangle \in \Gal(M(\zeta_p)/F)$
such that any rational prime unramified in~$H \compositum M(\zeta_p)$ whose Frobenius element corresponds
to~$\sigma$ is decomposed generic for~$\sbar$. Then~$(\sigma,1)$ will be decomposed
generic for~$\sbar |_{G_H}$. Similarly, if~$\sigma \in \Gal(M(\zeta_p)/F)  -  \Gal(M(\zeta_p)/F(\zeta_p))$
is an element such that~$\sbar(\sigma)$ is scalar, then the same is true of~$(\sigma,1) \in \Gal(H \compositum M(\zeta_p)/H)$.
\end{proof}

We shall also use the following group-theoretic fact.

\begin{lemma} \label{steinberg} Consider the collection of  groups~$G$ either of the form~$\PSL_n(\F_{p^k})$
or of the form~$\PSU_n(\F_{p^{2k}})$ for all primes~$p$ and integers~$k \ge 1$, $n \ge 2$.
Then~$G$ is simple unless~$(n,p) \in \{(2,2), (2,3), (3,2)\}$. 
  These groups are all pairwise mutually non-isomorphic as~$n$ and~$p$ both vary
except for the following isomorphisms:
 \begin{enumerate}
 \item $\PSL_2(\F_{p^k}) \simeq \PSU_2(\F_{p^{2k}})$, 
 \item $\PSL_2(\F_5) \simeq \PSL_2(\F_4)$,
 \item $\PSL_2(\F_7) \simeq \PSL_3(\F_2)$.
 \end{enumerate}
 If we restrict~$G$ to be of the form~$G = \PSL_2(\F_{p^k})$ or~$\PSU_n(\F_{p^{2k}})$,
and~$A \in G$ is the image of any matrix with eigenvalues in~$\F_p$,
then any automorphism of~$G$  preserves these eigenvalues up to scalar.
\end{lemma}

\begin{proof}  
If~$n=2$, then there is an isomorphism~$\PSU_2(\F_{p^{2k}}) \simeq \PSL_2(\F_{p^k})$. Otherwise,
$\PSL_n(\F_{p^k}) \simeq A_{n-1}(p^k)$ and~$\PSU_{n}(\F_{p^{2k}}) \simeq {}^2 A_{n-1}(p^{2k})$ is the Steinberg group. These groups are (twisted  in the second case) Chevalley groups. The simplicity statement
follows from~~\cite[Thm~37(b)]{MR0466335} (see also~\cite{MR0407163}).
The list of exceptional isomorphisms between
(possibly twisted) simple Chevelley groups was determined
in~\cite[Thm~37(a)]{MR0466335}.

By a theorem of Steinberg~\cite{MR121427} (and~\cite[Thm~12.5.1]{MR0407163}),
the outer automorphism group of either~$\PSL_n(\F_{p^k})$ or~$\PSU_{n}(\F_{p^{2k}})$ is generated by diagonal automorphisms (conjugation by diagonal elements),
by field automorphisms (acting on~$\F_{p^k}$ or~$\F_{p^{2k}}$ respectively), and the graph automorphism coming from the automorphism
of the  Dynkin diagram~$A_{n-1}$ if~$n > 2$ (associated to the inverse transpose map).
Certainly diagonal automorphisms preserve eigenvalues and field automorphisms act on the eigenvalues
and so preserve eigenvalues in~$\F_p$. There are no graph automorphisms for~$n=2$.
For~$n > 2$,
the graph automorphism
$M \mapsto (M^t)^{-1} = \sigma (M^{\dagger})^{-1} = \sigma M$ in the unitary group coincides with the field automorphism,
and so also preserves rational eigenvalues. 
\end{proof}

\begin{lemma} \label{yeomanslemmanew}  Let~$F/\Q$ be Galois.
Consider representations:
$$\begin{aligned}
\rbar_A:  G_F &  \rightarrow \GL_2(\F_p), \\
\rbar_B:  G_F  & \rightarrow \GU_m(\Fps) \rightarrow \GL_{m}(\Fps), \end{aligned}
$$
such that the images~$\rbar_A(G_{F(\zeta_p)})$ and~$\rbar_B(G_{F(\zeta_p)})$
equal~$\SL_2(\F_p)$ and~$\SU_m(\Fps)$ respectively.
Consider the representation:
$$\sbar =  \Sym^{n-1} \rbar_A \otimes  \rbar_B: G_F \rightarrow \GL_{mn}(\Fps)$$
Assume that~$p>2mn+1$, and if~$m=2$ assume that the fixed fields
of the kernels of the projective representations associated to~$\rbar_A$ and~$\rbar_B$
are linearly disjoint over~$F(\zeta_p)$.
\begin{enumerate}
\item The representation~$\sbar$ satisfies conditions~(\ref{generic}) and~(\ref{adequate}) of
Definition~\ref{tw}.
\item \label{isselfdual} If~$\det(\rbar_A) = \varepsilonbar^{-m}$ and~$\rbar_B$ has multiplier character~$\varepsilonbar^{1 - m}$,
then~$\sbar$ has image in~$\GU_{mn}(\Fps)$ with multiplier character~$\varepsilonbar^{1 - mn}$.
\item If in addition to the assumptions in~(\ref{isselfdual}), one additionally assumes that~$\varepsilonbar(G_F)=\F^{\times}_p$, 
then~$\sbar$ also satisfies  condition~(\ref{whatever}) of Definition~\ref{tw} and thus satisfies the Taylor--Wiles big image conditions.
\end{enumerate}
\end{lemma}

\begin{proof} 
Let~$H_A$ and~$H_B$ denote the extensions of~$F(\zeta_p)$ corresponding to the fixed fields
of the kernels of the projective representations associated to~$\rbar_A$ and~$\rbar_B$.
Our assumption on the images of~$\rbar_A$ and~$\rbar_B$ imply that
$$\Gal(H_A/F(\zeta_p)) \simeq \PSL_2(\F_p), \quad \Gal(H_B/F(\zeta_p))
\simeq \PSU_m(\F_{p^2}).$$
Let~$\wH_A$ and~$\wH_B$ denote the Galois closures of~$H_A$ and~$H_B$ over~$\Q$,
and let~$\wH$ denote the compositum of~$\wH_A$ and~$\wH_B$.
Since~$F/\Q$ is Galois and~$\PSL_2(\F_p)$ and~$\PSU_n(\Fps)$ are simple (as~$p \ge 5$),
we have isomorphisms~$\Gal(\wH_A/F(\zeta_p)) \sim \PSL_2(\F_p)^r$  and~$\Gal(\wH_B/F(\zeta_p)) \simeq \PSU_m(\Fps)^s$ respectively for some positive integers~$r$ and~$s$.
Thus
$$\Gal(\wH/F(\zeta_p)) \simeq \Gal(\wH_A/F(\zeta_p)) \times \Gal(\wH_B/F(\zeta_p)),$$
since either~$m > 2$ and the groups have no common quotients, or~$m=2$
and the fields are linearly disjoint by assumption.

If~$G = \PSL_2(\F_p)$ or~$G = \PSU_m(\Fps)$, then~$G$ is simple by Lemma~\ref{steinberg}.
Moreover, by the same lemma,  for any equivalence class of matrices~$A$
with eigenvalues in~$\F_p$,
any automorphism of~$G$
preserves the (unordered set of) eigenvalues of any member of $A$ up to scalars.
We also have~$\Aut(G^m) \simeq \Aut(G) \rtimes S_m$.

 Let~$\alpha  = \beta^2 \in \F^{\times 2}_p$ be an element such that~$1,\alpha,\ldots,\alpha^{mn-1}$ are all distinct; such an~$\alpha$ exists because~$p - 1 \ge 2mn$.
 Let~$A$  be a matrix in~$\PSL_2(\F_p)$ with eigenvalues (up to
scalars) $1$ and~$\alpha$, and let~$B \in \PSU_m(\Fps)$ have
eigenvalues~$1,\alpha^n,\ldots,\alpha^{n(m-1)}$. 
Explicitly, let~$A = \diag(\beta,\beta^{-1}) \in \SL_2(\F_p)$, and then construct~$B$
as follows. Certainly~$A^n \in \SL_2(\F_p)$. 
The image of~$A^n$ under the~$m-1$th symmetric power
map lands in~$\Sp_{m}(\F_p)$ or~$\SO_m(\F_p)$ depending on the parity of~$m$ (here we mean the symplectic or orthogonal groups defined using the bilinear form induced by the standard symplectic form on $\F_p^2$). These groups are conjugate to subgroups of~$\SU_m(\Fps)$ by Lemma~\ref{selfdual}. 
By Chebotarev, we find a prime~$q$ unramified in~$\wH$ and such that
for a fixed choice of $\mathfrak{q}|q$ in $\wH$,~$\Frob_{\q}\in \Gal(\wH/F(\zeta_p)) \subset 
\Gal(\wH/\Q)$ 
has the form
$$(A,A,\ldots,A) \times (B,B,\ldots,B) \in \PSL_2(\F_p)^r \times \PSU_m(\Fps)^s.$$
The eigenvalues (up to scalar) of~$A$ and~$B$ are preserved by the action of~$\Gal(F/\Q)$; this
follows from our description
of the automorphism group  of each factor.

The images of these elements in~$\Gal(F(\zeta_p)/\Q)$ are trivial, so 
 such a prime $q$ will split completely in~$F(\zeta_p)$ and so satisfy~$q \equiv 1 \bmod p$.
Moreover, the Frobenius elements at all other primes above~$q$ will be conjugate inside~$\Gal(\wH/\Q)$.
Hence the image of (any conjugate of) $\Frob_{\q}$ under~$\sbar$ has
eigenvalues (up to scalar) given by~$1, \alpha, \ldots, \alpha^{mn-1}$. In particular, they are all distinct.
Since~$q \equiv 1 \bmod p$, this implies that~$\sbar$ is decomposed
generic, which is property~(\ref{generic}) of Definition~\ref{tw}.

To see that~$\sbar |_{G_{F(\zeta_p)}}$ has adequate image, it suffices to show that the image 
is absolutely irreducible and thus is also adequate by~\cite[Theorem~A.9]{jack} (using
the assumption~$p > 2mn+1$). The irreducibility follows from the fact that the $\SL_2(\Fp)$-representation $\Sym^{n-1}\Fpbar^2$ and the standard representation of $\SU_m(\Fps)$ are both irreducible as long as~$p > n$, since the image of $\sbar |_{G_{F(\zeta_p)}}$ is $\SL_2(\Fp)\times\SU_m(\Fps)$.  This proves property~(\ref{adequate})of Definition~\ref{tw}.

Assume that~$\det(\rbar_A) = \varepsilonbar^{-m}$ and that~$\rbar_B$ has multiplier character~$\varepsilonbar^{1-m}$.
Then~$\rbar_A \otimes \rbar_B$ is absolutely irreducible and self-dual (i.e.~there is an isomorphism of the form \eqref{twistselfdual}) with multiplier character
$$\varepsilonbar^{-m(n-1)} \cdot \varepsilonbar^{1-m} = \varepsilonbar^{1-mn},$$
and so the image lies in~$\GU_{mn}(\Fps)$ with this multiplier character by Lemma~\ref{selfdual}.
This establishes condition~(\ref{isselfdual}).

Assume that~$\varepsilonbar(G_F)=\F^{\times}_p$.
Let~$M_A$ and~$M_B$ denote the fixed fields of the kernels of
$\rbar_A$ and $\rbar_B$, and let~$M$ be the compositum of~$M_A$ and~$M_B$. By our assumption on linear disjointness of
$H_A$ and $H_B$, $\Gal(M/F(\zeta_p))$ is the direct
product~$\SL_2(\F_p) \times \SU_m(\Fps)$, 
and~$\Gal(M/F)$ is the subgroup of matrices~$(A,B)$ of~$\GL_2(\F_p) \times \GU_m(\Fps)$ 
with~$\det(A) = \eta^{-m}$ and~$\nu(B) = \eta^{1-m}$ %
   for some~$\eta \in \F^{\times}_p$.
Hence the image certainly contains~$(\beta^{m} I_2,\beta^{m-1} I_m)$, where~$I_n$ denotes the trivial
matrix in~$\SL_n(\F_p)$ and~$\beta \in \F^{\times}_p$ is a primitive root.
Then, by Chebotarev, there exists~$\sigma \in G_F$
whose image in~$\Gal(M/F)$ is this element. Since~$p > 2mn+1 \ge 2m+1$, we have~$\beta^{2m} \ne 1$.
Since~$\varepsilonbar^{-m}(\sigma) = \beta^{2m}$,   the element~$\sigma$ is not contained
in~$G_{F(\zeta_p)}$. On the other hand,  we see that~$\sbar(\sigma)$ is also
scalar,  and we are done.
\end{proof}

We shall need  the following well-known property of  induced representations
(specialized to the context in which we shall apply it in the proof of the following lemma).
\begin{lemma} \label{useme} Let~$E/\Q$ be a cyclic
Galois extension of degree~$m$ linearly disjoint from~$F$, and let~$L = E \compositum F$.
Let~$\psibar: G_L \rightarrow \F^{\times}_p$ be a character
and let~$\rbar_B = \Ind^{G_F}_{G_L} \psibar: G_F \rightarrow \GL_m(\Fbar_p)$.
Let~$q$ be a prime of~$\Q$  such that~$\rbar_B$ is unramified at all~$v|q$, $q$
splits completely in~$F$, and~$\Frob_q$ generates~$\Gal(E/\Q)$.
Then, for~$v|q$ in~$F$, the eigenvalues of~$\rbar_B(\Frob_v)$ are of the form~$\lambda, \zeta \lambda, \ldots, \zeta^{m-1} \lambda$ 
for some~$\lambda$ where~$\zeta$
is a primitive~$m$th root of unity.
\end{lemma}

\begin{proof} The assumption that~$E$ is linearly disjoint from~$F$ ensures
that~$\Gal(L/F) \simeq \Gal(E/\Q)$ is cyclic of order~$m$.
There is an isomorphism~$\rbar_B \simeq \rbar_B \otimes \chi$
where~$\chi$ is a character (factoring through~$\Gal(L/F)$) of order~$m$.

Thus~$\rbar_{B}(\Frob_v)$ is conjugate to $\chi(\Frob_v) \rbar_{B}(\Frob_v) = \zeta  \cdot \rbar_{B}(\Frob_v)$,
where~$\zeta$ is a primitive~$m$th root of unity and the result follows.
\end{proof}

We  shall also need the following variant of Lemma~\ref{yeomanslemmanew}:

\begin{lemma} \label{yeomanslemmanewvariant}  Let~$F/\Q$ be Galois.
Consider representations:
$$\begin{aligned}
\rbar_A:  G_F &  \rightarrow \GL_2(\F_p), \\
\rbar_B:  G_F  & \rightarrow \GL_{m}(\F_p). \end{aligned}
$$
Assume that:
\begin{enumerate}
\item The image of~$\rbar_A(G_{F(\zeta_p)})$ 
equals~$\SL_2(\F_p)$.
\item There is a cyclic Galois
extension $E / \Q$ of degree~$m$ and linearly disjoint from~$F$, such that, setting $L = E \compositum F$, there is a character $\overline{\psi} : G_L \to \F_p^\times$ with $\rbar_B \cong \Ind_{G_L}^{G_F} \overline{\psi}$ and $\rbar_B|_{G_{F(\zeta_p)}}$ irreducible.
\end{enumerate}
Consider the representation:
$$\sbar =  \Sym^{n-1} \rbar_A \otimes  \rbar_B: G_F \rightarrow \GL_{mn}(\F_p).$$
Assume that~$p>2mn+1$. 
Then:
\begin{enumerate}
\item The representation~$\sbar$ satisfies conditions~(\ref{generic}) and~(\ref{adequate}) of
Definition~\ref{tw}.
\item If~$\det(\rbar_A) = \varepsilonbar^{-m}$ and~$\det(\rbar_B) = \varepsilonbar^{-m(m-1)/2}$
and the image of~$\varepsilonbar(G_L)=\F^{\times}_p$, 
then~$\sbar$ also satisfies  condition~(\ref{whatever}) of
Definition~\ref{tw} and thus satisfies the Taylor--Wiles big image conditions.
\end{enumerate}
\end{lemma}

\begin{proof}
The representation $\rbar_B$ has solvable image.  The assumption that~$\rbar_B|_{G_{F(\zeta_p)}}$ is irreducible implies that~$F(\zeta_p)$ and~$E$ are linearly disjoint.
As in the proof of Lemma~\ref{yeomanslemmanew},
let~$H_A$ and~$H_B$ denote the extensions of~$F(\zeta_p)$ corresponding to the fixed fields
of the kernels of the projective representations associated to~$\rbar_A$ and~$\rbar_B$ and $\wH_A$, $\wH_B$ their Galois closures over $\Q$. Let~$\wH$ be the compositum of~$\wH_A$ and~$\wH_B$.
We deduce  once more that
$$\Gal(\wH_A/F(\zeta_p)) \simeq \PSL_2(\F_p)^r$$
and, since~$\PSL_2(\F_p)$ has no solvable quotients, 
$$\Gal(\wH/F(\zeta_p)) \simeq \Gal(\wH_A/F(\zeta_p)) \times \Gal(\wH_B/F(\zeta_p)).$$
We have $E \subset \wH_B$ and, since $\rbar|_{G_{F(\zeta_p)}}$ is irreducible, $\Gal(\wH_B/F(\zeta_p)) \to \Gal(E / \Q)$ is surjective.

 Let~$\alpha \in \F^{\times}_p$ be an element such that~$1,\alpha,\ldots,\alpha^{mn-1}$ are all distinct; such an~$\alpha$ exists because~$p - 1 \ge mn$.
By Chebotarev, we find a prime~$q$ unramified in~$\wH$, split in $F(\zeta_p)$, and such that for a fixed choice of $\mathfrak{q}|q$ in $\wH$,~$\Frob_{\mathfrak{q}}\in \Gal(\wH/F(\zeta_p))  \subset \Gal(\wH/\Q)$ 
has the form
$$(A,A,\ldots,A) \times \sigma \in \PSL_2(\F_p)^r \times  \Gal(\wH_B/F(\zeta_p)),$$
where $A$ has eigenvalues with ratio~$\alpha$ and $\sigma$ projects to a generator of $\Gal(E / \Q)$.
Hence, by Lemma~\ref{useme}, the image of (any conjugate of) $\Frob_{\q}$ under~$\sbar$ has
eigenvalues (up to scalar) given by:
$$\alpha^i \zeta^j, i = 0,\ldots,n-1, j = 0,\ldots,m-1,$$
and in particular all eigenvalues are distinct
since otherwise~$\alpha^{km} = 1$ for some~$k < n$.
Since~$q \equiv 1 \bmod p$, this  implies that~$\sbar$ is decomposed
generic, which is property~(\ref{generic}).

Property~(\ref{adequate}) of Definition~\ref{tw} follows exactly as in the proof
of Lemma~\ref{yeomanslemmanew}.

Now suppose that~$\det(\rbar_A) = \varepsilonbar^{-m}$ and~$\det(\rbar_B) = \varepsilonbar^{-m(m-1)/2}$ and that~$\varepsilonbar(G_L)=\F^{\times}_p$. We
now show that property~(\ref{whatever}) of Definition~\ref{tw} holds.
Let~$M_A$ and~$M_B$ denote the fixed fields of the kernels of
$\rbar_A|_{G_{F(\zeta_p)}}$ and $\rbar_B|_{G_{F(\zeta_p)}}$ respectively. Since $\SL_2(\F_p)$ has no solvable quotients, the map $G_{M_B} \to \Gal(M_A / F(\zeta_p))$ is surjective.
Let $\beta \in \F_p^\times$ be a primitive root. We claim that we can
find $\sigma \in G_F$ such that $\rbar_A(\sigma) = \beta^{-m^2} \cdot
I_2$ and $\rbar_B(\sigma) = \beta^{m(1-m)} \cdot I_m$. To see this,
first choose $g \in G_L$ such that $\varepsilonbar(g) = \beta^2$, and
let $h = \prod_{\tau \in \Gal(L / F)} {}^\tau g $. Then $\rbar_B(h) =
\beta^{m(1-m)} I_m$, as $\prod_{\tau \in \Gal(L / F)} {}^\tau
\overline{\psi} = \det \rbar_B|_{G_L} =
\varepsilonbar^{-m(m-1)/2}$. %
  We now choose~$\sigma$ of the form~$h \gamma$ where~$\gamma \in G_{M_B}$;
  since~$\rbar_B(\gamma)$ is trivial, this means that~$\rbar_{B}(\sigma) = \rbar_{B}(h)$
  is of the correct form.
On the other hand, we have $\det \rbar_A(h) = \varepsilonbar^{-m}(g)^{m} = \beta^{-2m^2}$. Since $G_{M_B} \to \Gal(M_A / F(\zeta_p)) \simeq \SL_2(\F_p)$ is surjective, we choose~$\gamma \in G_{M_B}$
so that~$\rbar_A(\gamma) = \beta^{-m^2} \cdot \rbar_A(h)^{-1} $, and
then~$\rbar_A(\sigma) =  \beta^{-m^2} \cdot I_2$.

By construction, $\sbar(\sigma)$ is scalar. On the other hand, $\varepsilonbar(\sigma) = \beta^{2m} \neq 1$, as $p-1 > 2m$ because $p > 2nm + 1$, so $\sigma \in G_F - G_{F(\zeta_p)}$, as required. 
\end{proof}

\subsection{Character building lemmas}

In this section, we construct some induced extensions with certain desirable local properties.
We begin with the following  well-known lemma:

\begin{lemma}[Globalizing local characters] \label{realize} Let~$F$ be a number field, and let~$S$ be a finite set of places of~$F$. 
Let~$\psi_v: G_{F_v} \rightarrow \Z/n\Z$ be a collection of characters
for all~$v \in S$. Assume that 
$S$ does not contain any places~$v|2$.
Then
there exists a global character~$\chi: G_F \rightarrow \Z/n\Z$ such that
$\chi|_{G_{F_v}} = \psi_v$  for all $v \in S$.
 \end{lemma}

\begin{proof} This is a consequence of~\cite[\S X Thm.\ 5]{ArtinTate} (see also~\cite[Appendix~A]{conradlifting}).
More precisely, the claim holds
 (without the hypothesis on~$S$) if~$n$ is odd. If~$n$ is even,
there exists an explicitly defined element
$$a_{F,n} \in (F^{\times})^{n/2}$$
which is a perfect~$n$th power for all but a finite set
(possibly empty) places~$S_{F,n}$ of primes~$v|2$. %
Then  the~$\psi_v$ come from a global character~$\chi$ of order~$n$
if and only  if either ~$S_F \not\subset S$, or   ~$S_F \subset S$ and 
$$\prod_{v \in S_F} \psi_v(a_{F,n}) = 1.$$
Since we have assumed that~$S$ contains no places above~$2$,
either~$S_F$ is empty or~$S_F \not\subset S$, so the result follows.
\end{proof}

\begin{rem}
One cannot drop the hypothesis on~$S$ in general because of the Grunwald--Wang
phenomenon (e.g. $F = \Q$, $v = 2$, $n = 8$, and~$\psi_2$ unramified with order~$8$).
 If one considers general~$F$, one cannot even  globalize a
local character~$\psi_v$
up to a character~$\phi_v$ which is unramified at~$v$. Let~$F = \Q(\sqrt{-5})$,
and consider a character
$$\chi: F^{\times} \backslash \A^{\times}_F \rightarrow \Z/8\Z.$$
Since~$16 \in F^{\times}$ is a perfect~$8$th power for all~$v|F$ of odd residue characteristic
(this is true even for~$F=\Q$) and also in~$F^{\times}_{\infty} \simeq \C^{\times}$, it follows that the restriction:
$$\chi_2: F^{\times}_2 \rightarrow \Z/8\Z$$
satisfies~$\chi_2(16) = 1$. Since~$F_2/\Q_2$ is ramified of degree~$2$,
we find that~$16$ has valuation~$8$ and hence~$\phi_2(16) = 1$ for any unramified character ~$F^{\times}_2\rightarrow \Z/8\Z$.
Since neither~$2$ nor~$-1$ is a square in~$F_2 \simeq \Q_2(\sqrt{-5})$, it follows that~$16$ is not
a perfect~$8$th power in~$F^{\times}_2$, and thus there  exists a local character~$\psi_2: F^{\times}_2 \rightarrow \Z/8\Z$
such that~$\psi_2(16) = -1$. But from the above, we see that there is no global character~$\chi$ such that~$\chi_2 = 
\psi_2 \phi_2$ for an unramified character~$\phi_2$.
\end{rem}

We now show that any character over a CM field can be written as an~$m$th power of another
character over some finite CM extension satisfying certain properties.
The argument is essentially the same as in the proof
        of~\cite[Theorem~7.1.11]{10author}.

 \begin{lemma} \label{fixdeterminant} Let~$\eta$ be a finite order character of $G_F$ for a CM field~$F$, let~$m$ be an integer,
 and let~$F^{\avoid}/F$ be a finite extension.
 Then there exists a totally real Galois extension~$M/\Q$ linearly
 disjoint from~$F^{\avoid}$ and a character~$\psi$ of~$G_{M \compositum F}$ such that
 ~$\eta|_{G_{MF}}=\psi^m$.
 Moreover, if~$\eta$ is unramified at all~$v$ dividing some finite set of primes~$T$ of~$\Q$
 not including~$2$, then we may take~$M$ to be totally split at all primes dividing those in~$T$, and~$\psi$ to be unramified at primes dividing those in~$T$.
 \end{lemma}
 
 \begin{proof} By induction, it suffices to consider the case when~$m$ is prime. Assume that~$\eta$ has order~$n$. There is an exact sequence:
  $$0 \rightarrow \Z/m \Z \rightarrow \Z/mn\Z \rightarrow \Z/n\Z \rightarrow 0.$$
 The character~$\eta$ gives a class in~$H^1(F,\Z/n\Z)$ which we want to write as an~$m$th power,
which amounts to lifting this class to~$H^1(F,\Z/mn\Z)$.
The obstruction to this is an element $\partial\eta$ lying in
 $$H^2(F,\Z/m\Z) \hookrightarrow H^2(F(\zeta_m),\mu_m)) \simeq \Br(F(\zeta_m))[m],$$
 where the injectivity of the first map follows from Hochschild--Serre and the fact that~$[F(\zeta_m):F]$ is prime to~$m$
 (since~$m$ is prime). From the  Albert--Brauer--Hasse--Noether theorem, there is an injection
 $$ \Br(F(\zeta_m))[m] \hookrightarrow \bigoplus_{v} \Br(F(\zeta_m)_v)[m].$$
 The image of the class $\partial\eta$ is zero for all~$v$ not dividing a finite set of places~$S$ of~$\Q$ (the places where $\eta$ is ramified), and is zero for $v$ dividing places in $T$ (since $\eta$ is unramified there and there is no obstruction to lifting an unramified local character). Since~$F$ is totally imaginary and~$\Br(\C) = 0$, we may assume~$S$ consists only of finite primes and we may also assume that~$\infty \in T$.
 If~$K$ is a local field and~$L/K$ has degree~$m$ then the map~$\Br(K)[m] \rightarrow \Br(L)$
 is trivial~\cite[\S VI, Thm.~3]{MR911121}. Hence any class in~$\Br(F(\zeta_m))[m]$
 is trivial in~$F(\zeta_m) \compositum M$ whenever~$[M_v:\Q_v]$ is divisible
by~$m [F(\zeta_m):\Q]$ for any prime~$v$ in~$S$.
Hence it suffices to find such a Galois extension~$M/\Q$ disjoint from~$F^{\avoid}$, in which the places in $T$ are totally split (since $\infty\in T$ this implies that $M$ is totally real). This is essentially done in \cite[\S X Thm.~6]{ArtinTate} and we can appeal to \cite[Lemma 4.1.2]{cht} for the precise statement we need.
Since~$\eta$ is unramified at primes in~$T$, for each~$v|T$ the image of~$\psi|_{I_{(M \compositum F)_v}}$ has order dividing~$m$.
Thus by Lemma~\ref{realize} we may twist~$\psi$ by another character of order~$m$
(which doesn't change~$\psi^m$) so that it is unramified at~$v|T$.  \end{proof}

\section{Automorphy of compatible systems}\label{sec_automorphy_of_compatible_systems}

\subsection{Compatible systems and purity}

We recall the following definition from \cite[\S7]{10author}.
\begin{defn}
	Let $F$ be a number field. A very weakly compatible system $\cR$ (of rank $n$ representations of $G_F$, with coefficients in $M$) is by definition a tuple
	\[ (M, S, \{ Q_v(X) \}, \{ r_\lambda \}, \{H_\tau\}),\]
	where:
	\begin{enumerate}
	\item $M$ is a number field;
	\item $S$ is a finite set of finite places of $F$;
	\item for each finite place $v \not\in S$ of $F$, $Q_v(X) \in M[X]$ is a monic degree $n$ polynomial; 
	\item for each $\tau : F \hookrightarrow \overline{M}$, $H_\tau$ is a multiset of integers;
	\item for each finite place $\lambda$ of $M$ (say of residue characteristic $l$),
	\[ r_\lambda : G_F \to \GL_n(\overline{M}_\lambda) \]
	is a continuous, semi-simple representation satisfying the following conditions:
	\begin{enumerate}
	\item If $v \not \in S$ and $v \nmid l$, then $r_\lambda|_{G_{F_v}}$ is unramified and the characteristic polynomial of $\Frob_v$ equals $Q_v(X)$.
	\item For $l$ outside a set of primes of Dirichlet density 0, $r_\lambda$ is crystalline and $\mathrm{HT}_\tau(r_\lambda) = H_\tau$.
	\item For every $l$, we have $\mathrm{HT}_\tau(\det r_\lambda) = \sum_{h \in H_\tau} h$.
	\end{enumerate}
	\end{enumerate}
\end{defn}
If $F' / F$ is a finite extension then we may define the restricted very weakly compatible system
\[ \cR|_{G_{F'}} = (M, S_{F'}, \{ Q_w(X) \}, \{ r_\lambda|_{G_{F'}} \}, \{H_\tau' \}), \]
where $S_{F'}$ is the set of places of $F'$ lying above $S$, $Q_w(X) = \det r_\lambda(X - \Frob_w)$ (thus independent of $\lambda$), and $H'_{\tau} = H_{\tau|_F}$. If
\[ \cR_1 = (M, S_1, \{ Q_{1, v}(X) \}, \{ r_{1, \lambda} \}, \{H_{1, \tau}\}), \, \cR_2 = (M, S_2, \{ Q_{2, v}(X) \}, \{ r_{2, \lambda} \}, \{H_{2, \tau} \}) \]
are very weakly compatible systems with a common coefficient field $M$, then we can define the tensor product
\[ \cR_1 \otimes \cR_2 = (M, S_1 \cup S_2, \{ Q_{ v}(X) \}, \{ r_{1, \lambda} \otimes r_{2, \lambda} \}, \{H_{ \tau}\}), \]
where we take $Q_v(X) = \det (r_{1, \lambda} \otimes r_{2, \lambda})(X - \Frob_v)$ (thus independent of $\lambda$) and $H_\tau = \{ k + l \, | \, k \in H_{1 ,\tau}, l \in H_{2, \tau} \}$ (sums taken with multiplicity).

The following definition summarizes some possible properties of very weakly compatible systems. These were all defined in \cite{10author}, with the exception of (\ref{def_weak_aut}) (`weakly automorphic'). This condition arises for us because we consider tensor products of compatible systems, one of which has poorly controlled ramification. Lemma \ref{lem_weakly_automorphic_and_pure} gives conditions under which `weakly automorphic' can be upgraded to `automorphic'. 
\begin{defn}\label{def:vwcs props}
Let 
\[ \cR = (M, S, \{ Q_v(X) \}, \{ r_\lambda \}, \{H_\tau\}) \]
be a very weakly compatible system. We say that $\cR$ is:
\begin{enumerate} \item pure of weight $m \in \Z$, if it satisfies the following conditions:
	\begin{enumerate}
		\item for each $v \not\in S$, each root $\alpha$ of $Q_v(X)$ in $\overline{M}$, and each $\iota:\overline{M} \into \C$
		we have 
		\[ | \iota \alpha |^2 = q_v^m; \]
		\item for each $\tau:F \into \overline{M}$ and each complex conjugation $c$ in $\Gal(\overline{M}/\Q)$ we have
		\[ H_{c \tau} = \{ m-h: \,\,\, h \in H_\tau\}. \]
	\end{enumerate}
	\item automorphic, if there is a regular algebraic, cuspidal
          automorphic representation $\pi$ of $\GL_n(\A_F)$ and an
          embedding $\iota : M \hookrightarrow \C$ such that for every
          finite place $v \not\in S$ of $F$, $\pi_v$ is unramified and
          $\operatorname{rec}^T_{F_v}(\pi_v)(\Frob_v)$ has
          characteristic polynomial $\iota(Q_v(X))$. 
	\item\label{def_weak_aut} weakly automorphic of level prime to $T$, if $T$ is a
          finite set of finite places of~ $F$, disjoint from $S$, and
          there is a regular algebraic, cuspidal automorphic
          representation $\pi$ of $\GL_n(\A_F)$ and an embedding
          $\iota : M \hookrightarrow \C$ such that for all but
          finitely many finite places $v \not\in S$ of $F$, and for every
          $v \in T$, $\pi_v$ is unramified and
          $\operatorname{rec}^T_{F_v}(\pi_v)(\Frob_v)$ has
          characteristic polynomial $\iota(Q_v(X))$. We will say that $\RR$ is simply `weakly automorphic' if it is weakly automorphic of level prime to the empty set.
		\item irreducible, if for $l$ outside a set of primes of Dirichlet density 0, and for all $\lambda | l$ of $M$, $r_\lambda$ is irreducible.
	\item strongly irreducible, if for every finite extension $F' / F$, the compatible system $\cR|_{G_{F'}}$ is irreducible. 
	\end{enumerate}
\end{defn}

For a CM number field $F$, and a regular algebraic weight $\lambda$, cuspidal automorphic representation $\pi$ of $\GL_2(\A_F)$, there is an associated automorphic very weakly compatible system
\[ \RR=(M,S,\{ Q_v(X) \},
r_{\pi, \lambda}  ,H_{\tau} ),\]
where $H_\tau = \{ \lambda_{\tau, 1} + 1, \lambda_{\tau, 2} \}$ (see
\cite[Lemma 7.1.10]{10author}).

We now recall that the potential automorphy of symmetric 
powers of~$\cR$ is enough to imply purity.
\begin{lemma}\label{lem:pot aut implies purity}
Let~$\cR = (M, S, \{ Q_v(X) \}, \{ r_\lambda \}, \{H_\tau\})$ be a very weakly compatible system
of rank 2 representations of $G_F$ such that~$H_{\tau} = \{0,m\}$ for a fixed~$m \in \N$ for all~$\tau$.
	Fix a finite place $v_0$ of $F$ which is not in $S$, and let~$X_0 = \{v_0\}$.
	Suppose that for infinitely many~$n\ge 1$, we can find 
	a finite Galois extension  $F_n / F$ such that the very weakly compatible system $\Sym^{n-1} \RR |_{F_n}$ 
	 is weakly automorphic of level prime to $X_{0, F_n} = \{ v | v_0 \}$. 	 Then the roots $\alpha_1, \alpha_2$ of $Q_{v_0}(X)$ in $\overline{M}$ satisfy	\[ | \iota \alpha |^2 = q_{v_0}^m\] for each $\iota:\overline{M} \into \C$. 
\end{lemma}
\begin{proof}

Choose a place $v_n|v_0$ in $F_n$ and fix $\iota:\overline{M}\hookrightarrow\C$. We are assuming that $\Sym^{n-1} \RR |_{F_n}$ is associated to a cuspidal automorphic representation $\Pi$ of $\GL_{n}(\A_{F_n})$ and $\Pi_{v_n}$ is unramified. Up to a finite order character, the determinant of our rank $n$ automorphic compatible system is given by the~$-mn(n-1)/2$th power
 of the cyclotomic character, so the central character of $\Pi$ is (again, up to a finite order Hecke character)~$| \cdot |^{n(m-1)(1-n)/2}$,
 and in particular $\Pi | \cdot |^{(m-1)(n-1)/2}$ is unitary. 
Since we know that $|\iota(\alpha_1\alpha_2)| = q_{v_0}^m$, it suffices to prove that $|\iota\alpha_i|\le q_{v_0}^{m/2}$ for $i = 1, 2$. Let $q_{v_n} = q_{v_0}^f$.
	As in the proof of \cite[Cor.\ 7.1.13]{10author}, we can apply the Jacquet--Shalika bound \cite[Cor.\ 2.5]{jsajm103}
	to deduce that 
	that $|\iota\left(
	\alpha_i^{f(n-1)}\right)|\le q_{v_n}^{((m-1)(n-1) + n)/2}$, so $|\iota\alpha_i|\le
	q_{v_0}^{m/2+1/2(n-1)}$. Letting $n$ tend to $\infty$ gives the desired bound on $|\iota\alpha_i|$.
\end{proof}

\begin{lemma}\label{lem_weakly_automorphic_and_pure}
Let $F$ be a CM number field, and let 
\[ \cR = (M, S, \{ Q_v(X) \}, \{ r_\lambda \}, \{H_\tau\}) \]
 be a very weakly compatible system of rank $n$ representations of $G_F$ which is weakly automorphic, corresponding to a regular algebraic, cuspidal automorphic representation $\pi$ of $\GL_n(\A_F)$, and pure of weight $m \in \Z$. Then $\cR$ is automorphic.
\end{lemma}
\begin{proof}
Choose some embedding of $M$ in $\C$. By assumption, there is a finite
set $S'\supseteq S$ of finite places of $F$ such that for each $v \not
\in S'$, $\pi_v$ is unramified and
$\operatorname{rec}^T_{F_v}(\pi_v)(\Frob_v)$ has characteristic
polynomial $Q_v(X)$. We must show that this holds for all $v \not\in
S$. Choose $v \in S' - S$, a rational prime $p$ not lying under~$ v$, and an isomorphism $\iota : \overline{\Q}_p \to \C$. Let $\lambda$ denote the place of $M$ induced by $\iota^{-1}$. Then the Chebotarev density theorem implies that there is an isomorphism $r_\lambda \cong r_\iota(\pi)$. By assumption, $r_\lambda|_{G_{F_v}}$ is unramified and pure of weight $m$. By \cite[Theorem 1]{ilavarma}, there is an isomorphism $r_\iota(\pi)|_{W_{F_v}}^{\sss} \cong \iota^{-1} \operatorname{rec}^T_{F_v}(\pi_v)^{\sss}$. We deduce that $\pi_v$ is a subquotient of an unramified principal series, namely the one with Satake parameter determined by $Q_v(X)$. Since $r_\lambda|_{G_{F_v}}$ is pure, this principal series representation is irreducible and $\pi_v$ is unramified, as desired. 
\end{proof}
\begin{lemma}\label{lem_properties_of_strongly_irreducible_very_weakly_compatible_systems}
Let $F$ be a number field and let 
\[ \cR = (M, S, \{ Q_v(X) \}, \{ r_\lambda \}, \{H_\tau\}) \]
be a very weakly compatible system of rank $2$ representations of $G_F$ which is strongly irreducible. Let $\cL(\cR)$ denote the set of primes $l$ satisfying the following conditions:
\begin{enumerate}
\item $l \not\in S$ and for each place $\lambda | l$ of $M$, $r_\lambda$ is crystalline of Hodge--Tate weights $H_\tau$.
\item For each place $\lambda | l$ of $M$, $\overline{r}_\lambda(G_{\widetilde{F}})$ contains a conjugate of $\SL_2(\F_l)$, where $\widetilde{F}$ is the Galois closure of $F / \Q$.
\end{enumerate} 
Then $\cL(\cR)$ has Dirichlet density 1.
\end{lemma}
\begin{proof}
The set of primes $l$ having property (1) has Dirichlet density 1, by definition of a very weakly compatible system. The lemma therefore follows from \cite[Lemma 7.1.3]{10author}.
\end{proof}

\subsection{Potential automorphy theorems}
\label{subsec_proof_of_pot_aut}

Our goal in this section is to prove Theorem~\ref{thm_pot_aut_of_powers}. The proof will
occupy the whole section, but to keep the presentation organized and
somewhat motivated, we deduce it from
Theorem~\ref{thm_pot_aut_of_powers_moderately_many_assumptions} below, which
we will in turn deduce from Proposition~\ref{thm_pot_aut_of_powers_all_the_assumptions}.
\begin{thm}\label{thm_pot_aut_of_powers}
Let $F$ be an imaginary CM number field, and let 
\[ \cR = (M, S, \{ Q_v(X) \}, \{ r_\lambda \}, \{H_\tau\}) \]
be a very weakly compatible system of rank $2$ representations of
$G_F$. Let $m \geq 1$ be an integer, and suppose that the following conditions are satisfied:
\begin{enumerate} 
\item For each $\tau$, $H_\tau = \{ 0, m \}$.
\item $\det r_\lambda = \varepsilon^{-m}$.
\item\label{last_assumption} $\cR$ is strongly irreducible. 
\end{enumerate}
Then $\cR$ is pure of weight $m$, and for each $n \geq 1$, there exists a finite CM extension $F_n / F$ such that $F_n / \Q$ is Galois and $\Sym^{n-1} \cR|_{G_{F_n}}$ is automorphic. 
\end{thm}
\begin{proof}
  We may assume  that~$m \ge 2$, since otherwise
  the result follows from~\cite[Cor~7.1.12]{10author}. Let $v_0
  \not\in S$ be a place of $F$.   Theorem~\ref{thm_pot_aut_of_powers_moderately_many_assumptions} states that we can find, for each $n \geq 1$, a CM extension $F_n / F$, Galois over $\Q$, such that $\Sym^{n-1}\cR|_{G_{F_n}}$ is weakly automorphic of level prime to $\{ v | v_0 \}$. By Lemma~\ref{lem:pot aut implies purity}, the roots of $Q_{v_0}(X)$ are $q_{v_0}$-Weil numbers of weight
  $m$. Since $v_0 \not\in S$ is arbitrary, this shows that the compatible system $\cR$
  is pure of weight $m$. We can then apply Lemma
  \ref{lem_weakly_automorphic_and_pure} to conclude that $\Sym^{n-1} \cR|_{G_{F_n}}$ is automorphic, as required.
  \end{proof}

\begin{rem}\label{rem: why m bigger than 1}
We assume in our arguments below that~$m>1$. Our argument certainly applies
in principle to the case~$m=1$, but certain statements we make through the proof assume that~$m \ge 2$,
and so this assumption avoids having to make  the necessary extra remarks to cover the case~$m=1$.
Moreover, our argument in the case~$m=1$ would involve tensoring~$\cR$ with auxiliary~$1$-dimensional
representations, and not so surprisingly can be simplified to the point where 
it becomes very similar to the proof of~\cite[Cor~7.1.12]{10author}.
\end{rem}

Before giving our first technical result towards the proof of Theorem~\ref{thm_pot_aut_of_powers_moderately_many_assumptions} (and hence Theorem \ref{thm_pot_aut_of_powers} above), we sketch the idea of the proof. We begin with the strongly irreducible, very weakly compatible system $\cR$ of rank 2 and parallel Hodge--Tate weights $\{0, m \}$, and wish to show that $\Sym^{n-1} \cR$ is potentially (weakly) automorphic. This presents difficulties since the compatible system $\Sym^{n-1} \cR$ has parallel Hodge--Tate weights $\{ 0, m, 2m, \dots, (n-1) m\}$, while the auxiliary motives that we can construct to show potential automorphy have consecutive (and parallel) Hodge--Tate weights (and moreover, our automorphy lifting theorem Theorem \ref{thm:main_automorphy_lifting_theorem} applies only to Galois representations with consecutive Hodge--Tate weights). To get around this, we construct auxiliary compatible systems as follows:
\begin{itemize}
\item An auxiliary compatible system $\cR_{\aux}$ of rank $m$ and with consecutive (and parallel) Hodge--Tate weights $\{ 0, 1, \dots, m-1 \}$. Then $(\Sym^{n-1} \cR) \otimes \cR_{\aux}$ has rank $nm$ and consecutive (and parallel) Hodge--Tate weights $\{ 0, 1, \dots, nm-1 \}$. 
\item A second auxiliary compatible system $\cR_{\CM}$ of rank $m$ and with consecutive (and parallel) Hodge--Tate weights $\{ 0, 1, \dots, m-1 \}$ which is moreover induced from a character. 
\item A third auxiliary compatible system $\cS_{\UA}$ of rank $nm$ with consecutive (and parallel) Hodge--Tate weights $\{ 0, 1, \dots, mn-1 \}$, and which is moreover automorphic. We will construct $\cS_{\UA}$ 
(and $\cR_{\aux}$) as a member of the families of motives considered in \S \ref{ssec:dwork}. (The subscript `UA' stands for `universally automorphic'.)
\end{itemize}
These are  chosen to behave well with respect to distinct primes $p, r$ as follows:
\begin{itemize}
\item There is a congruence modulo $p$ linking $\cS_{\aux} := ( \Sym^{n-1} \cR ) \otimes \cR_{\aux}$ and $\cS_{\UA}$. We will apply Theorem \ref{thm:main_automorphy_lifting_theorem} to conclude that $\cS_{\aux}$  is automorphic.
\item There is a congruence modulo $r$ linking~$\cR_{\aux}$ and~$\cR_{\CM}$, and therefore also linking ~$\cS_{\aux}$ and~$\cS_{\CM} := ( \Sym^{n-1} \cR ) \otimes \cR_{\CM}$. We will apply Theorem \ref{thm:main_automorphy_lifting_theorem} a second time to conclude that $\cS_{\CM}$ is automorphic.
\item Since $\cR_{\CM}$ is induced from a Hecke character, $\cS_{\CM}$ is also induced (from an $n$-dimensional compatible system).
We will then be able to apply the description of the image of automorphic induction given in \cite{MR1007299} to conclude that $\Sym^{n-1} \cR$ is itself automorphic.
\end{itemize}
The most significant conditions that must be satisfied to apply
Theorem \ref{thm:main_automorphy_lifting_theorem} in each case are the
non-degeneracy of the residual images and the `connects' relation
locally at the $p$-adic (resp.\ $r$-adic) places of $F$. The non-degeneracy of the residual images will be easy to arrange by careful choice of data. It is the `connects' relation that is more serious and imposes the circuitous route followed here to prove the theorem. 

The statement of Proposition
\ref{thm_pot_aut_of_powers_all_the_assumptions} below is long but is
merely a precise formulation of the properties required of the various
auxiliary compatible systems needed to carry out the above sketch. The
main point in the proof of
Theorem~\ref{thm_pot_aut_of_powers_moderately_many_assumptions} will
be to show how to construct auxiliary compatible systems with these properties.

\begin{prop}\label{thm_pot_aut_of_powers_all_the_assumptions}
Let $F$ be an imaginary CM number field, let $m \ge 2$ and~$n\ge 1$ be integers, and let $X_0$ be a finite set of finite places of $F$. Let 
\[ \cR = (M, S, \{ Q_v(X) \}, \{ r_\lambda \}, \{H_\tau\}) \]
be a very weakly compatible system of rank $2$ representations of
$G_F$ satisfying the following conditions:
\begin{enumerate} 
\item\label{ass_HT} For each $\tau$, $H_\tau = \{ 0, m \}$.
\item $\det r_\lambda = \varepsilon^{-m}$.
\item\label{last_assumption_second} $\cR$ is strongly irreducible. 
\item $X_0 \cap S = \emptyset$.
\item\label{ass_galois} $F / \Q$ is Galois and contains an imaginary quadratic field $F_0$.
\end{enumerate}
We fix an embedding $M \hookrightarrow \C$, and regard $M$ as a subfield of $\C$. Suppose we can find the following additional data:
\begin{enumerate}
\setcounter{enumi}{5}
\item A cyclic totally real extension $E / \Q$ of degree $m$, linearly disjoint from $F$, and a character  $\Psi : \A_L^\times \to M^\times$, where $L = E \compositum F$, satisfying the following conditions:
\begin{enumerate} \item There is an embedding $\tau_0 : F_0 \to \C$, and a labelling $\tau_1, \dots,
  \tau_m : E \compositum F_0 \to \C$ of the embeddings $E
  \compositum F_0 \to \C$ 
   which extend $\tau_0$ such that for each
  $\alpha \in L^\times$, we have
  \[ \Psi(\alpha) = \prod_{i=1}^m \tau_i(\mathbf{N}_{L / E \compositum F_0}(\alpha))^{m-i}  c \tau_i(\mathbf{N}_{L / E \compositum F_0}(\alpha))^{i-1}. \]
  \end{enumerate}
We let $\{ \Psi_\lambda \}$ denote the weakly compatible system associated to $\Psi$, and let 
\[ \cR_{\CM} = \{ \Ind_{G_L}^{G_F} \Psi_\lambda \} = (M, S_{\CM}, \{ Q_{\CM, v }(X) \}, \{ r_{\CM, \lambda} \} , \{ H_{\CM, \tau} \}) \]
denote the induced weakly compatible system. \emph{(}Then $H_{\CM, \tau} = \{ 0, 1, \dots, m-1 \}$ for all $\tau$, and we take $S_{\CM}$ to be the set of places of $F$ ramified in $L$ or above which $\Psi$ is ramified.\emph{)}
\begin{enumerate}
\item[(b)] For all $\lambda$, $\det r_{\CM, \lambda} = \varepsilon^{-m(m-1)/2}$.
\end{enumerate}
\item Distinct primes $p, r$, not dividing any place of $S$, and places $\p, \r$ of $M$ lying above them. 
\item A weakly compatible system of rank $m$ representations of $G_F$
  \[ \cR_{\aux} = (M, S_{\aux}, \{ Q_{\aux, v}(X) \}, \{ r_{\aux, \lambda} \}, \{ H_{\aux, \tau} \}), \]
  satisfying the following conditions:
  \begin{enumerate}
  \item $\cR_{\aux}$ is pure of weight $m-1$ and $S_{\aux}$ does not intersect $X_0 \cup \{ v | pr \}$.
  \item For all $\lambda$, $\det r_{\aux, \lambda} = \varepsilon^{-m(m-1)/2}$.  For all $\tau$, $H_{\aux, \tau} = \{ 0, 1, \dots, m-1 \}$.
    \end{enumerate}
    \item A weakly compatible system %
\[ \cS_{\UA} = ( M, S_{\UA}, \{ Q_{\UA, v}(X) \}, \{ s_{\UA, \lambda} \}, \{ H_{\UA, \tau} \}) \]
of rank $nm$ representations of $G_F$, satisfying the following conditions:
\begin{enumerate} 
\item $\cS_{\UA}$ is pure of weight $nm-1$ and $S_{\UA}$ does not intersect $X_0 \cup  \{ v | pr \}$.
\item For all $\lambda$, $\det s_{\UA, \lambda} = \varepsilon^{-nm(nm-1)/2}$.  For all $\tau$, $H_{\UA, \tau} = \{ 0, 1, \dots, n m-1 \}$.

\item $\cS_{\UA}$ is weakly automorphic of level prime to $X_0 \cup \{ v | pr \}$.
\end{enumerate} 
\end{enumerate} 
Let~$\cS_{\aux} = \{ s_{\aux, \lambda} \} = (\Sym^{n-1} \cR) \otimes \cR_{\aux}$ and~$\cS_{\CM}
= \{ s_{\CM, \lambda} \} = (\Sym^{n-1} \cR) \otimes \cR_{\CM}$. These compatible systems of rank $nm$ have coefficients in the number field $M$.
Suppose that these data satisfy the following additional conditions:
\begin{enumerate}
\setcounter{enumi}{9}
\item $L/F$ is unramified at $X_0 \cup \{ v | pr \}$, and $\Psi$ is unramified at the places of $L$ lying above $X_0\cup \{ v | pr \}$. \emph{(}Then $S_\CM \cap (X_0 \cup \{ v | p r \}) = \emptyset.$\emph{)}
\item $p > 2 n m + 1$, and $[F(\zeta_p) : F] = p-1$.
\item\label{623cond:r} $r  > 2 n m + 1$, $r$ splits completely in $E \compositum F_0$, and $[L(\zeta_r) : L] = r-1$.
\item\label{623cond:sandwiches} Up to conjugation, there are sandwiches
 \[ \SL_2(\F_p) \leq \overline{r}_\p(G_F) \leq \GL_2(\F_p) \]
 and
 \[ \SL_2(\F_r) \leq \overline{r}_\r(G_F) \leq \GL_2(\F_r). \]
 If $m > 2$ then the image $\overline{r}_{\aux, \p}(G_F)$ is a conjugate of $\operatorname{GU}_m(\F_{p^2})$ and $\overline{r}_{\aux, \p}$ has multiplier character
    $\overline{\varepsilon}^{1-m}$. If $m = 2$ then the image $\overline{r}_{\aux, \p}(G_F)$ is a conjugate of    $\GL_2(\F_p)$. The representation $\overline{r}_{\CM, \r}|_{G_{F(\zeta_r)}}$ is irreducible. If $m = 2$, then the extensions of $F(\zeta_p)$ cut out by the projective representations associated to $\overline{r}_\p|_{G_{F(\zeta_p)}}$ and $\overline{r}_{\aux, \p}|_{G_{F(\zeta_p)}}$ are linearly disjoint. 
    \item There are isomorphisms $\overline{s}_{\UA, \p} \cong \overline{s}_{\aux, \p}$ and $\overline{r}_{\aux, \r} \cong \overline{r}_{\CM, \r}$.
    \item There is a decomposition $S_p = \Sigma^{\textrm{ord}} \sqcup \Sigma^{\textrm{ss}}$ of the set $S_p$ of $p$-adic places of $F$ such that for each place $v | p$ of $F$, $F_v$ contains $\Q_{p^2}$, $\overline{r}_\p|_{G_{F_v}}$ and $\rhobar_{2, m, 0}|_{G_{F_v}}$ (cf.~Definition \ref{df: defn of rho0}) are trivial, and: \begin{enumerate}
\item if $v \in \Sigma^{\textrm{ord}}$, then $r_{\p}|_{G_{F_v}}$ is crystalline ordinary;
\item if $v \in \Sigma^{\textrm{ss}}$, then $r_{\p}|_{G_{F_v}} \sim \rho_{2, m, 0}|_{G_{F_v}}$. 
 \end{enumerate} 
  \item If $v \in \Sigma^{\textrm{ord}}$, then $\overline{r}_{\aux, \p}|_{G_{F_v}}$ is trivial and $r_{\aux, \p}|_{G_{F_v}}$ and $s_{\UA, \p}|_{G_{F_v}}$ are both crystalline ordinary. If $v \in \Sigma^{\textrm{ss}}$, then $\overline{r}_{\aux, \p}|_{G_{F_v}}$ is trivial, $r_{\aux, \p}|_{G_{F_v}}$ and $s_{\UA, \p}|_{G_{F_v}}$ are both crystalline, and $r_{\aux, \p}|_{G_{F_v}} \sim \rho_{m, 1, 0}|_{G_{F_v}}$ and $s_{\UA, \p}|_{G_{F_v}} \sim  \rho_{nm, 1, 0}|_{G_{F_v}}$.
 
 \item\label{ass_CM_triv}  For each place $v | r$ of $F$, $\overline{r}_{\aux,\r}|_{G_{F_v}} \cong \overline{r}_{\CM, \r}|_{G_{F_v}}$ is trivial and  $r_{\aux,\r}|_{G_{F_v}}$ is crystalline ordinary. 
\end{enumerate} 
Then $\Sym^{n-1} \cR$ is weakly automorphic of level prime to~$X_0$. 
\end{prop}
\begin{proof}
We first show that $\cS_{\aux}$ is weakly automorphic of level prime to $X_0 \cup \{ v | r \}$ by applying
Theorem \ref{thm:main_automorphy_lifting_theorem} to $s_{\aux,
  \p}$. To justify this, we need to check that $\overline{s}_{\aux,
  \p} \cong \overline{s}_{\UA, \p}$ satisfies the Taylor--Wiles
conditions (as formulated in Definition \ref{tw}) and that for each
place $v | p$ of $F$, we have $s_{\aux, \p}|_{G_{F_v}} \sim s_{\UA,
  \p}|_{G_{F_v}}$.  The Taylor--Wiles conditions hold by assumption (\ref{623cond:sandwiches}) and Lemma
\ref{yeomanslemmanew}. %
If $v \in \Sigma^{\mathrm{ord}}$, then $\overline{s}_{\aux, \p}|_{G_{F_v}}$ is trivial, and both $s_{\aux, \p} \cong (\Sym^{n-1} r_\p) \otimes r_{\aux, \p}|_{G_{F_v}}$ and $s_{\UA, \p}|_{G_{F_v}}$ are crystalline ordinary, so Lemma \ref{prep2} implies that $s_{\aux, \p}|_{G_{F_v}} \sim s_{\UA, \p}|_{G_{F_v}}$. If $v \in \Sigma^{{\mathrm{ss}}}$, then $\overline{s}_{\aux, \p}|_{G_{F_v}}$ is trivial and our assumptions imply that $\Sym^{n-1} r_\p|_{G_{F_v}} \sim \rho_{n,m,0}|_{G_{F_v}}$, $r_{\aux, \p}|_{G_{F_v}} \sim \rho_{m, 1, 0}|_{G_{F_v}}$ and $s_{\UA, \p}|_{G_{F_v}} \sim \rho_{nm, 1, 0}|_{G_{F_v}}$, hence
\[ s_{\aux, \p}|_{G_{F_v}} \sim \rho_{n,m,0}|_{G_{F_v}} \otimes \rho_{m, 1, 0}|_{G_{F_v}} \cong \rho_{nm, 1, 0}|_{G_{F_v}} \sim s_{\UA, \p}|_{G_{F_v}}. \]
Therefore $\cS_{\aux}$ is weakly automorphic of level prime to $X_0 \cup \{ v | r \}$.

We next show that $\cS_{\CM}$ is weakly automorphic of level prime to~$X_0$ by applying Theorem \ref{thm:main_automorphy_lifting_theorem} to $s_{\CM, \r}$. The Taylor--Wiles conditions for $\overline{s}_{\CM, \r} \cong \overline{s}_{\aux, \r}$ hold by assumption (\ref{623cond:sandwiches}) and Lemma \ref{yeomanslemmanewvariant}. To check the connectedness conditions, let $v | r$ be a place of $F$. Then $\overline{r}_{\CM, \r}|_{G_{F_v}} \cong \overline{r}_{\aux, \r}|_{G_{F_v}}$ is trivial and $r_{\aux, \r}|_{G_{F_v}}$ is crystalline ordinary, by assumption (\ref{ass_CM_triv}). Since $r$ splits completely in $E$ by assumption (\ref{623cond:r}), $v $ splits completely  in $L$, and we can label the places $w_i | v$ so that $w_i|_E$ is the place induced by the embedding $\tau_i$. There is an isomorphism
\[ r_{\CM, \r}|_{G_{F_v}} \cong \oplus_{w | v} \alpha_{i}, \]
where for each $i = 1, \dots, m$, $\alpha_{i} : G_{F_{v}} \to \overline{M}_\r^\times$ is a continuous character with the property that for any $u \in \cO_{F_v}^\times$, we have
\[ \alpha_{ i}( \Art_{F_v}(u) ) = \prod_{\substack{ \tau \in \Hom(L_{w_i}, \overline{M}_\r) \\ \tau|_{F_0} = \tau_0}} \tau(u)^{-(m-i)}  \]
if $v$ lies above the place of $F_0$ induced by $\tau_0$, and
\[ \alpha_{w, i}( \Art_{F_v}(u) ) = \prod_{\substack{ \tau \in \Hom(L_{w_i}, \overline{M}_\r) \\ \tau|_{F_0} = c \tau_0}} \tau(u)^{-(i-1)} \]
otherwise. It follows that $r_{\CM, \r}|_{G_{F_v}}$ is also
crystalline ordinary, with Hodge--Tate weights $\{0,\ldots,m-1\}$ matching those of $r_{\aux, \r}|_{G_{F_v}}$. By Lemma \ref{prep2}, we have $r_{\CM,
  \r}|_{G_{F_v}} \sim r_{\aux, \r}|_{G_{F_v}}$, and using \cite[p. 530,
(5)]{BLGGT} it follows that
\[ s_{\CM, \r}|_{G_{F_v}} = \Sym^{n-1} r_{\r}|_{G_{F_v}} \otimes r_{\CM, \r}|_{G_{F_v}}  \sim  \Sym^{n-1} r_{\r}|_{G_{F_v}} \otimes  r_{\aux, \r}|_{G_{F_v}} = s_{\aux, \r}|_{G_{F_v}}. \]

We can now show that $\Sym^{n-1} \cR$ is weakly automorphic of level prime to $X_0$.  Let $\pi$ be the regular algebraic, cuspidal automorphic representation of $\GL_{nm}(\A_F)$ which is associated to the compatible system $\cS_{\CM}$. By construction, $\pi$ is unramified at~$X_0$. Let $\eta : F^\times \backslash \A_F^\times \to \C^\times$ be the character of order $m$ associated to the inducing field $L / F$ of $\cR_{\CM}$. Then $\pi \cong \pi \otimes (\eta \circ \det)$, so by cyclic base change~\cite[Ch.\ 3, Thm~4.2]{MR1007299}, we deduce that~$\pi$ is the induction
of a cuspidal automorphic representation~$\Pi$ for~$\GL_{n}(\A_{L})$, which by consideration of the infinity type of~$\pi$ must also be regular algebraic. More precisely, for any place $w$ of $L$ lying above a place $v$ of $F$, we have
\[ \operatorname{rec}_{F_v}(\pi)|_{W_{L_w}} = \oplus_{i=0}^{m-1} \operatorname{rec}_{L_w}(\Pi^{\sigma^i}), \]
where $\sigma$ is a generator for $\Gal(L / F)$. Since $L$ is CM and $\Pi$ is regular algebraic, $\Pi$ has an associated compatible system of $l$-adic Galois representations. If $l$ is a prime and $\iota : \overline{\Q}_l \to \C$ is an isomorphism, with $\iota^{-1}$ inducing the place $\lambda$ of $M$, then we find
\[ r_\iota(\pi)|_{G_L} = \oplus_{i=0}^{m-1} \Sym^{n-1} r_\lambda|_{G_L} \otimes \Psi_\lambda^{\sigma^{i}} \cong \oplus_{i=0}^{m-1} r_\iota(\Pi^{\sigma^i}). \]
Choosing $\lambda$ so that $\Sym^{n-1} r_\lambda$ is irreducible
(e.g. $\lambda = \p$), we find that $\Sym^{n-1} r_\lambda|_{G_L}$ is a
character twist of $r_\iota(\Pi)$. Undoing the twist and making cyclic
descent (using the irreducibility of $\Sym^{n-1} r_\lambda|_{G_L}$, as
in \cite[Proposition 6.5.13]{10author}) shows that $\Sym^{n-1} \cR$ is
weakly automorphic over $F$ of level prime to $X_0$, as desired. 
\end{proof}
The next theorem is proved by constructing the data required by Proposition \ref{thm_pot_aut_of_powers_all_the_assumptions} (after possibly extending the base field $F$).
\begin{thm}\label{thm_pot_aut_of_powers_moderately_many_assumptions}
Let $F$ be an imaginary CM number field, and let 
\[ \cR = (M, S, \{ Q_v(X) \}, \{ r_\lambda \}, \{H_\tau\}) \]
be a very weakly compatible system of rank $2$ representations of
$G_F$. Let $m \ge 2$ be an integer, and suppose that the following conditions are satisfied:
\begin{enumerate} 
\item For each $\tau$, $H_\tau = \{ 0, m \}$.
\item $\det r_\lambda = \varepsilon^{-m}$.
\item\label{last_assumption_third} $\cR$ is strongly irreducible. 
\end{enumerate}
Let $v_0 \not\in S$ be a place of $F$. Then
for each $n \geq 1$, there is a CM extension $F_n/ F$, Galois over $\Q$, such that $\Sym^{n-1} \cR|_{G_{F_n}}$ is weakly automorphic of level prime to $\{v|v_0\}$. 
\end{thm}
\begin{proof}
We can fix $n \geq 1$. Let $p_0$ denote the residue characteristic of $v_0$, let $F_0$ be an imaginary quadratic field, and let $F_1$ denote the Galois closure of $F \compositum F_0$ over $\Q$. Embed $M$ in $\C$ arbitrarily, and let $X_1$ denote the set of places of $F_1$ lying above $v_0$. It suffices to prove the following statement: 
\begin{itemize}
\item There exists a CM extension $F' / F_1$, Galois over $\Q$, such that (after possibly enlarging $M$) $\cR|_{G_{F'}}$ satisfies the hypotheses of   Proposition~\ref{thm_pot_aut_of_powers_all_the_assumptions} with $X_0$ taken to be the set of places of $F'$ lying above $X_1$. 
\end{itemize}
Indeed, Proposition~\ref{thm_pot_aut_of_powers_all_the_assumptions} will then imply that $\Sym^{n-1} \cR|_{G_{F'}}$ is weakly automorphic of level prime to $X_0 = \{ v | v_0 \}$, which is what we need to prove. To prove this statement, we will consider a series of CM extensions $F_{j+1} / F_j$ ($j = 1, 2, \dots$), each Galois over $\Q$. For any such extension $F_j / F_1$, $\cR|_{G_{F_j}}$ satisfies Assumptions (\ref{ass_HT})--(\ref{ass_galois}) of Proposition \ref{thm_pot_aut_of_powers_all_the_assumptions} with respect to $X_j$, the set of places of $F_j$ lying above $v_0$. The extensions $F_{j+1} / F_j$ will be chosen in order to satisfy the remaining assumptions. 

Let~$E/\Q$ be any totally real cyclic extension of degree~$m$ linearly disjoint from~$F_1$, in which $p_0$ is unramified.
(We can find such~$E$ by taking the degree~$m$ subfield of~$\Q(\zeta_{p'})$, where~$p'$ is any 
sufficiently large prime~$\equiv 1 \bmod 2m$.) 
Let~$L_1 = E \compositum F_1$. For any extension $F_j / F_1$, we will set $L_j = E \compositum F_j$. Choose an odd prime $q_1 \nmid X_0$ which splits
completely in $L_1$ %
and a
place $v_1 | q_1$ of $F_1$ which splits completely
 as $v_1
 = w_1 \dots w_m$ in $L_1$. Fix an embedding $\tau_0 : F_0 \to \C$, and a labelling $\tau_1, \dots,
  \tau_m : E \compositum F_0 \to \C$ of the embeddings $E
  \compositum F_0 \to \C$ 
   which extend $\tau_0$. After enlarging $M$, using \cite[Lemma 2.2]{HSBT}, we can find a character $\Psi_0 : \A_{L_1}^\times \to M^\times$, unramified at the places above $X_1$, such that for each
  $\alpha \in L_1^\times$, we have
  \[ \Psi_0(\alpha) = \prod_{i=1}^m \tau_i(\mathbf{N}_{L_1 / E \compositum F_0}(\alpha))^{m-i}  c \tau_i(\mathbf{N}_{L_1 / E \compositum F_0}(\alpha))^{i-1}, \]
and moreover such that the characters
$\Psi_0|_{{\cO^\times_{L_{w_i}}}}$ ($i = 1, \dots, m$) are wildly
ramified, pairwise distinct, and satisfy $\prod_{i=1}^m
\Psi_0|_{\cO^\times_{L_{w_i}}} = 1$ (where we identify $F_{v_1} =
L_{w_i}$ for each $i$). If $\lambda$ is a place of $M$, then $\varepsilon^{m(m-1)/2} \det
\Ind_{G_{L_1}}^{G_{F_1}} \Psi_{0, \lambda}$ is a character of finite order
which is unramified at $v_0$ and $v_1$. Using Lemma
\ref{fixdeterminant}, and possibly enlarging $M$ further, we can find a CM extension $F_2 / F_1$, linearly disjoint from $E/\Q$ and Galois over $\Q$, and a twist $\Psi_2 : \A_{L_2}^\times \to M^\times$ of $\Psi_0 \circ \mathbf{N}_{L_2 / L_1}$ by a character of $L_2^\times \backslash \A_{L_2}^\times$ of finite order, unramified above $v_0$ and $v_1$, such that for any place $\lambda$ of $M$, $\det \Ind_{G_{L_2}}^{G_{F_2}} \Psi_{2, \lambda} = \varepsilon^{-m(m-1)/2}$. (If~$v_0 | 2$, then the twist whose existence
is guaranteed from Lemma~\ref{fixdeterminant} may be ramified above $X_2$; if so,  it is certainly
ramified of finite order, and we enlarge $F_2$ further
so that~$\Ind_{G_{L_2}}^{G_{F_2}} \Psi_{2, \lambda}$ is unramified at all places above~$X_2$
and~$L_2/F_2$ is unramified at all places above~$X_2$.) If $F_j / F_2$ is a finite extension, then we set $\Psi_j = \Psi_2 \circ \mathbf{N}_{F_j / F_2}$. Let $\cR_{\CM} = \{ \Ind_{G_{L_2}}^{G_{F_2}} \Psi_{2, \lambda} \} = \{ r_{\CM, \lambda} \}$. 

 We now choose any primes $N$, $p \in \cL(\cR|_{G_{F_2}})$,  $r \in \cL(\cR|_{G_{F_2}})$ (cf.~Lemma \ref{lem_properties_of_strongly_irreducible_very_weakly_compatible_systems}) not
 dividing $v_0 v_1$  and satisfying the following conditions:
 \begin{itemize}
 \item $N > 100nm + 100$ and $N$ is unramified in $L_2$ and $M$. 
 \item $p \equiv -1 \text{ mod }N$ and $p > 2nm+1$.
\item $r \equiv 1 \text{ mod }N$ and $r > 2nm+1$.
 \item $p$ splits completely in $L_2$ and $M$ and $r$ splits completely in $L_2(\zeta_p)$ and $M$. 
 \item The character $\Psi_2$ is unramified at the places of $L$ above $p$ and $r$.
 \end{itemize} 
Choosing $\p |p$ and $\r | r$ arbitrarily, there will be sandwiches up to conjugation
 \[ \SL_2(\F_p) \leq \overline{r}_\p(G_{F_2}) \leq \GL_2(\F_p) \]
 and
 \[ \SL_2(\F_r) \leq \overline{r}_\r(G_{F_2}) \leq \GL_2(\F_r), \]
 and for each $p$-adic (resp. $r$-adic) place $v$ of $F$, $r_\p|_{G_{F_v}}$ (resp. $r_\r|_{G_{F_v}}$) is crystalline. (Here we are using the definition of $\cL(\cR|_{G_{F_2}})$ and the fact that $p$, $r$ split in $M$.) The representation $\overline{r}_{\CM, \r}$ can be chosen to take values in $\GL_m(\F_r)$. Since the prime $N$ is unramified in $L_2(\zeta_r)$, $E / \Q$ is linearly disjoint from $F_2(\zeta_N, \zeta_r) / \Q$.  The different inertial behaviour of $\Psi_0$ at places dividing $v_1$ implies that
 $\overline{r}_{\CM, \r}|_{G_{F_2(\zeta_N, \zeta_r)}}$ is absolutely irreducible. 

Let $v$ be a $p$-adic place of $F$. Then $F_{2, v} = M_\p = \Q_p$. By \cite[Th\'eor\`eme 3.2.1]{ber05},
either $r_\p|_{G_{F_{2, v}}}$ is (crystalline) ordinary, or there is an
isomorphism $\overline{r}_\p|_{G_{\Q_{p^2}}} \cong \rhobar_{2, m, 0}$
(notation as in Definition \ref{df: defn of rho0}). In the latter
case, Lemma \ref{prep} shows that for any finite extension $K /
\Q_{p^2}$, we have $r_\p|_{G_{K}} \sim \rho_{2, m, 0}|_{G_K}$. We write $\Sigma_2^{\text{ord}}$ (resp. $\Sigma_2^{\text{ss}}$) for the set of $p$-adic places of $F_2$ such that $r_\p|_{G_{F_{2, v}}}$ is (resp. is not) ordinary. If $F_j / F_2$ is a finite extension, then we write $\Sigma_j^{\text{ord}}$ for the set of places of $F_j$ lying above a place of $\Sigma_2^{\text{ord}}$ (and define $\Sigma_j^{\text{ss}}$ similarly). 

Let $B / F_2(\zeta_N, \zeta_p, \zeta_r)$ be the extension cut out by $\overline{r}_\p \times \overline{r}_\r \times \overline{r}_{\CM,
  \r}$. 
We now choose a solvable CM extension $F_3 / F_2(\zeta_N)$, Galois over
$\Q$ and linearly disjoint from $B \compositum F_2 / F_2(\zeta_N)$. Since $p \equiv -1 \text{ mod }N$, for each place $v | p$ of $F_3$, $F_v$ contains $\Q_{p^2}$. We moreover
adjoin $e^{2 \pi i / N}$ to $M$ and extend $\p$, $\r$ arbitrarily to places of this enlarged $M$. 

At this point we choose (for later use) a semistable elliptic curve  $A / \Q$ with good reduction at $p$, $r$, and $p_0$. We choose a prime $q$ with the following properties:
\begin{itemize}
\item $q > 2 n m + 1$ and $q$ splits in $M$. In particular, $q \equiv 1 \text{ mod }N$. We choose a place $\q | q$ of $M$.
\item $\rhobar_{A, q}(G_{F_3}) = \GL_2(\F_q)$ and $A$ has good ordinary reduction at $q$.
\end{itemize} 
Let $B'$ denote the composite of $B$ with the extension of $F_3$ cut out by $\rhobar_{A, q}$.

Having chosen an integer $N$ and extension $F_3 / \Q(\zeta_N)$, we have
  access to the families of motives over $T_0 = \mathbf{P}^1_{F_3} - \{
  \mu_N, \infty \}$ constructed in \S \ref{ssec:dwork}. We will use
  the families of motives both of rank $m$ and of rank $nm$. We write
  ${}_m W_{t, \lambda}$, ${}_{nm} W_{t, \lambda}$ for the
  $\cO_{M_\lambda}[G_K]$-modules of ranks $m$, $nm$ constructed in \S
  \ref{ssec:dwork} associated to an extension $K / F_3$ and point $t \in T_0(K)$. We claim that we can find a CM extension $F_4 / F_3$, Galois over $\Q$ and linearly disjoint from $B' \compositum F_3 / F_3$  such that for any place $v | p r p_0 q$ of $F_4$, the representations $\overline{r}_\p|_{G_{F_{4, v}}}$, $\overline{r}_\r|_{G_{F_{4, v}}}$, $\overline{r}_{\CM, \r}|_{G_{F_{4, v}}}$ and $\overline{\rho}_{A, q}|_{G_{F_{4,  v}}}$ are all trivial, and the following additional
data exists 
for $k \in
  \{ m, nm \}$:
  \begin{enumerate}[label=(\roman*)]
\item \label{item:MB ord} If $v \in \Sigma_4^{\textrm{ord}}$, then there is a non-empty open subset ${}_k \Omega_v \subset T_0(\cO_{F_{4,v}})$ such that if $t \in {}_k \Omega_v$, then ${}_k \overline{W}_{t, \p}$ is trivial and ${}_k W_{t, \p}$ is crystalline ordinary. Moreover, ${}_k \overline{W}_{t, \r}$ and ${}_k \overline{W}_{t, \q}$ are both trivial.
\item \label{item:MB ss} If $v \in \Sigma_4^{\textrm{ss}}$, then there is a non-empty open subset ${}_k \Omega_v \subset T_0(\cO_{F_{4, v}})$ such that if $t \in {}_k \Omega_v$, then ${}_k \overline{W}_{t, \p}$ and $\overline{\rho}_{k, 1, 0}|_{G_{F_{4, v}}}$ are trivial, ${}_k W_{t, \p}$ is crystalline, and ${}_k W_{t, \p} \sim \rho_{k, 1, 0}|_{G_{F_{4, v}}}$. Moreover, ${}_k \overline{W}_{t, \r}$ and ${}_k \overline{W}_{t, \q}$ are both trivial.
\item \label{item:MB r} If $v  | r$ is a place of $F_4$, then there is a non-empty open subset ${}_k  \Omega_v \subset T_0(\cO_{F_{4, v}})$ such that if $t \in \Omega_v$, then ${}_k \overline{W}_{t, \r}$ is trivial and $W_{t, \r}$ is crystalline ordinary. Moreover, ${}_k \overline{W}_{t, \q}$ and ${}_k \overline{W}_{t, \p}$ are both trivial.  
\item  \label{item:MB p_0} If $v  | p_0$ is a place of $F_4$, then there is a non-empty open subset ${}_k  \Omega_v \subset T_0(\cO_{F_{4, v}})$ such that if $t \in {}_k \Omega_v$, then  ${}_k \overline{W}_{t, \r}$, ${}_k \overline{W}_{t, \q}$ and ${}_k \overline{W}_{t, \p}$ are all trivial.
\item  \label{item:MB q} If $v  | q$ is a place of $F_4$, then there is a non-empty open subset $ {}_k \Omega_v \subset T_0(\cO_{F_{4, v}})$ such that if $t \in {}_k \Omega_v$, then ${}_k \overline{W}_{t, \q}$ is trivial and ${}_k W_{t, \q}$ is crystalline ordinary. Moreover,  ${}_k \overline{W}_{t, \p}$ and ${}_k \overline{W}_{t, \r}$ are both trivial. 
\end{enumerate}
Indeed, we can take $F_4 = K^+ \compositum F_3$, where $K^+ / \Q$ is a Galois, totally real extension with $K^+_v$ large enough for each place $v | prp_0 q$, as we now explain, dropping the subscript $k$ which is fixed for the next two paragraphs. For \ref{item:MB ord},
  we claim that it is enough to show that once $F_{4, v}$ is large enough, we
  can find a single point of $t \in T_0(F_{4, v})$ such that
  $\overline{W}_{t, \p}$, $\overline{W}_{t, \r}$, and $\overline{W}_{t, \q}$ are all trivial
  and $W_{t, \p}$ is crystalline ordinary. Indeed, by a version of
  Krasner's Lemma due to Kisin \cite[Theorem 5.1]{kisin-krasner}, for
  any~$c > 0$ there exists an open ball $U_t$ around~$t$ in
  $T_0(\cO_{F_{4, v}})$, such that for any $t' \in U_t$, the pairs of
  representations $W_{\p, t} / (p^c)$, $W_{\p, t'} / (p^c)$ and
  $W_{\r, t} / (r^c)$, $W_{\r, t'} / (r^c)$ and $W_{\q, t} / (q^c)$, $W_{\q, t'} / (q^c)$ are isomorphic. By Lemma
  \ref{smooth}, we can choose $c > 1$ so that this forces $W_{\p, t}
  \sim W_{\p, t'}$, hence (by Lemma \ref{prep2}) that $W_{\p, t'}$ is
  crystalline ordinary. The existence of a crystalline ordinary point
  $t$ follows from Proposition \ref{ordinarypoints} and Proposition
  \ref{prop_independence_of_l}(2), after which we enlarge $F_{4, v}$ further if
  necessary to force the residual representations  to be trivial. Then we take $\Omega_v = U_t$.

For \ref{item:MB r} and \ref{item:MB q}, the argument is essentially the same as case
\ref{item:MB ord}, while for   \ref{item:MB p_0}, it is even simpler. For \ref{item:MB ss}, we enlarge $F_{4, v}$ so
that $\overline{\rho}_{k, 1, 0}|_{G_{F_{4, v}}}$ and $\overline{W}_{\p, 0}|_{G_{F_{4, v}}}$
are trivial. By Lemma \ref{prep} and Lemma \ref{specializations}, we
have $W_{\p, 0}|_{G_{F_{4, v}}} \sim \rho_{k, 1, 0}|_{G_{F_{4, v}}}$. Employing the same
argument as in the previous paragraph, using \cite[Theorem 5.1]{kisin-krasner} and Lemma
\ref{smooth}, we can find a non-empty open neighbourhood $\Omega_v
\subset T_0(\cO_{F_{4, v}})$ of $0 \in T_0(\cO_{F_{4, v}})$ such that if $t \in
\Omega_v$, then $W_{\p, t}$ is crystalline and $W_{\p, t} \sim W_{\p,
  0}|_{G_{F_{4, v}}}$. Since $\sim$ is a transitive relation, this leads to a choice
of $\Omega_v$ with the desired property.

To construct the compatible system $\cR_{\aux}$, we will apply
Proposition \ref{prop_variation_of_MB}. If~$m=2$ we can  use a
modular curve with level $r$-structure, and since the argument in
this case is a straightforward (and considerably simpler) variant on
the argument that we use if~$m>2$, we leave this case to the
reader. In the case $m > 2$ we use the moduli space $T = T(
\overline{r}_{\CM, \r}|_{G_{F_4}}  )$ defined in Remark \ref{pqmoduliremark}, which is defined since $r \equiv 1 \text{
  mod }N$ and $\overline{r}_{\CM, \r}$ takes values in $\GL_m(\F_r)$,
with determinant $\overline{\varepsilon}^{-m(m-1)/2}$. We take
$F^{\avoid} = B' \compositum F_4$. We take the homomorphism
$\pi_1^{\text{\'et}}(T_{F_4}) \to \operatorname{GU}_m(\F_{p^2})$ to be the
one associated to the local system $\overline{\cW}_\p$. We take $S_0 =
\{ p, r, p_0, q \}$. If $v$ is a place lying above a prime in $S_0$, we take
$L_v = F_{4, v}$ and $\Omega_v$ to be the pre-image in $T(F_{4, v})$ of the set ${}_m \Omega_v$. Note that $\Omega_v$ is certainly open, and it is non-empty because we have arranged that for each place $v | S_0$ of $F_4$, and for each $t \in {}_m \Omega_v$, $\overline{r}_{\CM, \r}|_{G_{F_{4, v}}}$ and $\overline{W}_{t, \r}$ are both trivial (hence isomorphic!).
  
  Proposition \ref{prop_variation_of_MB} now yields an imaginary CM extension $F_5 /
  F_4$, Galois over $\Q$ and in which the places above $S_0$ all split completely, and a weakly compatible system $\{ W_{t, \lambda} \}$ of
  representations of $G_{F_5}$ with coefficients in $\Q(e^{2 \pi i /
    N}) \subset M$. We take $\cR_{\aux} = \{ r_{\aux, \lambda} \} = \{ W_{t, \lambda} \}$ and note that the statement of  Proposition \ref{prop_variation_of_MB} and the definition of the sets $\Omega_v$ imply that $\cR_{\aux}$ has the following properties:
 \begin{itemize} \item $\overline{r}_{\aux, \p}(G_{F_5}) = \operatorname{GU}_m(\F_{p^2})$ (note we are assuming that $m >2$).
         \item If $v \in \Sigma_5^{\mathrm{ord}}$, then $\overline{r}_{\aux, \p}|_{G_{F_{5, v}}}$ is trivial and $r_{\aux, \p}|_{G_{F_{5, v}}}$ is crystalline ordinary.
         \item If $v \in \Sigma_5^{\mathrm{ss}}$, then $\overline{r}_{\aux, \p}|_{G_{F_{5, v}}}$ is trivial and $r_{\aux, \p}|_{G_{F_{5, v}}} \sim \rho_{m, 1, 0}|_{G_{F_{5, v}}}$.
         \item  $S_{\aux}$ is disjoint from $X_5 \cup \{ v | p r \}$. (Use Proposition \ref{prop_independence_of_l}.)
         \item There is an isomorphism $\overline{r}_{\aux, \r} \cong \overline{r}_{\CM, \r}|_{G_{F_5}}$. For each place $v | r$ of $F_5$, $\overline{r}_{\aux, \r}|_{G_{F_{5, v}}}$ is trivial and $r_{\aux, \r}|_{G_{F_{5, v}}}$ is crystalline ordinary.
         \end{itemize} 
 We set $\cS_{\aux} = (\Sym^{n-1} \cR|_{G_{F_5}}) \otimes \cR_{\aux}$, and now construct $\cS_{\UA}$. The places $v | prp_0q$ split in $F_5 / F_4$, so if $v$ is a place of $F_5$ dividing $prp_0q$ we may define ${}_k \Omega_v = {}_k \Omega_{v|_{F_4}}$ to keep in hand the data   \ref{item:MB ord}--\ref{item:MB q} defined above.  We will apply Proposition \ref{prop_variation_of_MB} to the moduli space 
\[ T = T(\sbar_{\aux,\p}, \Sym^{nm-1} \rhobar_{A,q} |_{G_{F_5}}). \]
 We take $F^{\avoid} = B' \compositum F_5$. We do not specify a homomorphism $f$. We take $S_0 = \{ p, r, p_0, q  \}$. If $v$ is a place lying above a prime in $S_0$, we take
$L_v = F_{5, v}$ and $\Omega_v$ to be the pre-image in $T(F_{5, v})$ of the set ${}_{nm} \Omega_v$. Once again, this pre-image is non-empty because we have trivialized all of the relevant local residual representations. (Since $p \equiv -1 \text{ mod }N$, the definition of $T$ involves a choice of Hermitian structure. We are therefore invoking the fact here that over a finite field, any two Hermitian spaces of the same dimension are isomorphic.) Proposition \ref{prop_variation_of_MB} then yields a CM extension $F_6 / F_5$, Galois over $\Q$, and a point $t \in T(F_6)$ corresponding to a weakly compatible system $\cS_{\UA} = \{ s_{\UA, \lambda} \} = \{ W_{t, \lambda} \}$ of rank $nm$ representations of $G_{F_6}$ with the following properties:
\begin{itemize}
\item There are isomorphisms $\overline{s}_{\UA, \p} \cong \overline{s}_{\aux, \p}|_{G_{F_6}}$ and $\overline{s}_{\UA, \q} \cong \Sym^{nm-1} \overline{\rho}_{A, q}|_{G_{F_6}}$.
\item If $v \in \Sigma_6^{\mathrm{ord}}$, then $\overline{s}_{\UA, \p}|_{G_{F_{6, v}}}$ is trivial and $s_{\UA, \p}|_{G_{F_{6, v}}}$ is crystalline ordinary. 
\item If $v \in \Sigma_6^{\mathrm{ss}}$, then $\overline{s}_{\UA, \p}|_{G_{F_{6, v}}}$ is trivial and $s_{\UA, \p}|_{G_{F_{6, v}}} \sim \rho_{nm, 1, 0}|_{G_{F_{6, v}}}$. 
\item For each place $v | q$ of $F_6$, $\overline{s}_{\UA, \p}|_{G_{F_{6, v}}}$ is trivial and $s_{\UA, q}|_{G_{F_{6, v}}}$ is crystalline ordinary. 
\item $S_{\UA}$ is disjoint from $X_6 \cup \{ v | p r \}$.
\end{itemize} 
 We now claim that Assumptions (\ref{ass_HT})--(\ref{ass_CM_triv}) of Proposition \ref{thm_pot_aut_of_powers_all_the_assumptions} are satisfied for the compatible system $\cR|_{G_{F_6}}$, set $X_0 = X_6$ of places of $F_6$, and auxiliary compatible systems $\cR_{\CM}|_{G_{F_6}}$, $\cR_{\aux}|_{G_{F_6}}$, and $\cS_{\UA}$ (defined over $F_6$ by construction). Let us verify these assumptions in turn.
\begin{itemize}
\item As already observed,  (\ref{ass_HT})--(\ref{ass_galois}) are automatically satisfied.
\item We take $\Psi = \Psi_6$. The extension $E / \Q$ is linearly disjoint from $F_6$ because $E \leq B$, while $\Psi$ has the given infinity type, so (6) is satisfied.
\item The primes $p, r$ are prime to $S$ by construction, so (7) is satisfied. 
\item $\cR_{\aux}|_{G_{F_6}}$ has the claimed properties by construction, so (8) is satisfied. The same is true for $\cS_{\UA}$, except we need to justify the fact that $\cS_{\UA}$ is weakly automorphic of level prime to $X_6 \cup \{ v | p r \}$. Note that the $q$-adic representation
$\Sym^{nm - 1} \rho_{A, q}|_{G_{F_6}}$ is automorphic by the
combination of the main results of \cite{MR1839918, MR3394612,
  newton2022symmetric, MR1007299} (or alternately by \cite{Clo23}),
associated to a regular algebraic, cuspidal automorphic representation of $\GL_{nm}(\A_{F_6})$ which is $\iota$-ordinary with respect to any isomorphism $\iota : \overline{\Q}_q \to \C$. (We could also verify the automorphy, at the cost of further extending the field $F_6$, by a further application of Proposition \ref{prop_variation_of_MB}  as is done in  \cite{10author}.) We would now like to apply \cite[Theorem 1.3]{miagkov-thorne} to conclude that $\cS_{\UA}$ is weakly automorphic of level prime to $X_6 \cup \{ v | p r \}$ (noting that the cited result includes the conclusion that the automorphic representation witnessing the weak automorphy of $\cS_{\UA}$ is unramified at any place where both $\rho_{A, q}$ and $s_{\UA, \q}$ are unramified). We must verify that $\Sym^{nm - 1} \rhobar_{A, q}|_{G_{F_6}}$ satisfies the Taylor--Wiles conditions (as formulated in Definition \ref{tw}). By Lemma \ref{restricttw}, it suffices to check that  $\Sym^{nm - 1} \rhobar_{A, q}$ satisfies these conditions (as a representation of $G_\Q$), and this follows  from the definitions, together with an application of  \cite[Theorem A.9]{jackapp} (using our assumption $q > 2nm + 1$).

\item $L / F$ and $\Psi$ are unramified above $X_0 \cup \{ v | p r \}$ by construction, so (10) is satisfied. 
\item We have chosen the primes $p, r$ so that $p > 2nm +1$ and $r > 2nm + 1$. At each step the extension $F_{j+1} / F_{j}$ has been chosen linearly disjoint from $L_j(\zeta_p, \zeta_r)$, so (11) and (12) are satisfied.
\item The images $\overline{r}_\p(G_{F_2})$, $\overline{r}_\r(G_{F_2})$ and $\overline{r}_{\aux, \p}(G_{F_5})$ are large by construction, and at each step the extension $F_{j+1} / F_{j}$ has been chosen so that the image does not change on restriction to the smaller Galois group. Moreover, $\overline{r}_{\CM, \r}|_{G_{F_2(\zeta_r)}}$ is irreducible, and again the analogous property holds over $F_6$ by construction. Therefore (13) is satisfied. 
\item Assumptions (14)--(17) hold by construction.
\end{itemize} 
This completes the proof. 
\end{proof}

\section{Applications}

\subsection{The Ramanujan Conjecture}
We are now in a position to prove the (more general verisons of the) main theorems of the introduction
as a consequence of Theorem~\ref{thm_pot_aut_of_powers}.
Let~$F$ be an imaginary CM field, 
and let~$\pi$ be a regular algebraic cuspidal automorphic
representation of~$\GL_2(\A_F)$. We write $(a_\tau \ge b_\tau)_{\tau: F\hookrightarrow \C}$ for the weight of $\pi$. 
Recall that we say that~$\pi$ is \emph{of parallel weight} if $a_{\tau}-b_{\tau}$ is independent of $\tau$. 
\begin{theorem}\label{thm: Ramanujan thm main paper}
Let~$F$ be an imaginary CM field, and let~$\pi$ be a regular algebraic cuspidal automorphic representation
of~$\GL_2(\A_F)$ of parallel weight. Then, for all  primes $v$ of $F$, the representation $\pi_v$ is \emph{(}essentially\emph{)} tempered.
\end{theorem}
\begin{proof}
Since~$\pi$ is assumed to have parallel weight, there is an integer $m \geq 1$ such that~$a_{\tau}-b_{\tau} = m-1$ for all~$\tau:F\hookrightarrow\C$. 
By Clozel's purity lemma \cite[Lemma 4.9]{MR1044819}, there is an integer $w$
with $a_{\tau} + b_{c\tau} = w$ for all $\tau$. It follows that $b_{\tau}+b_{c\tau}=w-m+1$ is independent of~$\tau$. In particular, there exists an algebraic Hecke character of $\A_F^\times$ with weight $(b_{\tau})_{\tau:F\hookrightarrow \C}$, so after twisting we may
        assume that $(a_\tau,b_\tau) = (m-1,0)$ for all ~$\tau$. The
        central character of $\pi$ is then of the form
        $\psi|\cdot|^{1-m}$ for a finite order Hecke character
        $\psi$.

        Exactly as in the proof
        of~\cite[Theorem~7.1.1]{10author}, we can find a quadratic CM
        extension $F'/F$ for which the character~$\psi\circ N_{F'/F}$
        is a square. We can check temperedness after base change to
        $F'$. Twisting by a finite order Hecke character, we may then
        assume that~$\pi$ has central character $|\cdot|^{1-m}$. 
                  Exactly as in the proof of \cite[Cor.~7.1.15]{10author}, we
        can make a further solvable base change to reduce to checking
        temperedness of the unramified $\pi_v$.
        Hence it suffices to show that the associated very weakly compatible system~$\cR$ (cf.~\cite[Lemma 7.1.10]{10author}) is pure.
         By \cite[Lemma~7.1.2]{10author},
        either~$\cR$ is strongly irreducible, Artin up to twist, or induced
        from a quadratic extension.
        If~$\cR$ is induced, then purity follows from the purity of rank one (very weakly) compatible systems. The compatible family~$\cR$ cannot be Artin up to twist because
        that is incompatible with having distinct Hodge--Tate weights.
        Thus~$\cR$ is strongly irreducible,
        and the result follows from Theorem~\ref{thm_pot_aut_of_powers}.
        \end{proof}

\subsection{The potential automorphy of compatible systems and the Sato--Tate conjecture}\label{subsec:satotate}
\begin{thm} \label{PO}
Let~$F$ be a CM field, %
and let~$\RR=(M,S,\{ Q_v(X) \},
r_\lambda ,H_{\tau} )$ be a  very weakly compatible system of rank 2 representations of $G_F$ that is strongly
 irreducible. Suppose there exists an integer $m \geq 1$ such that $H_\tau = \{ 0, m \}$ for each embedding $\tau : F \to \overline{M}$. %
 Then $\cR$ is pure of weight $m$, and for each $n \geq 1$, there exists a finite CM extension $F' / F$, Galois over $\Q$, such 
 that~$\Sym^{n-1} \cR|_{G_{F'}}$ is automorphic. 

 If one alternatively assumes that~$\RR$ is irreducible but not strongly irreducible,
 then~$\cR$ is pure of weight~$m$, and for each $n \geq 1$,
 $\Sym^{n-1} \cR$ decomposes as a direct sum of compatible systems
 of dimension at most~$2$ which are automorphic.
 \end{thm}
\begin{proof} Assume that~$\cR$ is strongly irreducible.
As in the proof of Theorem~\ref{thm: Ramanujan thm main paper}, we can reduce to the case where~$\cR$
has determinant~$\varepsilon^{-m}$. But now Theorem~\ref{PO} follows directly from Theorem~\ref{thm_pot_aut_of_powers}.

If~$\cR$ is not strongly irreducible, then from~\cite[Lemma~7.1.2]{10author} it
follows that~$\cR$ is induced from a compatible system of algebraic
Hecke characters for some quadratic extension~$F'/F$ (the condition on the Hodge--Tate
weights ensures that~$\cR$ is not Artin up to twist).
Then the symmetric powers~$\Sym^{n-1} \cR$
decompose as a sum of  two-dimensional induced compatible systems and
(when~$n$ is odd) a one-dimensional compatible system. %
In particular for any~$n$, $\Sym^{n-1} \cR$ decomposes as a direct sum of %
automorphic compatible systems, %
and the purity statement follows from the purity
of (the Galois representations associated to) algebraic Hecke characters.
\end{proof}
We next give a statement of the Sato--Tate conjecture, including Theorem~\ref{SatoTate} as a special case, before giving the proof when $\pi$ has parallel weight. Let $F$ be an imaginary CM field, and let $\pi$ be a cuspidal automorphic representation of $\GL_2(\A_F)$ which is regular algebraic of weight $\lambda$ and not CM (i.e.\ not automorphically induced). Thus there is an integer $w$ such that $\lambda_{\tau, 1} + \lambda_{\tau c, 2} = w$  for all $\tau \in \Hom(F, \C)$. The central character of $\pi$ has the form $\omega_\pi = | \cdot|^{-w} \psi$, where $\psi : F^\times \backslash \A_F^\times \to \C^\times$ is a unitary Hecke character of type $A_0$. We define the Sato--Tate group of $\pi$, $\mathrm{ST}(\pi)$, as follows:
\begin{itemize}
\item If $\psi$ has finite order $a \geq 1$, then $\mathrm{ST}(\pi) = \mathrm{U}_2(\R)_a := \{ g \in \mathrm{U}_2(\R) \mid \det(g)^a = 1 \}$.
\item If $\psi$ has infinite order, then $\mathrm{ST}(\pi) = \mathrm{U}_2(\R)$.
\end{itemize}
\begin{lemma}
$\mathrm{ST}(\pi)$ is a compact subgroup of $\GL_2(\C)$. If $v$ is a finite place of $F$ such that $\pi_v$ is unramified and essentially tempered, then the $\GL_2(\C)$-conjugacy class of $q_v^{-w/2} \operatorname{rec}_{F_v}(\pi_v)(\Frob_v)$ intersects $\mathrm{ST}(\pi)$ in a unique conjugacy class of $\mathrm{ST}(\pi)$.
\end{lemma}
\begin{proof}
The group $\mathrm{ST}(\pi)$ is compact since $\mathrm{U}_2(\R)$ is. It is well-known that two elements of $\mathrm{U}_2(\R)$ which become conjugate in $\GL_2(\C)$ are conjugate by an element of $\mathrm{SU}_2(\R)$. All we need to show then is that if $v$ is a finite place of $F$ such that $\pi_v$ is unramified, and $ \operatorname{rec}_{F_v}(\pi_v)(\Frob_v) = \diag(\alpha_v, \beta_v)$, then $\alpha_v, \beta_v$ are complex numbers of absolute value $q_v^{w/2}$, and further if $\psi$ has finite order $a$ then $(q_v^{-w} \alpha_v \beta_v)^a = 1$. 

Since $\pi_v$ is essentially tempered, we have $| \alpha_v | = | \beta_v |$. On the other hand, we have $\alpha_v \beta_v = \psi(\varpi_v) q_v^w$, hence $| \alpha_v \beta_v | = q_v^w$ (as $\psi$ is unitary), and if $\psi$ has finite order $a$ then $(q_v^{-w} \alpha_v \beta_v)^a = 1$. 
\end{proof}
If $v$ is a place such that $\pi_v$ is unramified and essentially tempered, then we write $[\pi_v] \in \mathrm{ST}(\pi)$ for a representative of the conjugacy class of $q_v^{-w/2} \operatorname{rec}_{F_v}(\pi_v)(\Frob_v) \in \GL_2(\C)$.
\begin{thm}\label{thm_Sato_Tate_general_case}
Suppose that $\pi$ has parallel weight. Let $S_\pi$ denote the set of finite places of $F$ at which $\pi$ is unramified. With notation as above, the classes of elements $[\pi_v] \in \mathrm{ST}(\pi)$ \textup{(}$v \not\in S_\pi$\textup{)} are equidistributed with respect to the Haar probability measure $\mu_{\mathrm{ST}}$ of $\mathrm{ST}(\pi)$. More precisely, for any continuous, conjugation-invariant function $f : \mathrm{ST}(\pi) \to \C$, we have
\[ \lim_{X \to \infty} \frac{ \sum_{v \not\in S_\pi, q_v < X} f([\pi_v])}{ \#\{v \not\in S_\pi, q_v  < X\}} = \int_{g \in \mathrm{ST}(\pi)} f(g) \, d \mu_{\mathrm{ST}(\pi)}. \]
\end{thm}
\begin{proof}
If $\rho$ is a finite-dimensional irreducible representation of $\mathrm{ST}(\pi)$, let us define
\[ L^{S_\pi}(\pi, \rho, s) = \prod_{v \not\in S_\pi} \det(1 - q_v^{-s} \rho([\pi_v]))^{-1}, \]
an Euler product which converges absolutely in the right half-plane $\operatorname{Re}(s) > 1$. According to the criterion of Serre \cite[Ch. I, Appendix]{serreabladic}, the theorem will be proved if we can show that for each non-trivial such $\rho$, $L^{S_\pi}(\pi, \rho, s)$ admits a meromorphic continuation to $\C$ which is holomorphic and non-vanishing on the line $\operatorname{Re}(s) = 1$. This may be deduced from the potential automorphy of the symmetric powers $\Sym^{n-1} \cR$ of the compatible system associated to $\pi$, exactly as in e.g.\ \cite[\S 7]{GeeSTwt3} and \cite[\S 8]{blght}, after noting that $\cR$ is strongly irreducible (again invoking \cite[Lemma~7.1.2]{10author} and the assumption that $\pi$ is not CM). Note also that the list of non-trivial one-dimensional representations of $\mathrm{ST}(\pi)$ depends on the order of the character $\psi$.
\end{proof}

\emergencystretch=3em

\input{SatoTate.bbl}

\end{document}

%% file: SatoTate.bbl
\newcommand{\etalchar}[1]{$^{#1}$}
\renewcommand{\MR}[1]{}
\providecommand{\bysame}{\leavevmode\hbox to3em{\hrulefill}\thinspace}
\providecommand{\MR}{\relax\ifhmode\unskip\space\fi MR }
\providecommand{\MRhref}[2]{%
  \href{http://www.ams.org/mathscinet-getitem?mr=#1}{#2}
}
\providecommand{\href}[2]{#2}